\documentclass[11pt]{article}
\usepackage[ansinew]{inputenc}
\usepackage{graphicx}
\usepackage{color}

\usepackage{enumerate,latexsym}
\usepackage{latexsym}
\usepackage{amsmath,amssymb}
\usepackage{graphicx}
\usepackage{soul}
\usepackage{times}
\usepackage[dvipsnames]{xcolor}

\newfont{\bb}{msbm10 at 11pt}
\newfont{\bbsmall}{msbm8 at 8pt}
\def\rth{\mathbb{R}^3}
\def\R{\mathbb{R}}

\def\B{\mathbb{B}}
\def\N{\mathbb{N}}

\def\Hip{\mathbb{H}}

\def\D{\mathbb{D}}

\def\esf{\mathbb{S}}

\def\cL{\mathcal{L}}

\renewcommand{\S}{\Sigma}

\newcommand{\ben}{\begin{enumerate}}
\newcommand{\bit}{\begin{itemize}}
\newcommand{\een}{\end{enumerate}}
\newcommand{\eit}{\end{itemize}}
\newcommand{\wh}{\widehat}
\newcommand{\ov}{\overline}
\newcommand{\Int}{\mbox{Int}}

\newcommand{\wt}{\widetilde}

\def\a{{\alpha}}
\def\lc{{\cal L}}
\def\t{{\theta}}

\def\g{{\gamma}}
\def\G{{\Gamma}}
\def\l{{\lambda}}

\def\de{{\delta}}
\def\be{{\beta}}
\def\ve{{\varepsilon}}

\newcommand{\cS}{{\cal S}}
\newcommand{\cC}{{\cal C}}

\def\centerbmp#1#2#3{\vskip#2\relax\centerline{\hbox to#1{\special
    {bmp:#3 x=#1, y=#2}\hfil}}}

\newtheorem{theorem}{Theorem}[section]
\newtheorem{lemma}[theorem]{Lemma}
\newtheorem{proposition}[theorem]{Proposition}

\newtheorem{remark}[theorem]{Remark}
\newtheorem{corollary}[theorem]{Corollary}
\newtheorem{definition}[theorem]{Definition}
\newtheorem{conjecture}[theorem]{Conjecture}
\newtheorem{assertion}[theorem]{Assertion}

\newtheorem{claim}[theorem]{Claim}

\newtheorem{example}[theorem]{Example}

\newenvironment{proof}{\smallskip\noindent{\it Proof.}\hskip \labelsep}
{\hfill\penalty10000\raisebox{-.09em}{$\Box$}\par\medskip}

\begin{document}

\begin{title}
{Structure theorems for singular minimal laminations}
\end{title}
\vskip .5in

\begin{author}
{William H. Meeks III\thanks{This material is based upon work for
the NSF under Award No. DMS - 1309236. Any opinions, findings, and
conclusions or recommendations expressed in this publication are
those of the authors and do not necessarily reflect the views of the
NSF.} \and Joaqu\'\i n P\' erez \and Antonio Ros\thanks{The second and third authors
were supported in part by the MINECO/FEDER grant no. MTM2014-52368-P.}, }
\end{author}

\maketitle
\begin{abstract}
We apply the local removable singularity theorem for
minimal laminations~\cite{mpr10} and the local picture theorem on
the scale of topology~\cite{mpr14} to obtain two descriptive results
for certain possibly {\it singular minimal laminations} of $\R^3$.
These two global structure theorems will be applied
in~\cite{mpr8} to obtain bounds on the index
and the number of ends of complete, embedded minimal surfaces of
fixed genus and finite topology in $\R^3$, and in~\cite{mpr9} to prove that a complete, embedded
minimal surface in $\R^3$ with finite genus and a countable number
of ends is proper.
\vspace{.17cm}


\end{abstract}

\section{Introduction.}
The analysis of singularities of embedded minimal surfaces and more generally of
minimal laminations in three-manifolds is a transcendental
open problem in minimal surface theory. Theory developed by
Colding and Minicozzi~\cite{cm21,cm22,cm24,cm23,cm35,cm25}
and subsequent applications by
Meeks and Rosenberg~\cite{mr8,mr13} and Meeks, P\'erez and Ros~\cite{mpr3,mpr4,mpr6}
 demonstrate the importance of the analysis of singularities
of minimal laminations. Removable singularity theorems in~\cite{mpr10} have been instrumental
in obtaining classification results~\cite{mpr21} for CMC foliations of $\R^3$ and $\esf ^3$ with a countable
 set of singularities, in studying dynamical properties of
the space of properly embedded minimal surfaces in $\rth$~\cite{mpr20},
and in deriving local
pictures on the extrinsic geometry of an embedded minimal surface around points of arbitrarily
small injective radius~\cite{mpr14}.

In this paper we will improve the understanding of singularities of minimal laminations
in $\R^3$ with two new results on the global structure of these objects. In the first result,
Theorem~\ref{tttt}, we describe the possible limits (after extracting a subsequence) of a
sequence of embedded minimal surfaces with {\it locally positive injectivity radius}\footnote{See
Definition~\ref{deflpir} for this notion.} in the complement
of a countable closed set of $\R^3$. The second result, Theorem~\ref{global}, describes the
structure of a singular minimal lamination of $\R^3$ whose singular set is countable.
Both results depend on the local theory of embedded minimal surfaces
and minimal laminations developed in~\cite{mpr14,mpr20,mpr10}, and on the previously
mentioned work of Colding and Minicozzi.
For the definition and the general theory of
minimal laminations, see for instance~\cite{mpe2,mpr19,mpr18,mpr10,mr8,mr13}.

We next give a
formal definition of a singular lamination and of the set of
singularities associated to a leaf of a singular lamination.  Given
an open set $A\subset \R^3$ and a subset $B\subset A$, we will denote by
$\overline{B}^A$ the closure of $B$ in $A$. In the case $A=\rth$, we
simply denote $\overline{B}^{\rth}$ by $\overline{B}$.

\begin{definition}
\label{defsinglamin}
{\rm A {\it singular lamination} of an open set
$A\subset\R^3$ with {\it singular set} ${\cal S}\subset A$ is the
closure $\overline{\lc}^A$ of a lamination $\lc$ of $A - {\cal S}$,
such that for each point $p \in {\cal S}$, then  $p \in
\overline{\lc}^A$, and in every open neighborhood $U_p\subset A$ of
$p$, $\,\overline{\cal L}^A \cap U_p$ fails to have an induced
lamination structure in $U_p$. It then follows that ${\cal S}$ is
closed in $A$. The singular lamination $\overline{\cal L}^A$ is said
to be {\it minimal} if the leaves of the related lamination ${\cal
L}$ of $A-{\cal S}$ are minimal surfaces.

For a leaf $L$ of $\lc$, we call a point $p \in \overline{L}^A \cap
{\cal S}$ a {\em singular leaf point of $L$} if there exists an open set $V
\subset A$ containing $p$ such that $L \cap V$ is closed in $V - {\cal
S}$. We let ${\cal S}_L$ denote the {\em set of singular leaf points
of L.} Finally, we define
\begin{equation}
\label{eq:leaf}
\overline{\lc}^A (L) = L \cup {\cal S}_L
\end{equation}
to be the {\em leaf of $\overline{\lc}^A$ associated to the leaf $L$
of $\lc$}. In the case $A=\R^3$, we simply denote $\overline{\cal L}^A(L)$
by $\overline{\cal L}(L)$.
In particular, the leaves of $\overline{\lc}^A$ are of one of the following two types.
\begin{itemize}
\item If for a given leaf $L$ in ${\cal L}$ we have $\overline{L}^A\cap {\cal
S}=\mbox{\rm \O}$, then $L$ a leaf of $\overline{\cal L}^A$.

\item If for a given leaf $L$ in ${\cal L}$  we have $\overline{L}^A\cap {\cal
S}\neq \mbox{\rm \O}$, then  $\overline{\lc}^A(L)$ is a leaf of
$\overline{\lc}^A$.
\end{itemize}
Note that since ${\cal L}$ is a lamination of $A-{\cal S}$, then
$\overline{\cal L}^A= {\cal L}\cup 
{\cal S}$. Hence, the closure $\overline{\cal L}$ of ${\cal L}$
when considered to be a subset
of $\R^3$ is the set $\overline{\cal L}={\cal L}\cup
{\cal S}\cup
(\partial A \cap \overline{\cal L})$.
}
\end{definition}

In contrast to the behavior of (regular) laminations, it is possible
for distinct leaves of a singular lamination 
to intersect. In Section~\ref{sec12} we will give an example that illustrates this phenomenon.

\begin{definition}
{\rm
With the notation in Definition~\ref{defsinglamin},
a leaf $\overline{\cal L}{^A}(L)=
L\cup {\cal S}_L$
of $\overline{\cal L}^A$ is said to be a {\it limit leaf}
of $\overline{\cal L}^A$ if the related leaf $L\in {\cal L}$ is a limit leaf of ${\cal L}$
(i.e., there exists a point $p\in L$ that is a limit in $A$
of a sequence of points $p_n\in L_n$, where $L_n$ is a leaf of
$^{}{\cal L}$ for all $n$, and if $L_n=L$ after passing to
a subsequence, then the sequence
$p_n$ does not converge to $p$ in the intrinsic topology of $L$).
We will denote by Lim$(\overline{\cal L}^A)$ the set of limit leaves of
$\overline{\cal L}^A$.
}
\end{definition}

Throughout the paper, $\B (p,R)$ will denote the open Euclidean ball
of radius $R>0$ centered at a point $p\in \R^3$, $\B (R)=\B
(\vec{0},R)$, $\esf^2(p,R)=\partial\B(p,R)$ and
$\esf^2(R)=\esf^2(\vec{0}, R)$. For a surface $\Sigma \subset \R^3$,
$K_{\Sigma }$ will denote its Gaussian curvature function.

\begin{definition}
\label{deflpir}
{\rm
Let  $\{M_n\}_n$ be a sequence of surfaces (possibly with boundary)
in an open set $A \subset \R^3$. We will say that $\{M_n\}_n$ has
{\em locally positive injectivity radius in $A$,} if for every
$q\in A$, there exists $\ve _q>0$ and $n_q\in \N $ such that for
$n>n_q$, the restricted functions $(I_{M{_n}})|_{\B (q,\ve_q) \cap M_n}$ are
uniformly bounded away from zero, where $I_{M_{n}}$ is the injectivity
radius function of $M_n$.}
\end{definition}

Note that if the surfaces $M_n$ have boundary and $\{ M_n\} _n$ has locally
positive injectivity radius in $A$, then for any $p\in A$
there exists $\varepsilon _p>0$ and $n_p\in \N$ such that $\partial
M_n\cap \B (p,\ve _p)=\mbox{\rm \O}$ for  $n>n_p$, i.e., points in the boundary
of $M_n$ must eventually diverge in space or accumulate to points in the complement of $A$.

By Proposition 1.1 in Colding and Minicozzi~\cite{cm35}, the
property that a sequence of embedded minimal surfaces $\{M_n\}_n$
has locally positive
injectivity radius in an open set $A$ is equivalent to the property
that $\{M_n\}_n$ is {\it locally simply connected in $A$,} in the sense
that around any point $q\in A$, we can find $\de _q>0$ such that
$\B (q,\de_q)\subset A$ and for $n$ sufficiently large,  $\B (q,\de
_q)$ intersects $M_n$ in components that are disks with boundaries
in $\esf^2(q,\de _q)$.

\begin{theorem}
\label{tttt}
Suppose $W$ is a countable closed  subset of $\R^3$ and
$\{M_n\}_n$ is a sequence of embedded minimal surfaces (possibly
with boundary) in $A=\R^3-W$, that has locally positive injectivity
radius in $A$.  Then, there exist a closed subset ${\cal S}^A\subset
A$,  a minimal lamination ${\cal L}$ of $A-{\cal S}^A$ and a subset
$S({\cal L})\subset {\cal L}$ (in particular, $S({\cal
L})\cap {\cal S}^A =\O)$ such that after replacing by a subsequence,
$\{M_n\}_n$ converges $C^\a$, for all $\a \in (0,1)$, on compact subsets of $A-(S({\cal L})\cup
{\cal S}^A)$ to ${\cal L}$ (here $S(\cal L)$ is the singular
set of convergence\footnote{$S({\cal L})$ is the set of points $x\in {\cal L}$ such
that $\sup _{n\in \N}|K_{M_n\cap \B(x,\ve )}|$ is not bounded for any $\ve >0$.}
of $\{ M_n\} _n$ to ${\cal L}$), and the closure of
${\cal L}$ in $A$ has the structure of a possibly singular minimal
lamination of $A$ with singular set ${\cal S}^A$:
\[
\overline{\lc}^A={\cal L}\cup
{\cal S}^A.
\]
 Furthermore, the closure
$\overline{\lc}$ in $\R^3$ of $\lc$ has the structure
of a possibly singular minimal lamination of $\R^3$, with the
singular set ${\cal S}$ of $\overline{\lc}$ satisfying
${\cal S}^A\subset {\cal S}\subset {\cal S}^A\cup
(W\cap \overline{\cal L})$, and:
\begin{description}
\item[{\it 1.}]
\label{12.2.1}
The set ${\cal P}$ of planar leaves in
$\overline{\lc}$ forms a closed subset of $\R^3$.
\item[{\it 2.}]
\label{12.2.2} The set $\mbox{\rm Lim}(\overline{\cal L})$ of limit
leaves of $\overline{\lc}$ forms a closed set in $\R^3$ and satisfies
$\mbox{\rm Lim}(\overline{\cal L}) \subset {\cal P}$. Furthermore, if
$L=\overline{\cal L}(L_1)=L_1\cup {\cal S}_{L_1}$ is a leaf of $\overline{\cal L}$
(here $L_1$ is the related leaf of the regular lamination associated to $\ov{\cL}$,
see (\ref{eq:leaf})) and $A\cap {\cal S}_{L_1}\neq \mbox{\rm \O}$, then $L$
is a limit leaf of $\overline{\cal L}$. In particular, every singular leaf point
of a non-flat leaf of $\overline{\cal L}$ belongs to $W$.
\item[{\it 3.}]
\label{12.2.5} If $P$ is a plane in ${\cal P} - \mbox{\rm
Lim}(\overline{\cal L})$, then there exists $\delta >0$ such that
$P({\delta}) \cap \overline{\lc}=P$, where $P(\de )$ is the
$\delta$-neighborhood of $P$. In particular, ${\cal S}\cap [{ \cal
P}-\mbox{\rm Lim}(\overline{\cal L})]=\O$.
\item[{\it 4.}]
\label{12.2.3}
For each point $q\in {\cal S}^A\cup S(\lc)$, there passes
a plane $P_q\in \mbox{\rm Lim}(\overline{\cal L})$. Furthermore, $P_q$
intersects ${\cal S}^A\cup S(\lc) \cup W$ in a closed countable set.
\item[{\it 5.}]
\label{5tttt} Through each point of $p\in W\cap \overline{\cal L}$ satisfying one of the
conditions 5.1, 5.2 below, there passes a planar leaf $P_p$ in ${\cal P}$.
\begin{enumerate}[5.1.]
\item
\label{12.2.5A}
For all $k\in \N$, there exists $\ve_k\in (0,\frac1k)$ and an open subset $\Omega _k$ of $\B (p,\ve_k)$
such that $W\cap \B (p,\ve_k)\subset \Omega _k\subset \ov{\Omega}_k\subset \B (p,\ve_k)$
and the area of
$M_n\cap [\B (p,\ve_k)-\overline{\Omega }_k]$
 diverges to infinity as $n\to \infty $ (in this case, 
 the convergence of the $M_n$ to $P_p$ has infinite multiplicity).
\item
The convergence of the $M_n$ to some leaf of $\lc$
having $p$ in its closure is of finite multiplicity greater than one.
\end{enumerate}

\item[{\it 6.}]
\label{12.2.6} Suppose that there exists a leaf $L=L_1\cup{\cal S}_{L_1}$ of $
\overline{\lc}$ that is not contained in ${\cal P}$,
where $L_1$ is the related leaf of the regular lamination
${\cal L}_1:=\overline{\cal L}-{\cal S}$ of $\R^3-{\cal S}$ and
${\cal S}_{L_1}$ is the set of singular leaf points of $L_1$.
Then, $L\cap({\cal S}^A \cup S(\lc)) = \mbox{\rm \O}$ (note that $L$ might
contain singular points which necessarily belong to $W$), the convergence of
portions of the $M_n$ to $L_1$ is of multiplicity one, and one of the
following two possibilities holds:
\begin{enumerate}[6.1.]
\item
$L$ is  proper\footnote{As leaves of $\overline{\cal L}$
may have singularities, properness of such a leaf $L=L_1\cup {\cal S}_{L_1}$
just means that $L$ is a closed set of $\R^3$, or equivalently, ${\cal S}_{L_1}$ is closed in $\R^3$ and
$L_1$ is a proper surface in the complement of ${\cal S}_{L_1}$.} in $\R^3$, ${\cal S}={\cal S}_{L_1}\subset W$
and $L$ is the unique leaf of $\overline{\cal L}$.
\item
$L$ is not proper in $\R^3$ and ${\cal P}\neq \mbox{\rm \O}$.
In this case, $\overline{L}$
has the structure of a possibly singular minimal lamination of $\R^3$ with a countable
set of singularities, there exists a subcollection ${\cal P}(L)\subset
{\cal P}$ consisting of one or two planes such that $\overline{L}=L
\cup {\cal P}(L)$, $L$ is proper in a component $C(L)$ of $\R^3-{\cal P}(L)$ and
$C(L)\cap \overline{\cal L}=L$.
Furthermore:\ben[a.]
\item Every open $\ve$-neighborhood $P(\ve)$ of a plane $P\in {\cal
P}(L)$ intersects $L_1$ in a connected surface with unbounded Gaussian curvature.
\item   If some  $\ve$-neighborhood $P(\ve)$ of a plane $P\in {\cal
P}(L)$ intersects $L_1$ in a surface with finite genus, then  $P(\ve)$  is
disjoint from the singular set of $\ov{L}$.
\item  $L_1$ has infinite genus.
\een
\end{enumerate}
In particular, $\overline{\lc}$ is the disjoint union of its leaves,
regardless of which case 6.1 or 6.2 occurs (if case 6.2 occurs,
then each leaf of $\overline{\lc}$ is either a plane or a minimal
surface possibly with singularities in $W$, that is proper\footnotemark[2]
in an open halfspace or slab of $\R^3$).

\item[{\it 7.}]
\label{12.2.7}
Suppose that the surfaces $M_n$ have uniformly
bounded genus and ${\cal S} \cup S({\lc}) \neq \mbox{\rm \O}$. Then:
\begin{enumerate}[7.1.]
\item
\label{12.2.7A}
$\overline{\lc}={\cal P}$ and so, ${\cal S}=\mbox{\rm \O}$.
\item
\label{12.2.7B}
$\overline{\lc}$ contains a foliation ${\cal F}$ of an
open slab of $\R^3$ by planes and $S(\lc)\cap {\cal F}$
consists of one or two straight line segments orthogonal to the
planes in ${\cal F}$, where each line segment intersects every plane in
${\cal F}$. Furthermore, if there are 2 different line segments in
$S(\lc)\cap {\cal F}$, then in the related limit minimal parking garage structure
of the slab, the two multivalued graphs occurring inside the surfaces
$M_n$ along $S(\lc)\cap {\cal F}$ are oppositely handed.
\item
\label{12.2.7C}
If the $M_n$ are compact with
boundary, then $\overline{\lc}$ is a foliation of $\R^3$ by planes
and $\overline{S({\lc})}$ consists of one or two complete lines orthogonal to the planes in this foliation.
\end{enumerate}
\end{description}
\end{theorem}

In item~7.2 of Theorem~\ref{tttt} we mentioned the ``related
limit minimal parking garage structure of the slab''; we refer the reader
to our paper~\cite{mpr14} for
the notion of limit minimal parking garage structure of $\rth$
(see Colding and Minicozzi~\cite{cm25} for a related discussion).
Limit minimal parking garage structures in~\cite{mpr14} are foliations of $\R^3$ by planes,
that appear as the limit outside a discrete set of lines orthogonal to the planes,
of certain sequences of embedded minimal surfaces that are uniformly
locally simply connected in $\R^3$.
The fact that the sequence $\{ M_n\} _n$
in Theorem~\ref{tttt} is only locally simply connected outside $W$ is what might
produce a foliated slab rather than the whole $\R^3$.
In spite of this problem that arises from
$W$, we feel that our language here appropriately describes the
behavior of the limit configuration, since if ${\cal F}$ is a union of planar leaves of
$\overline{\lc}$ that forms an open slab and ${\cal F}\cap S({\cal
L})\neq \mbox{\rm \O}$, then ${\cal F} \cap {\cal S} = \mbox{\O}$ and
for $n$ large, $M_n \cap {\cal F}$ has the appearance of a parking
garage surface on large compact domains of this open slab, away
from $W$.  In Example~\ref{expg} below we will exhibit
a parking garage structure of the upper halfspace of $\R^3$.

Regarding applications, Theorem~\ref{tttt} will be crucial in the proof of the following results:
\begin{enumerate}[(I)]
\item In Theorem~\ref{global} below we will
describe the structure of any singular minimal lamination $\overline{\lc}={\cal L}\cup
{\cal S}$ of $\R^3$ with countable singular set ${\cal S}$.
Roughly  speaking, either $\overline{\lc}$ consists of a single leaf which is
a properly embedded minimal surface (${\cal S}=\mbox{\O }$ in this case), or
 $\overline{\lc}$ consists of a closed family ${\cal P}$ of parallel planes
that contains all limit leaves of $\overline{\lc}$, together with non-flat
leaves in $\R^3-{\cal P}$, each of which has infinite genus,
unbounded Gaussian curvature and is properly embedded in an open slab or halfspace bounded by
one of two planes in ${\cal P}$, and limits to these planes (in particular, non-flat leaves are
not proper in $\R^3$).

\item
In~\cite{mpr8} we will apply
Theorem~\ref{tttt} to prove that for each $g\in \N \cup \{ 0\} $,
there exists a bound on the number of ends
of a complete, embedded minimal surface in $\R^3$ with finite
topology and genus at most $g$. This topological boundedness result
implies that the stability index of a complete, embedded minimal
surface of finite index in $\R^3$ has an upper bound that depends only on its
finite genus.

\item We will use Theorem~\ref{tttt} in~\cite{mpr9}
to show that a connected, complete, embedded minimal surface in
$\R^3$ with an infinite number of ends, finite genus and compact
(possibly empty) boundary, is proper if and
only if it has a countable number of limit ends, if and only if it
has one or two limit ends, and when the boundary of the proper surface is
empty, then it has exactly two limit ends (limit ends are the limit
points in the space of ends endowed with its natural topology).
Both (II) and (III) were announced a long time ago, but we found some 
problems in the original proof that have finally been resolved by applications
of results in the present paper. 
\end{enumerate}

Besides the above applications, Theorems~\ref{tttt} and~\ref{global}
provide geometrical insight for possibly resolving the
following fundamental conjecture, at least when the set ${\cal S}$ is countable.

\begin{conjecture}[Fundamental Singularity
   Conjecture]
\label{conjlamination} \mbox{}\newline Suppose ${\cal S} \subset
\R^3$ is a closed set whose one-dimensional Hausdorff measure is
zero. If $\lc$ is a minimal lamination of $\R^3 - {\cal S}$, then
the closure $\overline{\cal L}$ has the structure of a minimal
lamination of~$\R^3$.
\end{conjecture}

Since the union of a catenoid with a plane passing through its waist
circle is a singular minimal lamination of $\R^3$ whose singular set
is the intersecting circle, the above conjecture represents the
strongest possible removable singularity conjecture.

The paper is organized as follows. In Section~\ref{sec12} we
give examples of singular minimal laminations and
obtain some results to be used in the proof of the main
Theorem~\ref{tttt}; these auxiliary results are based on the
local removable singularity theorem~\cite{mpr10} and the stable limit leaf
theorem for the limit leaves of a minimal lamination~\cite{mpr19,mpr18}.
In Section~\ref{sec4} we prove Theorem~\ref{tttt}.
Section~\ref{secproofthm1.3} contains the statement
and proof of the application (I) (Theorem~\ref{global}) listed above.
In the final Section~\ref{sec7} we will describe the subsequential limit of a sequence $\{ M_n\} _n$ of compact
embedded minimal surfaces of genus at most $g\in \N \cup \{ 0\} $, with boundaries $\partial M_n$
diverging in $\R^3$, provided that the $M_n$ contain disks that converge
$C^2$ to a nonflat minimal disk: a subsequence converges
smoothly on compact subsets of $\R^3$ with multiplicity one to a connected
nonflat minimal surface of genus at most $g$ which is properly embedded
in $\R^3$, has bounded Gaussian curvature, and either it has finite total
curvature, or is a helicoid with handles or a two-limit-ended surface.

\section{Preliminaries.}
\label{sec12}
We start by giving an example of a singular minimal lamination whose leaves intersect.
We will use the notation introduced in Definition~\ref{defsinglamin}.
\begin{example}
\label{example2.1}
{\rm
The union of two transversal planes $\Pi _1,\Pi _2\subset \R^3$
is a singular lamination $\overline{\lc}$ of $A=\R^3$ with
singular set ${\cal S}$ being the line $\Pi _1\cap \Pi _2$.
In this example, Definition~\ref{defsinglamin} yields a related
lamination $\lc$ of $\R^3-{\cal S}$ with four leaves that are open
halfplanes in $\Pi _i-(\Pi _1\cap \Pi _2)$, $i=1,2$,
\and $\overline{\lc}$ has four leaves that are the
associated closed halfplanes that intersect along ${\cal S}$; thus,
$\overline{\lc}$ {\em is not} the disjoint union of its leaves:
every point in ${\cal S}$ is a singular leaf point of each of the
four leaves of ${\cal L}$.
}
\end{example}

In our second example, the leaves of the singular minimal lamination will not
intersect.
\begin{example}
\label{example1.3}
{\rm
Colding and Minicozzi~\cite{cm28} constructed a singular minimal lamination
$\overline{\lc}_1$ of the open unit ball $A=\B (1)\subset \R^3$ with
singular set ${\cal S}_1$ being the origin $\{ \vec{0}\} $; the
related (regular) lamination ${\cal L}_1$ of $\B (1)-\{ \vec{0}\} $
consists of three leaves, which are the punctured unit disk $\D -\{
\vec{0}\} =\{ (x_1,x_2,0)\ | \ 0<x_1^2+x_2^2<1\} $ and two
nonproper disks $L^+\subset \{ x_3>0\} $ and $L^-\subset
\{ x_3<0\} $ that spiral to
$\D -\{ \vec{0}\} $ from opposite sides, see Figure~\ref{fig1}.
In this case, $\vec{0}$ is a singular leaf point of
$\D -\{ \vec{0}\} $ (hence $\overline{{\cal L}_1}^A(\D -\{ \vec{0}\}
)$ equals the unit disk $\D $), but $\vec{0}$ is not a singular leaf
point of either $L^+$ or $L^-$ because $L^+\cap V$ fails to be
closed in $V-{\cal S}_1$ for any open set $V\subset \B (1)$
containing $\vec{0}$. Thus, $\overline{{\cal L}_1}^A(L^+)=L^+$ and
analogously $\overline{{\cal L}_1}^A(L^-)=L^-$. Hence,
$\overline{\lc}_1$ {\em is} the disjoint union of its leaves in this
case.
}
\end{example}
\begin{figure}
\begin{center}
\includegraphics[width=8.3cm]{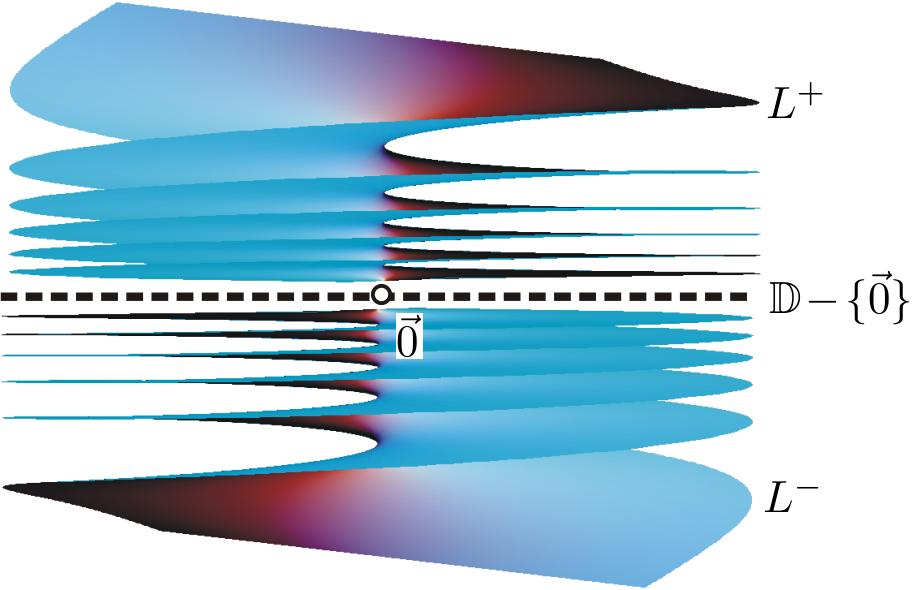}
\caption{The origin is a singular leaf point of the horizontal disk
passing through it, but not of the two nonproper spiraling leaves
$L^+,L^-$.}
 \label{fig1}
\end{center}
\end{figure}

Example~\ref{example1.3} is produced as a limit of a sequence of embedded minimal
surfaces. We will use this sequence to produce an example of a
parking garage structure of the upper halfspace of $\R^3$,
as announced right after the statement of Theorem~\ref{tttt}.

\begin{example}
\label{expg}
{\rm
Colding and Minicozzi~\cite{cm28} proved the existence of a sequence
$\{ D_n\} _n$ of compact minimal disks contained in the closed unit ball
$\overline{\B }(1)$ of $\R^3$, with $\partial D_n\subset \esf^2(1)$,
such that $\{ \Int (D_n)\} _n$ converges as $n\to \infty $ to the singular minimal
lamination $\overline{{\cal L}_1}=\lc_1\cup{\cal S}_1$ of ${\B }(1)$
that appears in Example~\ref{example1.3}.
By the local removable singularity theorem for minimal
laminations (see Theorem~1.1 in~\cite{mpr10} or see
Theorem~\ref{tt2} below), the Gaussian curvature function $K_{{\cal L}_1}$ of
${\cal L}_1$ satisfies that $|K_{{\cal L}_1}|R^2$ is unbounded in
arbitrarily small neighborhoods of $\vec{0}$, where
$R=\sqrt{x_1^2+x_2^2+x_3^2}$. Defining $\l _n=|p_n|^{-1/2}$,
where $p_n$ are points in $D_n$ such that $p_n\to \vec{0}$ and
$|K_{D_n}|(p_n)|p_n|^2\to \infty $, then the Gaussian curvature of the
homothetically expanded
disks $\l _nD_n$ blows up around $\vec{0}$. By Theorem~0.1 in Colding
and Minicozzi~\cite{cm23}, after passing to a subsequence, the $\l
_nD_n$ converge to a foliation ${\cal F}$ of $\R^3$ by planes (which
can be proved to be horizontal due to the properties of $D_n$
in~\cite{cm28}), with singular set of convergence $S({\cal F})$
being a transverse Lipschitz curve to the planes in the foliation,
which in fact is the $x_3$-axis by the $C^{1,1}$-regularity theorem of
Meeks~\cite{me25}). It then follows that $M_n=\l _nL^+\subset \{
x_3>0\} $ is a nonproper, embedded minimal disk, and the sequence
$\{ M_n\} _n$ has locally positive injectivity radius in $\R^3-\{
\vec{0}\} $ and converges in $\R^3$ minus the nonnegative $x_3$-axis
to the minimal lamination ${\cal L}$ of
$\R^3-\{ \vec{0}\}$ by all horizontal planes with positive
heights together with the $(x_1,x_2)$-plane
punctured at $ \vec{0}$, whose singular set of convergence $S({\cal L})$ is the
positive $x_3$-axis. In this case, $W=\{ \vec{0}\} $ and
$\overline{\cal L}$ is the foliation of the closed upper halfspace
of $\R^3$ by horizontal planes; in particular ${\cal S}=\mbox{\rm \O}$.
}
\end{example}

Conjecture~\ref{conjlamination} stated in the Introduction has a
global nature, because there exist interesting minimal laminations
of the open unit ball in $\R^3$ punctured at the origin that do not
extend across the origin, see Figure~\ref{fig1} and also see
Examples I and II in Section~2 in~\cite{mpr10}. In Example III of
Section~2 of~\cite{mpr10} we described a rotationally invariant
global minimal lamination of hyperbolic three-space $\Hip^3$, which
has a similar unique isolated singularity. The existence of this
global singular minimal lamination of $\Hip^3$ demonstrates that the
validity of Conjecture~\ref{conjlamination} must depend on the metric
properties of $\R^3$. However, in \cite{mpr21} and \cite{mpr10}, we obtained a
remarkable local removable singularity result, valid in any
Riemannian three-manifold~$N$ for certain possibly singular
laminations all whose leaves have the same constant mean curvature.
Since we will apply this theorem and a related corollary repeatedly
in the minimal case, we give their complete statements below in this
minimal case.

Given a three-manifold $N$ and a point $p\in N$, we will denote by
$B_N(p,r)$ the metric ball of center $p$ and radius $r>0$.

\begin{theorem} [Local Removable Singularity Theorem~\cite{mpr10}]
\label{tt2}
A minimal lamination $\lc$ of a punctured ball $B_N(p,r)-\{ p\} $
in a Riemannian three-manifold $N$ extends to a minimal lamination
of $B_N(p,r)$ if and only if there exists a positive constant $c$
such that $|\sigma _{\lc} |\, d < c$ in some subball centered at $p$,
where $| \sigma _{\lc}|$ is the norm of the second fundamental form of
the leaves of $\lc$ and $d$ is the distance function in $N$ to $p$.
\end{theorem}

The following result is a consequence of Theorem~\ref{tt2}; see
Corollary~7.1 in~\cite{mpr10} for a proof.

\begin{corollary}
\label{corrs}
Suppose that $N$ is a (not necessarily complete) Riemannian
three-manifold. If $W\subset N$ is a closed countable subset and
${\cal L}$ is a minimal lamination of $N-W$, then the closure of any
collection of its stable leaves extends across $W$ to a minimal
lamination of $N$ consisting of stable minimal surfaces. In
particular,
 \ben
\item The closure $\overline{\mbox{\rm Stab}(\lc)}$ in $N$ of the collection of stable
leaves of $\lc$ is a minimal lamination of $N$ whose leaves are
stable minimal surfaces.
\item  The closure $\overline{\mbox{\rm Lim}(\lc)}$ in $N$ of the sublamination
$\mbox{\rm Lim}(\lc)$ of limit leaves of $\lc$ is a sublamination
 of $ \overline{\mbox{\rm Stab}(\lc)}$.
\item If ${\cal L}$ is a minimal foliation of $N-W$, then ${\cal L}$
extends across $W$ to a minimal foliation of~$N$.
\een
\end{corollary}

Theorems~\ref{tttt}
and~\ref{global} deal with the structure of
certain possibly singular minimal laminations of
$\R^3$.
In both results, the singular laminations can be expressed as a disjoint
union of its possibly singular  minimal leaves (see the last
statement of item~6 of Theorem~\ref{tttt} and of item~5 of Theorem~\ref{global}).
The key result for proving this disjointness
property of leaves is the next proposition; it gives a condition
under which two different leaves of a singular minimal lamination
cannot share a singular leaf point.

\begin{proposition}
\label{propos1} Let $\overline{\cal L}^A$ be a singular minimal lamination
of an open set $A\subset \R^3$, with countable singular set ${\cal S}$ and
related (regular) lamination ${\cal L}$ of $A-{\cal S}$.
Then, any singular point is a singular leaf point
of at most one leaf of $\overline{\cal L}^A$.
\end{proposition}
\begin{proof}
Reasoning by contradiction, suppose that $p\in {\cal S}$ is a singular
leaf point of two different leaves $\overline{\cal L}^A(L_1)$, $\overline{\cal L}^A(L_2)$ of
$\overline{\cal L}^A$, associated to leaves $L_1,L_2$ of ${\cal L}$
(with the notation in (\ref{eq:leaf})).
By definition, $p\in \overline{L_1}^A\cap \overline{L_2}^A$ and
there exists a ball $\B (p,2\ve )\subset A$  such that $L_i\cap \B
(p,2\ve )$ is closed in $\B (p,2\ve )-{\cal S}$, $i=1,2$. Since ${\cal S}$
is countable, we may assume that the sphere $\esf^2(p,\ve )$
is disjoint from~${\cal S}$ and intersects transversely
$L_1\cup L_2$. Next define $L_i(\ve )=L_i\cap \overline{\B
}(p,\ve )$, $i=1,2$. Then $L_1(\ve ),L_2(\ve )$ are disjoint,
properly embedded minimal surfaces in $\overline{\B }(p,\ve )-{\cal S}$.
We will obtain a contradiction after replacing each $L_i(\ve )$
by a component of it having $p$ in its closure (we will
use the same notation $L_i(\ve )$ for this component); hence we
will assume from now on that $L_i(\ve )$ is connected, $i=1,2$.
Since ${\cal S}$ is countable and closed, $\overline{\B}(p,\ve )-
{\cal S}$ is a simply connected three-manifold with boundary.
Hence, $\overline{L_1(\ve )}$
separates $\overline{\B }(p,\ve )$ into two
connected components, and the same
holds for $\overline{L_2(\ve )}$. Let $N$ be the closure of the
component of $\overline{\B }(p,\ve )-(\overline{L_1(\ve )}\cup
\overline{L_2(\ve )})$ that contains both $\overline{L_1(\ve )},
\overline{L_1(\ve )}$ in its boundary.

Using $\partial N$ as a barrier for solving Plateau problems in $N$,
then from a compact exhaustion of $\overline{L_1(\ve
)}-{\cal S}$, we produce a properly  embedded, area-minimizing varifold
$\Sigma _1\subset N-{\cal S}$ with $\partial \Sigma _1=\partial
L_1(\ve )$ and  $p\in \overline{\Sigma}_1$ and that separates
$\overline{L_1(\ve )}$ from
$\overline{L_2(\ve )}$
(see Meeks and
Yau~\cite{my2} for similar construction and a description of
this barrier type construction). By regularity properties of
area-minimizing varifolds, $\overline{\Sigma}_1$ is regular except
possibly at points in $\overline{\Sigma}_1\cap{\cal S} $.  Now
consider $\overline{\Sigma}_1-{\cal S}$ to lie in $\overline{\B
}(p,\ve )-{\cal S}$ and so, $\overline{\Sigma}_1-{\cal S}$
represents a minimal lamination of $\B(p,\ve)-{\cal S}$ with stable
leaves. Since ${\cal S}$ is closed and countable, Corollary~\ref{corrs}
implies that $\Sigma _1$ extends smoothly
across ${\cal S}\cap \B (p,\ve )$. Exchanging $\overline{L_1(\ve )}$
by $\overline{\Sigma} _1$ and reasoning analogously, we find an
embedded, area-minimizing surface $\Sigma _2$ between
$\overline{\Sigma} _1$ and $\overline{L_2(\ve )}$, with $\partial
\Sigma _2=\partial L_2(\ve )$, such that $p\in\overline{\Sigma} _2$
and $\overline{\Sigma}_2 $ is smooth. Clearly $\overline{\Sigma}
_1,\overline{\Sigma}_2 $ contradict the interior maximum principle
at $p$, which  proves the proposition.
\end{proof}
Using similar arguments, we can extend Proposition~\ref{propos1} to
the case of a general Riemannian three-manifold (for the proof to
work and using the same notation as above, we also need the part of
the boundary of $N$ coming from the boundary of $\B (p,\ve )$ to have
positive mean curvature, which can be assumed by choosing $\ve $ small enough). The
following result is a consequence of this
generalization.
\begin{corollary}
 \label{corol2.3}
 Let $\overline{B}_N(p,R)$ be a compact ball
centered at a point $p$ in a Riemannian three-manifold $N$, with
radius $R>0$. Suppose $M_1,M_2\subset \overline{B}_N(p,R)-\{ p\} $
are two disjoint, properly embedded minimal surfaces with boundaries
$\partial M_i\subset \partial B_N(p,R)$, $i=1,2$. Then, at most
one of these surfaces is noncompact, and in this case the noncompact component has just
one end. Furthermore, if $M$ is a  properly embedded, smooth minimal
surface of finite genus in $\overline{B}_N(p,R)-\{ p\} $ with
$\partial M\subset \partial B_N(p,R)$, then its closure
$\overline{M}$ is a smooth, compact, embedded minimal surface in $\overline{B}_N(p,R)$.
\end{corollary}
\begin{proof}
Let $M_1,M_2$ be surfaces as in the statement of the corollary. If
both surfaces are noncompact, then $\overline{M_1}\cup
\overline{M_2}$ forms a singular minimal lamination ${\cal L}$ of
$\overline{B}_N(p,R)$ with $p$ as the only singular point of ${\cal
L}$. By definition, $p$ is a singular leaf point of both $M_1$ and
$M_2$, which contradicts Proposition~\ref{propos1} in the general
Riemannian manifold setting (note that it suffices to find a contradiction
in a sufficiently small ball centered at $p$, hence we can assume
convexity for the boundary of this smaller ball).
By applying the same argument in a
smaller ball centered at $p$, we deduce directly that if $M_1$ is
not compact, then $M_1$ has only one end.

Finally, consider a properly embedded, smooth minimal surface
$M$ of finite genus in $\overline{B}_N(p,R)-\{ p\} $
with $\partial M\subset \partial B_N(p,R)$, and suppose that
$M$ is noncompact. We may assume, by passing to a
smaller $R>0$ and using the arguments in the previous
paragraph, that $M$ has just one end and that
$M$ is an annulus. We can also assume that
the exponential map $\exp _p$ yields $\R^3$-coordinates on
$B_N(p,r)$ centered at $p \equiv \vec{0}$, for $r>0$
small enough. Since $M$ is a locally rectifiable 2-dimensional
varifold with bounded (actually zero) mean curvature,
Theorem~3.1 in Harvey and Lawson~\cite{hl1} implies that
$M$ has finite area. Under this finiteness condition, Allard
proved (\cite{al1}, Section 6.5) the existence of minimal
limit tangent cones of $M$ in $\R^3$ at the origin after
homothetic rescaling of coordinates. By Corollary~5.1(3)
of~\cite{al1}, $M$ satisfies a monotonicity formula for the
extrinsic area (even in this Riemannian setting,
see Remark~4.4 in~\cite{al1}), valid for surfaces with
bounded mean curvature. In the present setting that $M$
has mean curvature zero, Allard's monotonicity formula
implies that with respect to the metric $g_M$ on $M$ induced by
the ambient metric $g$ on $N$, $(M,g_M)$ has at most quadratic extrinsic area growth,
in the sense that
\[
r\in (0,R)\mapsto r^{-2}\mbox{Area}\left(
M\cap \overline{B}_N(p,r),g_M\right)\quad \mbox{is bounded.}
\]
Now consider the ambient conformal change of metric
$g_1=\frac{1}{d^2}g$, where $d$ denotes the distance function in $(N,g)$ to $p$.
Then, $(M,g_1|_M)\subset \left( \overline{B}_N(p,R)-\{ p\} ,g_1
\right) $ is a complete annulus with linear area growth and compact boundary.
Such a surface is conformally a
punctured disk $\D ^*$ (see  Grigor'yan~\cite{gri1}). Thus, the
related conformal harmonic map of  $\D ^*$ extends to a harmonic map
on the whole disk $\D $, that gives rise to a conformal, branched
minimal immersion defined on $\D$ (see e.g., Gr\"{u}ter~\cite{gr1}).
Since $M$ is embedded near $p$, then $p$
cannot be a branch point; hence $M$ extends across $p$ to a smooth,
compact, embedded minimal surface. This finishes the proof of the
corollary.
\end{proof}

\section{The proof of Theorem~\ref{tttt}.}
\label{sec4}
Suppose $W$ is a closed countable  subset of $\R^3$ and
$\{M_n\}_n$ is a sequence of embedded minimal surfaces (possibly
with boundary) in $A=\R^3-W$, such that $\{ M_n\}_n$ has locally positive injectivity
radius in $A$.
We will first produce the possibly singular limit lamination
$\overline{\cal L}^A$ that appears in Theorem~\ref{tttt}.
If the $M_n$ have uniformly locally bounded
curvature in~$A$,
then it is a standard fact that a subsequence of the $M_n$ converges
to a minimal lamination $\lc$ of $A$ with empty singular set and
empty singular set of convergence (see for instance the arguments in
the proof of Lemma~1.1 in Meeks and Rosenberg~\cite{mr8}). In this case,
$\overline{\cal L}^A={\cal L}$ and ${\cal S}^A=\mbox{\rm \O}$. Otherwise,
there exists a point $p \in A$ such that, after replacing by a subsequence,
the supremum of the absolute Gaussian curvature of $M_n \cap \B (p, 1/k)$
diverges to $\infty $ as $n\to \infty $, for any $k$ fixed. Since $A$ is
open, we can assume $\B (p,1/k)\subset A$ for $k$ large and thus,
Proposition~1.1 in~\cite{cm35} (see also Theorem~13 in~\cite{mr13})
implies that the sequence of surfaces $\{ M_n\cap \B (p,1/k)\} _n$ is locally
simply connected in $\B (p,1/k)$. We will next describe both
the limit object of the surfaces $M_n\cap \B (p,1/k)$ as $n\to \infty $
and the surfaces themselves for $n$ large; this description relies on
Colding-Minicozzi theory and is adapted
from a similar description in~\cite{mpr14}; we have include it here as well
for the sake of completeness.

\begin{enumerate}[(D)]
\item For $k$ and $n$ large,  $M_n\cap \overline{\B }(p,1/k)$ consists
of compact disks with boundaries in
$\esf ^2(p,1/k)$.
By Theorem~5.8 in~\cite{cm22}, after a rotation of $\R^3$ and extracting a subsequence,
each of the disks $M_n\cap \overline{\B }(p, 1/k)$ contains a
$2$-valued minimal graph\footnote{In polar coordinates $(\rho,\theta )$ on $\R^2-\{ 0\} $
with $\rho > 0$ and $\t \in \R$, a $k$-valued graph on an annulus of inner radius $r>0$ and outer radius $R>r$,
is a single-valued graph of a function $u(\rho ,\t )$ defined over $\{ (\rho ,\t ) \ | \ r\leq \rho \leq R,\ |\t |\leq k\pi \} $,
$k$ being a positive integer.}
 defined on an annulus $\{ (x_1,x_2,0)\ | \ r_n^2\leq x_1^2+x_2^2\leq R^2\} $
with inner radius $r_n\searrow 0$, for certain $R\in (r_n,1/k)$ small but fixed.
By the one-sided curvature estimates and other
results in~\cite{cm23}, for some $k_0$
sufficiently large, a subsequence of the surfaces $\{  M_n\cap \B (p,
1/k_0) \}_n$ (denoted with the same indexes $n$)
converges to a possibly singular minimal lamination
$\overline{\lc_p}$ of $\B (p,1/k_0)$ with singular set
${\cal S}_p \subset \B (p, 1/k_0)$,
related (regular) minimal lamination ${\cal L}_p \subset \B (p,
1/k_0) -{\cal S}_p$ and singular set of convergence $S(\cL_p)\subset \cL_p$.
Moreover, $\cL_p$ contains a limit leaf with $p$ in its closure, that
is either a stable minimal disk $D(p)$ (if $p\in S(\cL _p)$) or a stable
 punctured minimal
 disk $D(p,*)$ (if $p\in \cS_p$), and in this last case $D(p,*)$
 extends smoothly across $p$ to a stable minimal disk $D(p)$
 that is a leaf of $\overline{\lc_p}$; this is Lemma II.2.3 in~\cite{cm25}.
In fact, $D(p)$ appears as a limit of
the previously mentioned 2-valued minimal graphs inside the $M_n$,
that collapse into it. In both cases,
the boundary of $D(p)$ is contained in $\esf ^2(p,1/k_0)$
and $D(p) \cap {\cal S}_p\subseteq \{ p\}$.
By Corollary~I.1.9 in~\cite{cm23}, there is a solid
double cone\footnote{A {\it solid double cone} in $\R^3$ is a set that after a
rotation and a translation, can be written as $\{ (x_1,x_2,x_3)\ | \ x_1^2+x_2^2
\leq \de ^{-2}x_3^2\} $ for some $\de >0$. A solid double cone in a ball is the intersection
of a solid double cone with a ball centered at its vertex.}
 ${\cal C}_p\subset \B (p, 1/k_0)$ with vertex at
$p$ and axis orthogonal to the tangent plane $T_pD(p)$, that
intersects $D(p)$ only at the point $p$ and such that the
complement of ${\cal C}_p$ in $\B (p,1/k_0)$ does not
intersect ${\cal S}_p$. Also, Colding-Minicozzi theory implies that
for $n$ large, $M_n\cap \B (p, 1/k_0)$ has the appearance
outside ${\cal C}_p$ of two highly-sheeted multivalued graphs over
$D(p,*)$, see Figure~\ref{fig2cone}.
\begin{figure}
\begin{center}
\includegraphics[width=10cm]{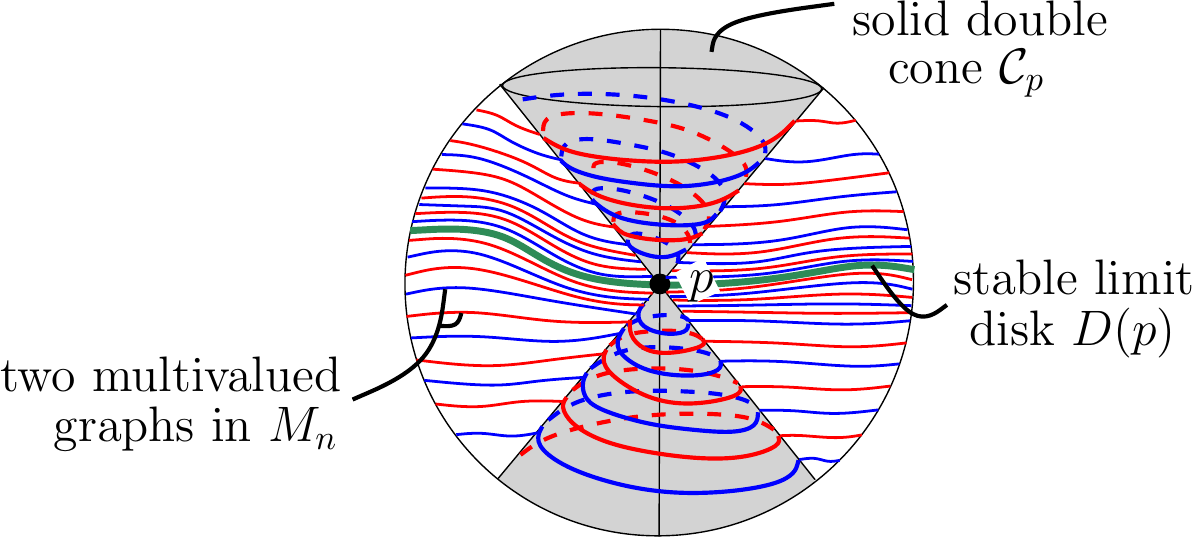}
\caption{The local picture of disk-type portions of $M_n$ around an
isolated point $p\in \cS$. The stable minimal punctured disk
$D(p,*)$ appears in the limit lamination ${\cal L}_p$, and extends smoothly
through $p$ to a stable minimal disk $D(p)$ that is orthogonal at $p$ to the axis of
the solid double cone ${\cal C}_p$.}
\label{fig2cone}
\end{center}
\end{figure}
Furthermore:
\end{enumerate}

\begin{enumerate}[(D1)]
\item If $p \in S(\cL_p)$ (in particular, $\overline{{\cal L}_p}=\cL_p$ admits a local lamination
structure around $p$), then after possibly choosing a larger $k_0$,
there exists a neighborhood of $p$
 in $\overline{\B}(p,1/k_0)$ that is foliated by compact disks in ${\cal L}_p$,
 and $S(\cL_p)$ intersects this family of disks transversely in a connected
Lipschitz arc. This case corresponds to case~(P) described in Section~II.2 of~\cite{cm25}.
In fact, the Lipschitz curve $S(\cL_p)$ around $p$ is a $C^{1,1}$-curve orthogonal to the
local foliation~(Meeks~\cite{me25, me30}), see Figure~\ref{CM} left.

\item If $p\in \cS_p$, then
after possibly passing to a larger $k_0$, a subsequence of the
surfaces $\{ M_n\cap \B (p,1/k_0)\}_n$
(denoted with the same indexes $n$) converges $C^\a$, $\a\in (0,1)$,
on compact subsets of $\B (p,1/k_0)-[\cS_p\cup S(\cL_p)]$ to
the (regular) lamination $\cL_p-S(\cL_p)$.
\end{enumerate}

To continue with the local description of case (D2), it is worth distinguishing
two subcases:
\begin{figure}
\begin{center}
\includegraphics[width=14cm]{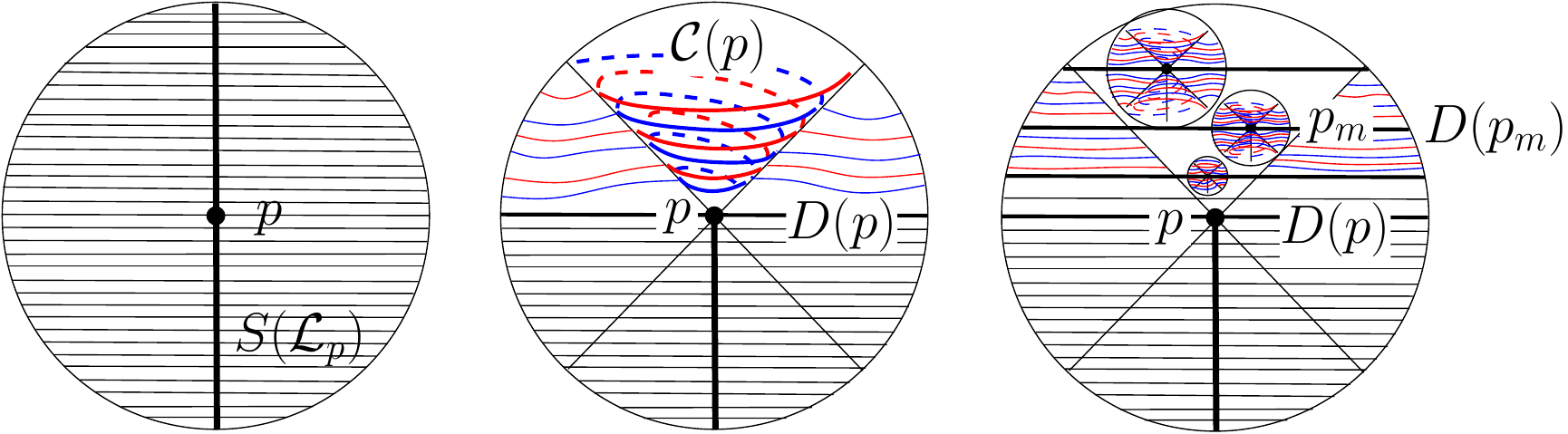}
\caption{Left: Case (D1), in a neighborhood of a point $p\in S({\cal L}_p)$.
Center: Case (D2-A) for an isolated point $p\in {\cal S}_p$. In the picture,
$p$ is the end point of an arc contained in $S({\cal L}_p)$, although $D(p,*)$ could
also be the limit of two pairs of multivalued graphical leaves, one pair on each side.
Right: Case (D2-B) for a nonisolated point $p\in {\cal S}_p$.}
\label{CM}
\end{center}
\end{figure}
\begin{description}
\item[{\rm (D2-A)}] If $p$ is an isolated point in $\cS_p$, then the limit leaf $D(p,*)$
of $\cL_p$ is either the limit of two pairs of multivalued graphical leaves in $\cL_p$ (one
pair  on each side of $D(p,*)$),
or $D(p,*)$ is the limit on one side of just one pair of multivalued graphical leaves
in $\cL_p$; in this last case, $p$ is the end point of an open arc $\G \subset
S(\cL_p)\cap \cC_p$, and a neighborhood of $p$ in the closure of the component
of $\B(p,1/k_0)-D(p,*)$ that contains $\G $ is entirely foliated by disk leaves of $\cL_p$,
see Figure~\ref{CM} center.

\item[{\rm (D2-B)}] $p$ is not isolated as a point in $\cS_p$. In this case, $p$ is the limit
of a sequence $\{ p_m\} _m\subset \cS_p\cap \cC_p$. In particular, $D(p)$ is the limit
of the related sequence of stable minimal disks $D(p_m)$, and $D(p,*)$ is the limit of a sequence of
pairs of multivalued graphical leaves of $\cL_p\cap [\B (p,1/k_0)-(\cC_p\cup \{ D(p_m)\} _m]$.
Note that these singular points $p_m$ might be isolated or not in $\cS_p$, see Figure~\ref{CM} right.
\end{description}

A standard diagonal argument implies, after extracting a
subsequence, that the sequence $\{M_n \}_n$ converges to a possibly
singular minimal lamination $\overline{\cal L}^A={\cal
L}\cup {\cal S}^A$ of $A$, with related (regular)
lamination ${\cal L}$ of $A-{\cal S}^A$, singular set ${\cal
S}^A\subset A$ and with singular set of convergence $S({\cal
L})\subset A-{\cal S}^A$ of the $M_n$ to ${\cal L}$. Furthermore,
in a neighborhood of every point $p
\in {\cal S}^A\cup S({\cal L})$, $\overline{\cal L}^A$ has the
appearance of the possibly singular minimal lamination $\overline{\lc_p}$
described above. Note that when $p\in S({\cal L})$,
since $S({\cal L})\subset {\cal L}$, then $\overline{{\cal L}_p}$ is a
regular minimal lamination that must be a foliation near $p$.

Next we describe the structure of the closure $\overline{\cal L}$ of
$\cal L$ in $\rth$. $\overline{\cal L}$ can be
written as
\begin{equation}
\label{eq:laminA}
 \overline{\cal L}=\overline{\cal
L}^A\cup  (W\cap \overline{\cal L}) =({\cal
L}\cup {\cal S}^A)\cup  (W\cap
\overline{\cal L})^{\mbox{\footnotesize lam}} \cup
(W\cap \overline{\cal L})^{\mbox{\footnotesize sing}},
\end{equation}
(all unions in (\ref{eq:laminA}) are disjoint), where
\[
\begin{array}{l}
(W\cap \overline{\cal L})^{\mbox{\footnotesize sing}}=\{ p\in W\cap
\overline{\cal L}\ \mid \ \overline{\cal L} \mbox{ does not admit
locally a lamination structure around }p\}
\\
\rule{0cm}{.4cm}(W\cap \overline{\cal L})^{\mbox{\footnotesize lam}}
=(W\cap \overline{\cal L})-
(W\cap \overline{\cal L})^{\mbox{\footnotesize sing}}.
\end{array}
\]
\par
\vspace{-.2cm}
\noindent
Consider the  set
\begin{equation}
\label{eq:S}
{\cal S}={\cal S}^A\cup  (W\cap \overline{\cal
L})^{\mbox{\footnotesize sing}},
\end{equation}
that is closed in $\rth$. If we define
\begin{equation}
\label{eq:reglam}
{\cal L}_1={\cal L}\cup (W\cap
\overline{\cal L})^{\mbox{\footnotesize lam}},
\end{equation}
 then ${\cal L}_1$
can be endowed naturally with a structure of a (regular) minimal
lamination of the open set $\R^3-{\cal S}$. Thus, the decomposition
(\ref{eq:laminA}) gives that $\overline{\cal L}$ is a possibly
singular lamination of $\R^3$, with singular set ${\cal S}$ and
related (regular) lamination ${\cal L}_1$, and so, to finish this
section in remains to prove items~1, \ldots , 7 in the statement of
Theorem~\ref{tttt}.

\begin{lemma}
\label{lemma4.1}
Items~1, 2 of Theorem~\ref{tttt} hold.
\end{lemma}
\begin{proof}
Item 1 holds since the
limit of a convergent sequence of planes is a plane.
Next we show that the first sentence in item~2 holds.

Consider a limit leaf
$\overline{\cal L}(L_1)=L_1\cup {\cal S}_{L_1}$
of $\overline{\cal L}$, where $L_1$ is its related leaf of
the regular lamination ${\cal L}_1$ and ${\cal S}_{L_1}$
is the set of singular leaf points of $L_1$. Thus,
$L_1$ is a limit leaf of ${\cal L}_1$.
As ${\cal L}_1$ is a regular lamination of $\R^3-{\cal S}$,
then the stable limit leaf theorem~\cite{mpr19,mpr18} applies
in this case and gives that the two-sided cover of $L_1$ is stable.
Since the set $\mbox{\rm Lim}({\cal L}_1)$ of limit leaves
of ${\cal L}_1$ forms a sublamination
(closeness of $\mbox{\rm Lim}({\cal L}_1)$
follows essentially from taking double limits),
then
the first sentence of item~2 of Theorem~\ref{tttt} reduces to checking that
$\overline{\cal L}(L_1)$ is a plane. To
do this, we will distinguish two cases.
\begin{enumerate}[(C1)]
\item If the $M_n$ have uniformly
locally bounded Gaussian curvature in $A$, then $\overline{\cal L}^A$ is a
(regular) minimal lamination of $A$, i.e., $\overline{\cal L}^A={\cal L}$
and ${\cal S}^A=\mbox{\rm \O}$. Hence, ${\cal S}\subset W$ and thus,
${\cal S}$ is a closed countable set of $\R^3$.
Applying Corollary~\ref{corrs} we deduce that $L_1$ extends across $\overline{L_1}\cap W$
and its two-sided cover is a stable minimal surface. Since
such an extension is clearly complete, it follows that the extension
of $L_1$ across $\overline{L_1}\cap W$ is a plane. But this extension
coincides with $\overline{\cal L}(L_1)$ and we are done in this case.

\item Suppose now that the $M_n$ do not have uniformly locally
bounded Gaussian curvature in~$A$. By construction, we can decompose
$L_1=L_A\cup [L_1\cap
(W\cap \overline{\cal L})^{\mbox{\footnotesize lam}}]$,
where $L_A$ is a leaf of the (regular) minimal lamination
${\cal L}=\overline{\cal L}^A-{\cal S}^A$ of $A-{\cal S}^A$.
Note that the two-sided cover of $L_A$ is stable,
since the same holds for $L_1$ by the stable limit leaf theorem~\cite{mpr19,mpr18}.

 Consider the union $\wt{L}_A$ of $L_A$
with all points $q\in {\cal S}^A$ such that the related punctured disk
$D(q,*)$ defined in (D) above is contained in $L_A$. Clearly, $\wt{L}_A$ is a (smooth)
minimal surface and the two-sided
cover of $\wt{L}_A$ is stable. We claim that $\wt{L}_A$ is complete outside $W$ in the sense that
every divergent  arc $\a \colon [0,1)\to \wt{L}_A$ of finite length has its limiting end point in $W$.

Arguing by contradiction, suppose that there exists a divergent arc
$\a \colon [0,l)\to \wt{L}_A$ of length $l$ such that
\[
q:=\lim _{t\to l^-}\a (t)\in \overline{\wt{L}_A}-W=\overline{L_A}^A=\overline{L_A}\cap A.
\]
Therefore, there exists $\de >0$ such that $\a (t)\in \overline{\B }(q,\ve )$ for every
$t\in [l-\de ,l)$, where $\overline{\B }(q,\ve )$ is the closed ball that appears in
description (D)
(with $\ve =1/k_0$),
and $\ve >0$ is taken sufficiently small so that $\overline{\B }(q,\ve )\subset A$.
Note that by construction, $\a (t)\notin D(q,*)$ for every $t\in [l-\de ,l)$.
As $D(q)$ separates $\overline{\B }(q,\ve )$, then $\a ([l-\de ,l))$ is contained in one of the
two halfballs of $\B (q,\ve )-D(q)$,
say in the upper ``halfball'' $\B ^+$
(we can choose orthogonal coordinates in $\R^3$ centered at $q$ so that $T_qD(q)$ is the $(x_1,x_2)$-plane).
In particular, there cannot exist a sequence $\{ q_m\} _m\subset \cS ^A$ converging to $q$
in $\B ^+$, because otherwise $q_m$ produces via (D2-B)
a related disk $D(q_m)$
that is proper in $\overline{\B ^+}$, such that the sequence $\{ D(q_m)\} _m$
converges to $D(q)$ as $m\to \infty$; as $\a (l-\de )$ lies
above one of these disks $D(q_k)$ for $k$ sufficiently large, then $\a ([l-\de ,l))$ lies
entirely above $D(q_k)$, which contradicts that $\g $ limits to $q$.
Therefore,  after possibly choosing a smaller $\ve $, we can assume that there are no
points of ${\cal S}^A$ in $\B ^+$ other than $q$.
Now consider the lamination $\cL'$ of $\B (q,\ve )-\{ q\} $ given
by $D(q,*)$ together with the closure of $L_A\cap \overline{\B ^+}$ in
$\B (q,\ve )-\{ q\} $. As the leaves of $\cL '$ are
all stable (if $L_A$ is two-sided; otherwise
we pass to a two-sided cover), then Corollary~\ref{corrs} implies that $\cL'$ extends
smoothly across $q$, which is clearly impossible. This contradiction proves our claim
that $\wt{L}_A$ is complete outside $W$.

Applying Corollary~7.2 in~\cite{mpr10} to $\wt{L}_A$, we deduce that closure of
$\wt{L}_A$ in $\R^3$ is a plane, which finishes the proof of the first
sentence in item~2 of Theorem~\ref{tttt}.
\end{enumerate}

As for the second sentence in item~2 of Theorem~\ref{tttt},
take a leaf $L=\overline{\cal L}(L_1)=L_1\cup {\cal S}_{L_1}$ of $\overline{\cal L}$
and suppose that $p$ is a point in $A\cap {\cal S}_{L_1}$. As $\{ M_n\} _n$ is
locally simply connected outside $W$, then the description
in (D)-(D1)-(D2) above implies that for $\ve >0$
sufficiently small, $L_1\cap \overline{\B }(p,\ve )$ equals the punctured disk $D(p,*)$ that
appears in this description, since $p$ is a singular leaf point of $L_1$, also see
Example~\ref{example1.3}-(B) in the Introduction. In particular,
$L_1\cap \overline{\B }(p,\ve )$ is a limit leaf of the local lamination in $\B (p,\ve )$ minus
a certain solid cone centered at $p$. As $L$ is connected, we
conclude that $L$ is a limit leaf of $\overline{L}$. In this situation, the  first sentence
in item~2 of Theorem~\ref{tttt} implies that $L\in {\cal P}$. This finishes the proof of
Lemma~\ref{lemma4.1}.
\end{proof}

Next we prove item~4
of Theorem~\ref{tttt}, since we shall made use of it in the proof of item~3
of the same theorem.

\begin{lemma}
Item 4 of Theorem~\ref{tttt} holds.
\end{lemma}
\begin{proof}
The local picture
of $\overline{\lc}^A$ described in (D)-(D1)-(D2)
implies that through each point $q\in {\cal
S}^A\cup S({\cal L})$ there passes a limit leaf of $\overline{\cal
L}$ that, by item~{2} of Theorem~\ref{tttt}, must be a plane $P_q\in
\mbox{\rm Lim}(\overline{\cal L})$. Next we will prove that $P_q\cap
({\cal S}^A\cup S(\lc) \cup W)$ is a closed countable set. By the same local picture,
we have that $P_q\cap ({\cal S}^A\cup S(\lc))$ is a discrete
subset of $P_q-W$, that is clearly closed in the intrinsic topology
of $P_q-W$. Thus the limit points of
$P_q\cap ({\cal S}^A\cup S(\lc))$ lie in the closed countable
set $P_q\cap W$. It then follows that  $P_q\cap ({\cal S}^A\cup
S(\lc) \cup W)$ is a closed countable set of $\R^3$, and
the lemma follows.
\end{proof}

\begin{lemma}
\label{lemma3.4}
Item 3 of Theorem~\ref{tttt} holds.
\end{lemma}
\begin{proof}
Suppose that
$P$ is a plane in ${\cal P} -\mbox{\rm Lim}(\overline{\cal L})$. Since
$\mbox{\rm Lim}(\overline{\cal L})$ is a closed set of planes, we can choose
$\de >0$ such that the $2\delta$-neighborhood of $P$ is disjoint
from  $\mbox{\rm Lim}(\overline{\cal L})$. By item~{4}
of Theorem~\ref{tttt}, $S(\lc)\cup {\cal S}^A$ is at a positive distance at
least $2\de $ from $P$.

If the $\delta$-neighborhood $P(\delta)$ of $P$ intersects $\overline{\lc}$
in a portion of some leaf $L'$ of $\overline{\lc}$
different from $P$, then $L'\cap P(\de )$, while it
may have singularities in $W$, is proper as a set in $P(\delta)$:
properness of the smooth surface $L'\cap [P(\de )-W]$ is clear (as $P(2\de )$
is disjoint from Lim$(\overline{\cal L})$); hence $L'\cap P(\de )$ only intersects
$W$ in singular leaf points.

We now check that $L'$ is disjoint from $P$.
Arguing by contradiction, suppose that $L'$ and $P$ intersect. Note that
every such intersection point $q$ must lie in $W$
by the maximum principle, and that $q$ is a singular leaf point of both
$L'$ and $P$. This is impossible by Proposition~\ref{propos1}
since $W$ is countable. Therefore, $L'$ does not intersect $P$.
In this setting, we can use
the proof of the halfspace theorem (Hoffman and Meeks~\cite{hm10})
with catenoid barriers (adapted to this situation with
countably many singularities via Proposition~\ref{propos1})
to obtain a contradiction to the existence of $L'$.
Hence, $P(\delta) \cap \overline{\lc} =
P$, which proves the lemma.
\end{proof}

\begin{lemma}
\label{lemma3.3}
Item 5 of Theorem~\ref{tttt} holds.
\end{lemma}
\begin{proof}
Suppose now that $p \in W\cap \overline{\cal L}$ satisfies one of
the conditions of items~5.1, 5.2 in Theorem~\ref{tttt}. First note
that if $p$ lies in the closure of ${\cal S}^A\cup S({\cal L})$,
then there passes a plane in ${\cal P}$ through $p$ by items~1
and~4 of Theorem~\ref{tttt} and we have the conclusion of item~5
in this case. Otherwise, we find an $R>0$ such that the closed
ball $\overline{\B }(p,R)$ does not intersect ${\cal S}^A\cup
S({\cal L})$. In particular, $\overline{\cal L}\cap [\overline{\B }
(p,R)-W]={\cal L}\cap [\overline{\B }(p,R)-W]$ and the surfaces
$M_n$ converge to ${\cal L}$ on compact subsets of $\overline{\B }(p,R)-W$.
Arguing by contradiction, suppose no plane in ${\cal P}$
passes through $p$. Since ${\cal P}$ is closed in $\R^3$, then we can assume
no plane in ${\cal P}$ intersects $\overline{\B }(p,R)$, and hence
item~2 of Theorem~\ref{tttt} implies that each leaf of $\overline{\cal L}\cap
[\overline{\B }(p,R)-W]$ is proper in $\overline{\B }(p,R)-W$. By
Proposition~\ref{propos1}, the leaves of $\overline{\cal L}\cap
\overline{\B }(p,R)$ are compact in $\overline{\B }(p,R)$ and
pairwise disjoint.

Let $L_p$ be the leaf of $\overline{\cal L}\cap
\overline{\B }(p,R)$ that passes through $p$ ($p$ is a singular leaf point of the
regular part of $L_p$, which in turn is contained in $L_p-W$). Note that the
distance between $L_p$ and the other leaves of $[\overline{\cal
L}\cap \overline{\B }(p,R)]-L_p$ is positive, as follows also from
Proposition~\ref{propos1} together with the fact that $L_p$ is not a
limit leaf.
If 5.1 or 5.2 holds, then given a compact disk $D\subset L_p - W$ and
$\ve\in (0,R)$, there exists an integer $n_0=n_0(D,\ve)$
such that for $n\geq n_0$, there exist
two pairwise disjoint disks $D_1^n$, $D_2^n$ in $M_n$ such that these disks are
normal graphs over $D$ with graphing functions $f_1^n,f_2^n$, respectively,
each having norms less than $\ve$.
\begin{figure}
\begin{center}
\includegraphics[width=9.cm]{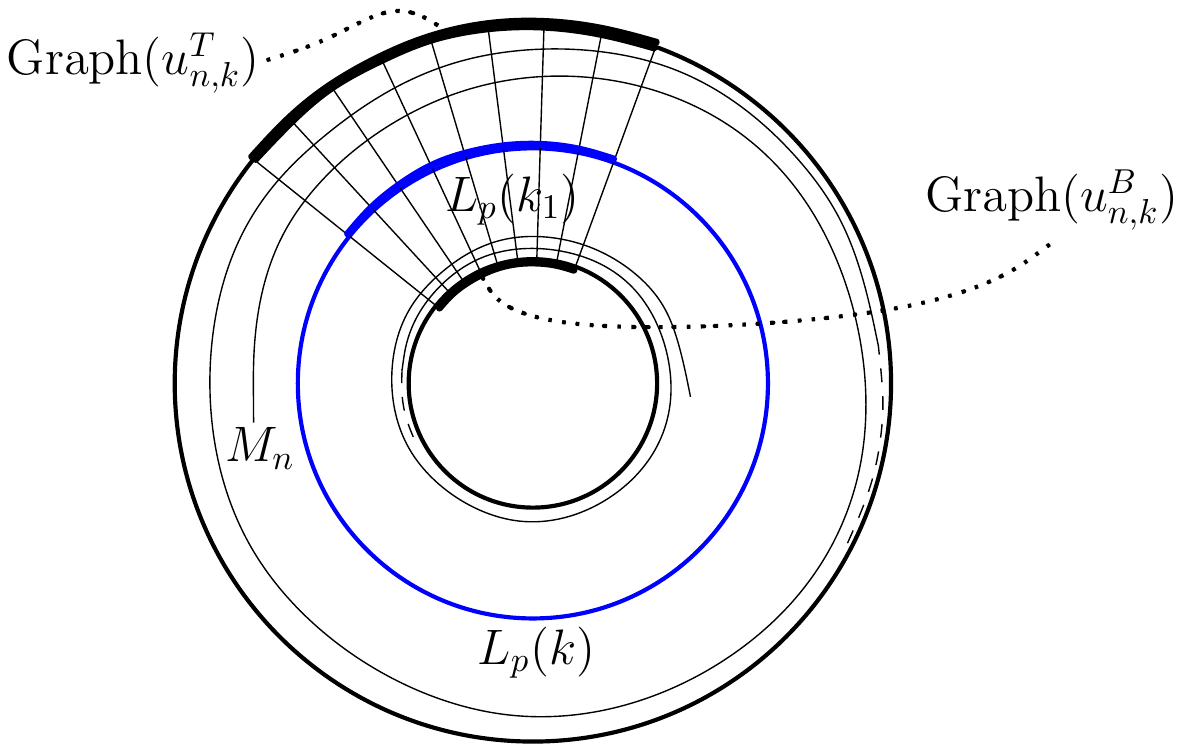}
\caption{Schematic representation of the construction of a positive Jacobi function on
$L_p(k)-\partial L_p(k)$ (we have simplified the figure
by taking one dimension less); here $k_1<k$,
the components of $M_n\cap U(k_1)$
are normal graphs over $L_p(k_1)-\partial L_p(k_1)$
but the components of $M_n\cap U(k)$ are infinitely valued
graphs
over $L_p(k)-\partial L_p(k)$, created by holonomy. In any case, the closure
$M_n\cap U(k)$ in $U(k)$ contains a ``top'' graphical leaf Graph$(u_{n,k}^T)$
and a ``bottom'' one Graph$(u_{n,k}^B)$ that are disjoint.
}
 \label{fig4}
\end{center}
\end{figure}
Observe that if we replace the disk $D$ by a compact subdomain in $L_p-W$, then
the graphing functions $f_1^n,f_2^n$ might fail to be univalent.
More precisely, consider a smooth\footnote{Smoothness of the compact exhaustion
can be assumed as $\esf ^2(p,R)$ can be supposed to be transverse to $L_p$.
Note that the topological boundary $\partial L_p$ is nonempty and contained in $\esf ^2(p,R)$.}
compact exhaustion of $L_p -W$
\[
L_p(1)\subset L_p(2)\subset \ldots \subset L_p(k) \subset \ldots \qquad \mathrm{with}
\quad \partial L_p \subset \partial L_p(1).
\]
Fix $k\in \N$ large and consider the $r(k)$-normal open regular neighborhood
\[
U(k)=\{ x+tN(x)\ | \ x\in L_p(k)-\partial L_p(k), \
|t|<r(k)\} ,
\]
where $N$ stands for the unit normal vector to $L_p$, and $r(k)\in (0,1/k]$ is to be defined.
For $k$ sufficiently large, $U(k)$ is embedded in $\R^3$ for some $r(k)\in (0,1/k]$. For
$n\geq k$ sufficiently large, each component of $M_n\cap U(k)$ is either a normal graph
or an infinitely valued graph
over $L_p(k)-\partial L_p(k)$. In both cases, the validity of 5.1 or 5.2 implies that
the closure of  $M_n\cap U(k)$ in $U(k)$ contains two distinct
leaves that are normal graphs over $L_p(k)-\partial L_p(k)$ with graphing functions
$u_{n,k}^T,u_{n,k}^B\colon L_p(k)-\partial L_p(k) \to [-r(k),r(k)] $; see
Figure~\ref{fig4}.
In order to obtain this description we are using that the
surfaces $M_n\cap \B(p,R)$ have locally bounded
Gaussian curvature in $ \B(p,R)-W$, $ \B(p,R)-W$ is
simply connected and so $L_p-W$ is a two-sided minimal surface,
and the fact that the leaf $L_p$ is a positive distance
from the other leaves of $\overline{\cL}\cap \B(p,R)$.
If we fix a point $p_0\in L_p(k)-\partial L_p(k)$, then
a subsequence of the positive functions
\[
f_{n,k}=\frac{1}{(u_{n,k}^T-u_{n,k}^B)(p_0)}(u_{n,k}^T-u_{n,k}^B)
\]
converges as $n\to \infty $ to a positive Jacobi function $f$ on $L_p(k)-\partial L_p(k)$.
This proves that $L_p(k)-\partial L_p(k)$ is stable for every $k$, which gives that
$L_p-W$ is also stable. By Corollary~\ref{corrs} we deduce that
$L_p-W$ extends across $W\cap \ov{\B} (p,R)$ to a smooth compact minimal surface that is $L_p$.

Let $L=L_1\cap {\cal S}_{L_1}$ denote the leaf of $\overline{\lc}$ that contains $L_p$,
where $L_1$ is the (smooth) leaf of the regular part ${\cal L}_1$ of ${\cal L}$
defined in (\ref{eq:reglam}), and ${\cal S}_{L_1}$ is the set of singular leaf points of $L_1$.
As no plane in ${\cal P}$ passes through $p$, then $L_1$ is not flat and so,
$L_1$ is not a limit leaf of ${\cal L}_1$  by item~2 of Theorem~\ref{tttt}.
Let Lim$(L_1)$ be the set of limit points of $L_1$. We claim that through
every point $q\in \mbox{Lim}(L_1)\cap A$ there passes a plane that is contained
in $\overline{L}$: If $q\in {\cal S}^A$, this follows from item~4 of Theorem~\ref{tttt};
if on the contrary $q\in A-{\cal S}^A$, then the leaf $L_2$ of ${\cal L}_1$
that passes through $q$ is a limit leaf of ${\cal L}_1$, and thus, $\overline{L_2}$
is a plane by item~2 of Theorem~\ref{tttt}. Now our claim holds.
As through every point of Lim$(L_1)\cap A$ there passes a plane in $\overline{L}$, then
a connectedness argument shows that $L_1$ is proper in $\Delta -W$,
where $\Delta \subset \R^3$ is
either an open halfspace or an open slab. As $\Delta -W$ is simply connected and $L_1$ is
properly embedded in $\Delta -W$, then $L_1$ is orientable.
 Now consider a compact subdomain $\Omega \subset L_1$. As $L_1$ is not a limit leaf
 of ${\cal L}_1$, then $\Omega $ is at a positive distance from any leaf of ${\cal L}_1$
 different from $L_1$. In particular, $\Omega $ admits an normal open neighborhood
 that is disjoint from any other leaf of ${\cal L}_1$. In this setting,
 we can repeat the argument in the previous paragraph to construct
 a positive Jacobi function on $\Omega $,
 which proves that $\Omega $ is stable. As $\Omega $ is any compact subdomain in $L_1$,
 then we conclude that $L_1$ is stable as well.

We next prove that $L_1$ stays at a positive distance from every point $p_1\in {\cal S}^A$:
again arguing by contradiction, if this property fails to hold for a point $p_1\in {\cal S}^A$, then
portions of $L_1$ enter in every ball $\B (p_1,\ve )$ of arbitrarily small radius.
In this setting, the local description in (D)-(D1)-(D2) for a sufficiently
small ball $\B (p_1,\ve )$ implies that either $L_1$ contains
 the punctured disk $D(p_1,*)$ that appears in (D), or $L_1\cap
 \B (p_1,\ve )$
 contains two multivalued graphs $\Sigma $ that spiral together
infinitely many times into $D(p_1,*)$
 at one side of $D(p_1,*)$. The first possibility cannot
 occur as $L_1$ is not a limit leaf of ${\cal L}_1$;
 the second possibility cannot occur either, by
 Theorem~\ref{tt2} applied to the lamination
 $\Sigma \cup D(p_1,*)$ of $\B (p_1,\ve )-\{ p_1\} $,
 because $\Sigma $ is stable. This contradiction
 shows that $L_1$ stays at a positive distance from every point $p_1\in {\cal S}^A$.

 Finally, as $L_1$ is a leaf of the lamination ${\cal L}_1$ of $\R^3-{\cal S}$
 (the singular set ${\cal S}$ was defined in (\ref{eq:S})) and
  $L_1$ stays at a positive distance from every point $p_1$ of ${\cal S}^A$, then we deduce
  that $L_1$ is complete outside $W$. As $L_1$ is stable, then Corollary~\ref{corrs} implies
  that $L_1$ extends across $W$ to a complete stable minimal surface in $\R^3$, hence
  a plane passing through $p$, which is absurd. Now the proof of the lemma is complete.
\end{proof}

\begin{proposition}
\label{propos3.5}
Item~6 of Theorem~\ref{tttt} holds.
\end{proposition}
\begin{proof}
 Let $L=L_1\cup {\cal S}_{L_1}$ be a nonplanar leaf of $\overline{\cal L}$,
where $L_1$ is the leaf of the regular part ${\cal L}_1$ of $\overline{\cal L}$
defined in~(\ref{eq:reglam}) and ${\cal S}_{L_1}$ is the set of singular leaf
points of $L_1$. As the argument to prove the proposition is
delicate, we will organize it into four assertions.
\begin{assertion}
\label{ass1}
$L\cap ({\cal S}^A \cup S(\lc)) = \mbox{\rm \O}$ and the convergence of
portions of the $M_n$ to $L_1$ is of multiplicity one.
\end{assertion}
\begin{proof}
If $L$ intersects $S({\cal L})$ at a point $x$, then $L$ is a smooth minimal surface
around~$x$. Since item~4 of Theorem~\ref{tttt} implies that there passes
a plane $P_x\in {\cal P}$ through $x$, we conclude that $L=P_x$, which is impossible.
If $L$ intersects ${\cal S}^A$ at a point $y$, then $y\in L\cap {\cal S}={\cal S}_{L_1}$
where ${\cal S}$ is the singular set of $\overline{\cal L}$ defined in (\ref{eq:S}).
By item~4 of Theorem~\ref{tttt}, there passes
a plane $P_y\in {\cal P}$ through $y$, which implies that both $L_1$, $P_y$
share the singular leaf point $y$. Since $P_y$ intersects ${\cal S}$ in a closed
countable set (again by item~4 of Theorem~\ref{tttt}), then
Proposition~\ref{propos1} leads to a contradiction. Therefore, we have proved that
$L\cap ({\cal S}^A \cup S(\lc)) = \mbox{\rm \O}$. Finally, the property that
the convergence of portions of the $M_n$ to $L_1$ is of multiplicity one
follows from the proof of Lemma~\ref{lemma3.3}. Now Assertion~\ref{ass1} follows.
\end{proof}

To prove that either item~6.1 or 6.2 of Theorem~\ref{tttt}
holds, we will distinguish two cases,
depending on whether or not $L$ is proper as a set in $\R^3$.
\par
\vspace{.2cm}
\noindent (E1) {\sc Suppose that $L$ is proper in $\R^3$.}\newline
Our goal is to show that item~6.1 of Theorem~\ref{tttt} holds.
Since $L$ is proper in $\R^3$, all the points
in ${\cal S}\cap \overline{L}$ are singular leaf points
of $L_1$, in particular $L=\overline{L}$.
In this setting, the proof of the halfspace theorem  that uses catenoid barriers together with
Proposition~\ref{propos1} imply ${\cal P}=\mbox{\rm \O}$.
By item~4 of Theorem~\ref{tttt}, we have ${\cal S}^A\cup S({\cal L})=\mbox{\O }$.
Thus, ${\cal S}\subset W$ by equality (\ref{eq:S}).
To deduce item~6.1 of Theorem~\ref{tttt}, it remains to prove
that $L $ is the unique leaf of $\overline{\cal L}$. Otherwise,
$\overline{\cal L}$ contains a leaf $L'\neq
L$, and $L'$ is not flat since ${\cal P}=\mbox{\rm \O}$. Furthermore,
$L'$ is proper in $\R^3$ (if $L'$ were nonproper then
$\overline{L'}$ would contain a limit leaf that is a plane in ${\cal P}$).
Proposition~\ref{propos1} implies that $L$ and $L'$ do not
intersect.  The existence of the proper,  possibly singular surfaces
$L,L'$ contradicts the proof of the strong halfspace theorem
adapted to this singular setting via Proposition~\ref{propos1}
(see \cite{hm10,msy1}), in which one first constructs a plane between
$L$ and $L'$ and then applies the proof of the halfspace theorem.
This proves that item~6.1 of Theorem~\ref{tttt} holds, as desired.
\par
\vspace{.2cm}
\noindent (E2) {\sc Suppose that $L$ is not proper in $\R^3$.} \newline
In this second case we will demonstrate that item~6.2 of Theorem~\ref{tttt} holds,
which will finish the proof of Proposition~\ref{propos3.5}.
As $L$ is not proper in $\R^3$,
there exists a limit point $q_0$ of $L$, in the sense that
there exists a sequence of points in $L$ that
converges to $q_0$ in $\R^3$ and that is intrinsically
divergent in $L$. Therefore, $L_1$ is also nonproper
in any extrinsic neighborhood of $q_0$, which implies
that $q_0$ is not a singular leaf point of
$L_1$ and thus, $q_0$ is not contained in $L$.

\begin{assertion}
\label{ass3.6}
Through any limit point $q$ of $L$ there passes a
plane $P\in {\cal P}$. Furthermore,
every point in  such a plane $P$ is a limit point of $L$.
\end{assertion}
\begin{proof}
First note that such a limit point $q$ of $L$
cannot lie in $L$, by the discussion in the last paragraph. If $q$ lies in the regular
lamination~${\cal L}_1$, then the leaf $L'_1$ of $\cL_1$ that contains $q$
is a limit leaf of ${\cal L}_1$.
By the arguments in the proof of Lemma~\ref{lemma4.1}, the closure
$\overline{L'_1}$ must be a plane in ${\cal P}$, and the assertion holds in this case.
Then, we may assume $q\in (\overline{L_1}\cap {\cal S})-{\cal S}_{L_1}$.
By Definition~\ref{defsinglamin},
this implies that for every open neighborhood $V$ of $q$ in $\R^3$,
then $L_1\cap V$ fails to be closed in $V-{\cal S}$.
Thus one can find a sequence $\{  V_k\} _k$
of open neighborhoods of $q$ and a sequence of points $x_k\in
\overline{L_1\cap V_k}^{V_k-{\cal S}}-(L_1\cap V_k)$, $k\in \N $.
Without loss of generality, we can assume $V_k\to \{ q\} $ as $k\to
\infty $. Fix $k\in \N$. Since $x_k$ lies in the closure of $L_1\cap
V_k$ relative to $V_k-{\cal S}$, then there exists a sequence $\{
y_k(m)\} _m\in L_1\cap V_k$ with $y_k(m)\to x_k$ as $m\to \infty $.
As $x_k\in (V_k-{\cal S})-(L_1\cap V_k)$, then $x_k\notin L_1$. Thus
$\{ y_k(m)\} _m$ converges to $x_k$ in the topology of $\R^3$ but it
does not converge to $x_k$ in the intrinsic topology of $L_1$
(otherwise $x_k$ would lie in $L_1$ since $x_k\notin {\cal S}$).
This gives that $x_k\in \mbox{Lim}(L_1)$, and our previous arguments
imply that there passes a plane in ${\cal P}$ through $x_k$. Since
this happens for all $k$, $x_k\to q$ as $k\to \infty $ and ${\cal
P}$ is a closed set of planes, then there also passes a plane in
${\cal P}$ through $q$ and the assertion is proved.
\end{proof}

We continue with the proof of item~6.2 of Theorem~\ref{tttt} in case (E2).
Since $L$ is not proper in $\R^3$ and through any limit point of $L$ there
passes a plane in ${\cal P}$, a straightforward connectedness
argument shows that $\overline{L}=L\cup {\cal P}(L)$ with ${\cal
P}(L)$ consisting of one or two planes. In particular, $L$ is
proper in the component $C(L)$ of $\R^3-{\cal P}(L)$ that contains
$L$.

\begin{assertion}
\label{ass3.7}
In the above situation, $C(L)\cap \overline{\cal L}=L$.
\end{assertion}
\begin{proof}
Since $L$ is
connected and nonflat, there are no planar leaves of
$\overline{\lc}$ in $C(L)$. Reasoning by contradiction, suppose that
$L'$ is a nonflat leaf of $\overline{\lc}$ that is different from
$L$ and that intersects $C(L)$. Since $L$ and $L'$ are proper
in $C(L)$, the maximum principle together with Proposition~\ref{propos1}
imply that $L\cap L'=\mbox{\rm \O}$.
Reversing the roles of $L$ and $L'$
one can easily check that ${\cal P}(L) = {\cal P}(L')$ and $C(L) =
C(L')$. As both $L-{\cal S}$ and $L'-{\cal S}$ are properly embedded smooth surfaces in the
simply connected region $C(L)-{\cal S}$ (because ${\cal S}\cap C(L)\subset W$ is countable),
then $L\cup L'$ bounds a closed
region $X$ in $C(L)$; since the two boundary components of $X$ are
good barriers for solving Plateau problems in $X$ (in spite of being
singular by using Proposition~\ref{propos1}), a standard argument
(see Meeks, Simon and Yau~\cite{msy1}) shows that there exists a
properly embedded, least-area surface $\Sigma\subset X$ that
separates $L$ from $L'$ in $X$, and hence separates $L$ from $L'$ in $C(L)$. However, since $X$ is not
necessarily complete (note that every divergent path in $X$ with
finite length must have a limit point in ${\cal S}\cap [L\cup L'\cup
{\cal P}(L)]$), then the surface $\Sigma$ might fail to be complete.
On the other hand, Assertion~\ref{ass1} applied to $L,L'$ implies
that neither of the surfaces $L,L'$ intersects
${\cal S}^A \cup S(\lc)$, because both $L,L'$ are not flat. This implies that
${\cal S}\cap (L\cup L')\subset W$; in particular,
${\cal S}\cap [L\cup L'\cup {\cal P}(L))]$ is closed and countable.
As $\Sigma$, when considered to
be a surface in $\R^3$, is complete outside the closed countable
${\cal S}\cap [L\cup L'\cup {\cal P}(L))]$, then Corollary~\ref{corrs} implies
that $\Sigma$ extends to a complete, stable minimal surface
$\overline{\Sigma}$ in $\R^3$.
Therefore, $\overline{\Sigma}$ is a plane. This is impossible as
${\cal P}(L) = {\cal P}(L')$ but $L$ and $L'$
lie on opposite sides of a plane.
This proves the assertion.
\end{proof}

\begin{assertion}
\label{ass3.88}
Every open $\ve$-neighborhood $P(\ve)$ of a plane $P\in {\cal
P}(L)$ intersects the surface $L_1$ in a connected smooth surface
with unbounded Gaussian curvature.
\end{assertion}
\begin{proof}
After a rotation, we may assume that $P=\{ x_3=0\}
$ and $L$ limits to $P$ from above $P$. Given $\ve
>0 $ small enough so that $\{ 0<x_3\leq \ve \} \subset C(L)$, we consider
the smooth minimal surface
\begin{equation}
\label{eq:L1(ve)}
L_1(\ve )=L_1\cap \{ 0<x_3\leq \ve \} .
\end{equation}
 Note that $L_1(\ve )$ is possibly incomplete
(completeness of $L_1(\ve )$ may fail in the set ${\cal S}\cap \{
0\leq x_3\leq \ve \} $). Since ${\cal S}\subset {\cal S}^A\cup W$,
$W$ is countable and $L\cap {\cal S}^A
=\mbox{\O }$ by item~4 of Theorem~\ref{tttt}, then we may also
choose $\ve $ so that the closure $\overline{L_1(\ve )}$ in $\rth$ of
$L_1(\ve )$ does not have singularities in the plane $\{ x_3=\ve \} $.
In a similar way as in the proof of Assertion~\ref{ass3.7}, applying the proof of
Theorem~1.6 in~\cite{mr8} and using the local extendability of a
stable minimal surface in $\overline{C(L)}$ that is complete
outside a closed countable set  and has its boundary in a
plane in $C(L)$, one sees that $\{ 0\leq x_3\leq \ve \} $ intersects
$L$ in a connected set.

We next prove that the Gaussian curvature of $L_1(\ve )$ is
unbounded. Reasoning by contradiction, assume $L_1(\ve )$ has bounded
Gaussian curvature. In this case, $L_1(\ve )\cup [P-
(W\cap \overline{\cal L})^{\mbox{\footnotesize sing}}]$ is a
relatively closed set of $\{ -1<x_3<\ve \} -
(W\cap \overline{\cal L})^{\mbox{\footnotesize sing}}$ with bounded
second fundamental form, hence $L_1(\ve )\cup [P-
(W\cap \overline{\cal L})^{\mbox{\footnotesize sing}}]$ is a
minimal lamination of $\{ -1<x_3<\ve \} -
(W\cap \overline{\cal L})^{\mbox{\footnotesize sing}}$.
By Theorem~\ref{tt2}, $L_1(\ve ) \cup P$ is a minimal
lamination of $\{ -1<x_3 < \ve \}$.  In this situation with
bounded Gaussian curvature, one can
apply Lemma~1.4 in~\cite{mr8} to deduce that $L_1(\ve )$ is a graph
over its projection to $P$, in particular it is proper in the
closed slab $\{ 0\leq x_3\leq \ve \} $, which contradicts the proof
of the Halfspace Theorem. Hence, $L_1(\ve )$ has unbounded Gaussian
curvature.
\end{proof}

The main statements of item~6.2 and item~6.2(a)  of Theorem~\ref{tttt} are now
proven under the hypothesis
of Case~(E2); it remains to prove that the additional statements 6.2(b) and 6.2(c) hold to complete the
proof of Proposition~\ref{propos3.5}.
This is a technical part of
the proof, where the local picture theorem on the scale of topology~\cite{mpr14} will play a crucial role.

\begin{remark} {\em
In the first item of the next assertion, one can ask  if it is the case that every open $\ve$-neighborhood $P(\ve)$ of $P$
intersects the surface $L_1$ in a connected smooth surface
with infinite genus,  without making the additional assumption that the plane $P\in {\cal
P}(L)$  contains a singularity of $\ov{L}$. The answer to this question is unclear to the authors.
}
\end{remark}

\begin{assertion}
\label{ass3.8}
\ben
 \item If a plane $P\in {\cal
P}(L)$  contains a singularity of $\ov{L}$,
then every open $\ve$-neighborhood $P(\ve)$ of $P$
intersects the surface $L_1$ in a connected smooth surface
with infinite genus.
\item The leaf $L_1$ has infinite genus.
\een
\end{assertion}
{\it Proof.}
We first prove that item~1 of the assertion implies item~2.
Suppose that $L_1$ has finite genus and item~1 holds.  Item~1 implies that  each plane
in ${\cal
P}(L)$ contains no singularities of $\ov{L}$. As $L$ is proper in $C(L)$, then
Corollary~\ref{corol2.3} implies that $\overline{L_1}$
has no singularities in $C(L)$ (to see this,
observe that such a singularity $q$ would belong to $W$, hence
$q$ could be assumed to be isolated in $W$ by Baire's Theorem,
and now Corollary~\ref{corol2.3} applies to give a contradiction).
Hence, $\ov{L}$ is a minimal lamination of $\rth$ whose leaves are the
nonflat surface $L_1$ together with the
nonempty set of planes in ${\cal P}(L)$. The fact that $L_1$ has finite genus
genus contradicts Corollary~1
in~\cite{mpr3}, which states that every nonplanar leaf of a minimal lamination of $\rth$
with  more than one leaf has infinite genus.

We next prove item~1 holds; this will complete the proof of the assertion and the proof
of Proposition~\ref{propos3.5}. Suppose that the $(x_1,x_2)$-plane $P\in {\cal
P}(L)$  contains a singularity of $\ov{L}$.
To finish the proof of Assertion~\ref{ass3.8}, it remains to
demonstrate that for every $\ve >0$,
the surface $L_1(\ve )$ given by (\ref{eq:L1(ve)}) has infinite genus.
If this infinite genus property were
to fail, then we first choose $\ve$ sufficiently small so that
$L_1(\ve )$ has genus zero, keeping the property that
$\overline{L_1(\ve )}$ does not have singularities in the plane
$\{ x_3=\ve \} $. As $L$ is proper in $C(L)$, then
Corollary~\ref{corol2.3} implies that $\overline{L_1(\ve )}$
has no singularities in $\{ 0<x_3\leq\ve \} $.
Thus,  $L_1(\ve )$ is a
smooth, connected minimal surface with genus zero, that is complete
outside a nonempty closed countable set
\begin{equation}
\label{def:S}
{\cal S}'\subset {\cal S}\cap P
\end{equation}
(none of the points in ${\cal S}'$ can be a singular leaf point of $L_1$,
by Proposition~\ref{propos1}),
and the boundary of $L_1(\ve )$ lies in the
plane $\{ x_3=\ve \} $.
Since ${\cal S}'$ is nonempty, closed and countable, Baire's
 Theorem insures that
there exists an isolated point $q\in{\cal S}'$. After a translation
and homothety, assume $q=\vec{0}$ and ${\cal S}'\cap\B(2\delta)=\{
\vec{0}\} $ for some positive $\delta<\frac{\ve}{2}$.

Let $I_{L_1}$ be the injectivity radius function of $L_1$. We will find
the desired contradiction by discussing the cases (E2-A), (E2-B) below,
depending on whether or not $(I_{L_1})/|\cdot|$ is bounded away from zero
 in $L_1\cap \B(\delta) $
($|.|$ denotes distance to the origin in $\R^3$).
\par
\vspace{.2cm}
 \noindent
(E2-A) {\sc Suppose that $(I_{L_1})/|\cdot|$ is
not bounded away from zero in $L_1\cap \B(\delta) $.}
\newline
We will use the local picture theorem on the scale of topology
together with a flux argument to discard this case.
By Theorem~1.1 in~\cite{mpr14} (see also Remark~4.31 in the same paper),
there exists a sequence of points $\{ p_n\}_n\subset L_1$
called {\it points of almost-minimal injectivity radius,} such
that the following properties hold:
\begin{enumerate}[(F1)]
\item $p_n\to \vec{0}$ and $\frac{I_{L_1}(p_n)}{|p_n|}\to 0$ as $n \to \infty$.
\item For all $n\in \N$, there exists $\ve _n\in (0,|p_n|/2)$
such that the closure $L(n)$ of the component of $L_1(\ve )\cap \B(p_n,
\ve_n)=L\cap \B(p_n,\ve_n)$ that contains $p_n$ is compact
and has its boundary in $\esf ^2(p_n,\ve _n)$.
\item Defining $\lambda _n=1/I_{L_1}(p_n)\in \R^+$, then:
\begin{enumerate}[(F3.1)]
\item The injectivity radius function of $L_1$ restricted to $L(n)$,
denoted by $I_{L(n)}$, satisfies $\lambda _nI_{L(n)}\geq
1-\frac{1}{n}$ on $L(n)$.
\item $\l _n\, \ve _n\to \infty $ as $n\to \infty $.
\item The sequence of surfaces $\{ \widehat{L}(n):= \l _n[L(n)-p_n]\} _n$
converges as $n\to \infty $ to either a nonsimply connected, properly
embedded minimal surface $\widehat{L}(\infty )\subset \R^3$ of genus zero or
to a minimal parking garage structure in $\R^3$ with two oppositely oriented columns.
By the classification of genus zero properly embedded minimal
surfaces in $\R^3$~\cite{col1,lor1,mpr6}, the surface
$\widehat{L}(\infty )$ is either a catenoid
or a Riemann minimal example if it occurs.
\end{enumerate}
\end{enumerate}

We first consider the case where the limit object of
$\{ \widehat{L}(n)\}_n$ is a catenoid~$\widehat{L}(\infty )$.
Let $\Pi $ be the plane in $\rth$ that
intersects $\widehat{L}(\infty )$ orthogonally along its waist
circle $\Gamma$. Let $\Gamma_n\subset \widehat{L}(n)\cap \Pi $ be
nearby simple closed planar curves in $\widehat{L}(n)$ for $n$ large
and let $\g_n=p_n+\frac{1}{\l _n}\G_n$ be the related simple
closed planar curves in $L(n)$. In particular, the sequence of simple
closed curves $\gamma_n \subset L(n)$ near $p_n$ have lengths converging
to zero as $n\to \infty $, and when viewed to be sets, these curves converge
to the origin $\vec{0}$. Note that by the convex hull property, the curves
$\g_n$ are not homologous to zero in $L_1(\ve)$.

As $L_1(\ve )$ has genus zero and the $\g_n$
are not homologous to zero, then $\g _n$ separates $L_1(\ve )$
into two subdomains. Since $\{ 0<x_3\leq \ve \} $ is simply connected and
$L_1(\ve )$ is properly embedded in $\{ 0<x_3\leq \ve \} $, then $L_1(\ve )$ separates
$\{ 0<x_3\leq \ve \} $ into two components. Let $X_n$
be the closure of a component of $\{ 0<x_3\leq \ve \} -L_1(\ve )$ in which
$\gamma_n$ fails to bound a disk; after extracting a subsequence, we
can assume that $X=X_n$ does not depend upon $n$. Our previous
arguments using $L_1(\ve )$ as a barrier (see e.g., the proof of Assertion~\ref{ass3.7})
imply that $\gamma_n$ is contained in the
boundary of a connected, area-minimizing, noncompact, orientable, properly
embedded minimal surface $\Sigma_n\subset X$
(possibly incomplete, as $X$ might not be complete) with the remainder
of its boundary contained in $\partial L_1(\ve)$.
As $L_1(\ve )$ has no singularities in
$\{ 0<x_3\leq \ve \} $, then $\Sigma_n$ is a surface
with boundary in $\partial L_1(\ve)\cup\g_n$
and it is complete in $\R^3$ outside of the closed countable set
$\{ x_3=0\} \cap {\cal S}$.
Since $\Sigma_n$ is stable, then
Corollary~\ref{corrs} insures that $\Sigma _n$ extends to a complete,
orientable, stable minimal surface
$\overline{\Sigma}_n\subset \{ 0\leq x_3\leq \ve \} )$
with boundary in $\partial L_1(\ve)\cup\g_n$.
By the maximum principle, $\overline{\Sigma}_n$ is
disjoint from $P=\{ x_3=0\} $ and so
 $\overline{\Sigma}_n= \S_n$. By curvature estimates
 for stable minimal surfaces, the second
 fundamental form of $\S_n$ is bounded in a fixed
 sized neighborhood $V_n(P)$ of $P$ (size depending on $n$),
 which implies that  $\S_n$ is proper in $\rth$ (properness of $\Sigma _n$
 follows from an application of Lemma~1.4 in~\cite{mr8}
 to each component of the intersection of $\Sigma _n$ with a sufficiently small
 fixed size neighborhood of $P$ contained in $V_n(P)$).
 Therefore, $\S_n$ is a parabolic surface
 by Theorem~3.1 in \cite{ckmr1}.

Now  fix a point $p_0\in \partial L_1(\ve)\cap \{ x_3=\ve \} $. The curve $\g_n$
separates $L_1(\ve)$ into two components and let $L_1(n,\ve)$ be the component
containing $p_0$.  Note that for some regular value $\eta\in (0,\ve)$ of $x_3|_{L_1(\ve)}$
so that $\ve -\eta $ is sufficiently small,
the component $L_1(n,\ve,\eta)$ of $L_1(n,\ve)\cap \{x_3\leq \eta\}$ that contains
$\g_n$ must contain a boundary component $\partial \subset L_1(n,\ve)\cap \{ x_3=\eta\}$
intrinsically close to $p_0$ and $\partial$ does not depend on $n$.

Suppose for the moment that
\begin{equation} \label{separation}
x_3|_{L_1(n,\ve,\eta)}\geq \min(x_3|_{\g_n}).
\end{equation}
\noindent Under the above hypothesis, $L_1(n,\ve,\eta)$
is properly embedded in $\rth$ and then
Theorem~3.1 in~\cite{ckmr1} implies that  $L_1(n,\ve,\eta )$ is a parabolic surface.
Since  $L_1(n,\ve,\eta)$ is parabolic,
the arguments in the proof of Claim~4.19  in~\cite{mpr14} show that the scalar flux
of the intrinsic gradient $\nabla x_3$ of $x_3$ on $L_1(n,\ve,\eta)$ across
$\g_n\subset \partial L_1(n,\ve,\eta)$, given by
\[
F(\nabla {x_3},\g_n)=\int _{\g_n}\langle \nabla{x_3},\nu \rangle
 \]
 where $\nu$ is the inward pointing conormal to $L_1(n,\ve,\eta)$ along
its boundary, is bounded from below by the positive number
\[
-F(\nabla {x_3},\partial)=-\int _{\partial }\langle \nabla{x_3},\nu \rangle .
 \]
 But this conclusion is impossible
since the lengths of $\g_n$ are converging  to zero
as $n\to \infty$.  Thus to find the desired contradiction in the case that (E2-A)
holds with the limit object of $\{ \widehat{L}(n)\} _n$ being a catenoid,
it remains to show that \eqref{separation} holds for $n$ sufficiently large.

Next note that as the stable minimal surface $\S_n$ is parabolic, then
\begin{equation} \label{separation2}
x_3|_{\S_n}\geq \min(x_3|{\g_n)}.
\end{equation}
Inequality~(\ref{separation2})
implies that the ends of the catenoid $\widehat{L}(\infty )$ are horizontal.
Also, since $\S_n$ is parabolic and the scalar flux
of the intrinsic gradient $\nabla ^{\S_n} x_3$ across $\g_n\subset \partial \S_n$ is
converging to zero as $n\to\infty$, then similar
reasoning as in the previous paragraph implies
that the  scalar flux
of $\nabla ^{\S_n} x_3$ across $ \S_n\cap \{x_3=\eta\}$
is converging to zero as $n\to \infty$.
By curvature estimates for the stable minimal surface $\S_n$,
we conclude:
\par
\vspace{.2cm}
\noindent
(G0) The spherical image of the Gauss map of $\S_n$
along $\S_n\cap \{x_3=\eta\}$ is contained in arbitrarily small
neighborhoods of the north or south
pole of $\esf^2(1)$ for $n$ sufficiently large.
\par
\vspace{.2cm}
Arguing similarly with fluxes also we deduce that there
exists a positive  constant $C$ depending
only on curvature estimates for stable minimal surfaces such that for any point
of $\S_n$ of intrinsic distance greater than $C\cdot
$Length$(\g_n)$ from $\partial \Sigma _n$,
the normal line to $\S_n$ must make an angle of less than
$\frac{\pi}{4}$ with the horizontal.

Let $\widehat{X}(\infty )$ be the nonsimply connected component of
$\R^3-\widehat{L}(\infty )$, and let $X(n)\subset \overline{\B }(p_n,\ve _n)$ be
the related solid annular regions defined by the condition
\[
\l _n[X(n)-p_n] \mbox{ converges to } \widehat{X}(\infty ) \mbox{ as }n\to \infty .
\]
By letting $\ve_n$ converge to zero sufficiently quickly and
after replacing by a subsequence,  we can also
assume that the domains $\l _n[\partial X(n)-p_n] \cap \B(n)$ are annuli that can be expressed as
normal graphs over their projections to $\widehat{L}(\infty )$ with the $C^1$-norm of
the graphing functions less than $\frac1n$ and
so that $\l_n\ve_n =n$.

Curvature estimates for the stable minimal surface ${\Sigma
}_n$ and the flat horizontal asymptotic geometry of the catenoid $\widehat{L}(\infty )$
imply there is a constant $R>1$ such that, for $n$ sufficiently large, the components of
$\Sigma _n\cap\left[ \B (p_n,\frac{n}{2\l _n})-\B ( p_n,\frac{R}{\l _n}) \right] $
are graphs over their orthogonal projections to the $(x_1,x_2)$-plane.
Since these graphs are part of $\S_n$ which is an area-minimizing surface in  $X(n)$,
 then there is
only one such graph for $n$ sufficiently large.  Note that by picking $R$ sufficiently large, the
graph $\Sigma _n\cap\left[ \B (p_n,\frac{n}{2\l _n})-\B ( p_n,\frac{R}{\l _n}) \right] $
can be assumed to have arbitrarily small gradient.

There exists a connected compact neighborhood $U(\g _n)$ of $\g _n$
in $\S_n\cap\B (p_n,\ve_n)$ with two boundary components, $\g_n,\a_n$,
such that the component
$\Omega_n$ of ${\Sigma }_n-[U(\g_n)\cup \{ \eta\leq x_3\leq\ve\}]$ with
boundary curve  $\a_n:=\partial\Omega_n \cap \partial U(\g_n)$ satisfies:
\begin{enumerate}[(G1)]
\item $\a_n$ can be chosen so that it corresponds to the inner boundary component of the annular
graph  $\Sigma _n\cap\left[ \B (p_n,\frac{n}{2\l _n})-\B ( p_n,\frac{R}{\l _n}) \right] $.
\item After choosing $R$ sufficiently large, then the Gauss map $G_n$ of $\Omega_n$
along $\a_n$ is almost constant
and equal to the vertical normal vector of one  of the ends of $\widehat{L}(\infty )$.
\item For $n$ sufficiently large, the
intrinsic distance function $d_{\S_n}(\cdot,\partial \S_n)$ restricted to $\Omega_n$
is greater than $C\cdot $Length$(\g_n)$.  In particular,
the normal line to $\Omega_n$ must make an angle of less than
$\frac{\pi}{4}$ with the horizontal.
\end{enumerate}

 Properties (G0) and (G2) together with the stability of $\Omega _n$ imply that
 $G_n(\Omega _n)$ can be assumed to be contained in an arbitrarily small neighborhood of the north or south
 pole in $\esf^2(1)$, for $n$ sufficiently large. In particular, the orthogonal
 projection from $\Omega _n$ to $P$ is a proper submersion for $n$ large, and its restriction to
 each boundary component of $\Omega _n$ is injective.  In this situation, one can
apply Lemma~1.4 in~\cite{mr8} to deduce that for $n$ sufficiently
large, $\Omega_n$ is a graph over its projection to $P$ whose gradient has norm at most 1.

We next prove that for $n$ sufficiently large, the curve $\a_n$ is the boundary of a small
almost-horizontal disk $T_n \subset \B (p_n,\ve_n)$ that
intersects $L_1(n,\ve,\eta)$ only along $\g_n$. The construction of
this disk is clear from Figure~\ref{fig6} if,  for $n$ sufficiently large,
$L_1\cap \B(p_n,\frac{R}{\l_n})$ is the annulus $\partial X(n)\cap \B(p_n,\frac{R}{\l_n})$.
Otherwise since $L_1\cap \ov{\B}(p_n,\ve_n)$ is
compact, then
$L_1\cap \ov{\B}(p_n,\ve_n)$  would contain a compact component
$\Lambda_n$,  disjoint from the annulus $\partial X(n)$,
with its boundary in $\esf^2(p_n,\ve _n)$
and which intersects $ \B(p_n,\frac{R}{\l_n})$.
Then $\Lambda_n$  could
be used  as a barrier to construct a compact
stable minimal surface $\Lambda_n'$  in $\B(p_n,\ve_n)-L_1$ with boundary
in $\esf^2(p_n,\ve _n)$
that intersects $\B(p_n,\frac{R}{\l_n})$. By fixing $R$ and letting $n\to \infty$,
the stable minimal surfaces  $\Lambda_n'$ can then be used to prove that the limit catenoid
$\widehat{L}(\infty )$ lies on one side of a complete stable minimal surface that is a plane, which is impossible.
 This shows that $T_n$ exists.
\begin{figure}
\begin{center}
\includegraphics[width=11cm]{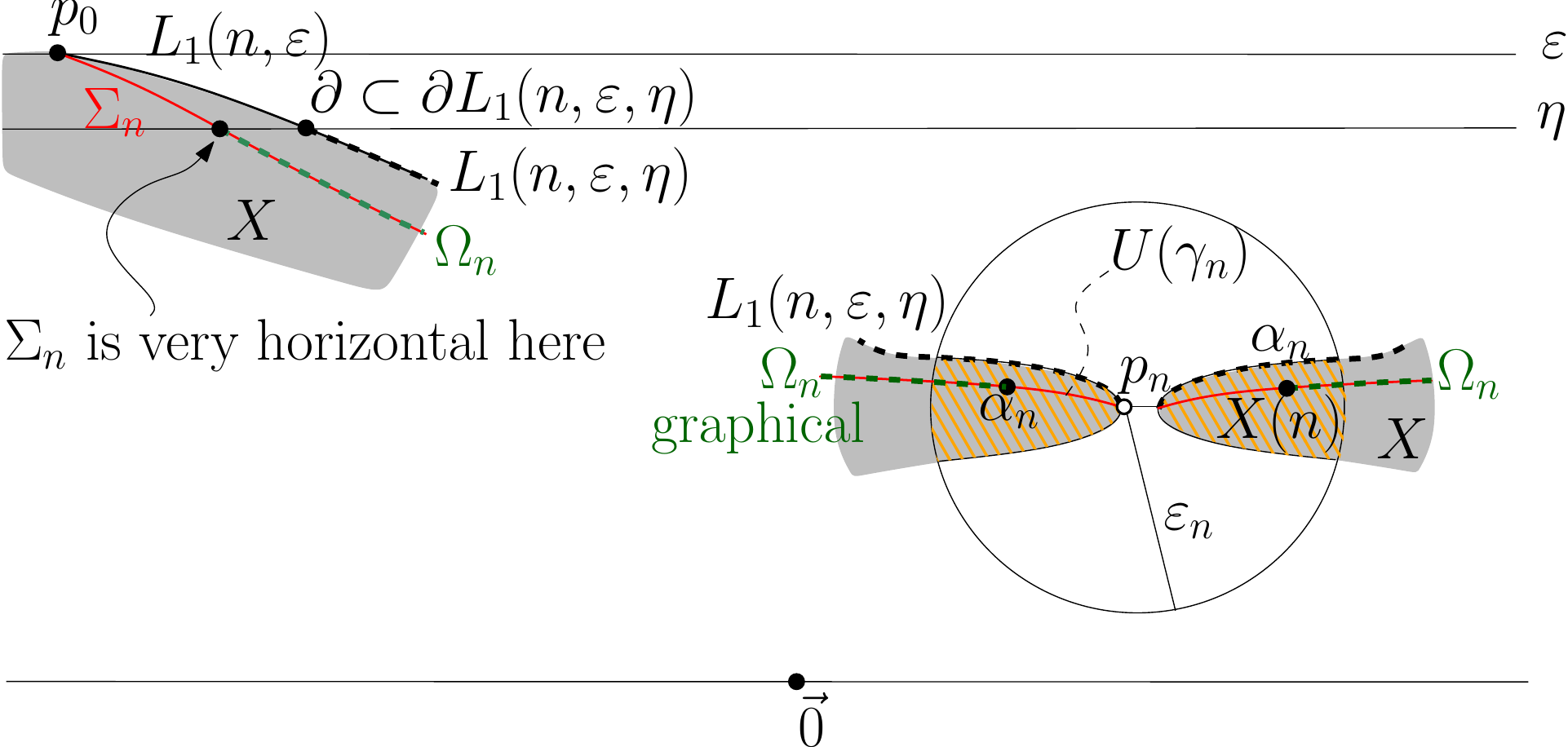}
\caption{The graphical piece $\Omega _n$ inside the stable surface $\Sigma _n$ together with
a disk $T_n$ (not represented in the figure) produce a piecewise smooth graph $Y_n$
that separates the slab $\{ 0\leq x_3\leq \eta \} $, leaving the surface $L_1(n,\ve ,\eta )$ above
$Y_n$. }
\label{fig6}
\end{center}
\end{figure}

The union $Y_n$ of $\Omega_n$ with $T_n$  is a graph over
its projection to $P$ and $Y_n$ separates the slab
$\{0\leq x_3\leq \eta\}$ into two components,
one of whose closures contains the surface $L_1(n,\ve,\eta)$.
Since a subsequence of the graphs $\Omega_n$ converges to a minimal graph
$G$ contained in
$\{0\leq x_3\leq \ve-\eta\}$ that has $\vec{0}$ in its closure,
then $G\cup \{\vec{0}\}$ is a smooth minimal surface, possibly with boundary.
By the maximum principle
for minimal surfaces, it now follows that the  graphs
$Y_n$ converge as $n\to \infty$ to the entire plane $P$.
Hence, for $n$ sufficiently large, the boundary component
$\partial$ of $L_1(n,\ve,\eta)$ must lie in the region above the
graph $Y_n$. This implies that the surface
 $L_1(n,\ve,\eta)$ is properly embedded in $\rth$ since it lies above the proper graph
  $Y_n$ (recall that $L_1$ is proper in $\{ 0<x_3\leq \ve \} $)
 Hence by Theorem~3.1 in~\cite{ckmr1}, $L_1(n,\ve,\eta)$ is a parabolic surface
 and  $x_3|_{ L_1(n,\ve,\eta)}\geq \min({x_3|_{\g_n}})$.
This completes the proof that \eqref{separation} holds.

The above arguments show that the limit object $\widehat{L}(\infty )$ of the
surfaces $\widehat{L}(n)$ defined in (F3.3)
is not a catenoid. Thus, either the limit object is a Riemann
minimal example or a minimal parking garage structure of $\rth$ with two
oppositely oriented columns. In either of these cases, there
exist homotopically nontrivial closed curves $\tau_n\subset L(n)$
converging to $\vec{0}$ with lengths converging to zero
that play the role of the waist curves $\g_n$ of the forming
catenoids in the previously considered case.
In the case that the limit object is a Riemann minimal example, then
the $\tau_n$ correspond to a circle  of the limit Riemann minimal
example  and in the case that the limit object is
a minimal parking garage structure, then the $\tau_n$  correspond to ``connection loops''
as described in item~(B) of Proposition~4.20 in~\cite{mpr14}.
Using the closed curves $\tau_n$ in place of the curves $\g_n$, the arguments
in the  case where $\widehat{L}(\infty )$ was a catenoid
can be adapted in a straightforward manner to obtain a contradiction.
Thus, Case (E2-A) does not occur at any isolated point in ${\cal S}'$.
\par
\vspace{.2cm}
 \noindent
(E2-B) {\sc Suppose that there is a constant $c>0$ such that
$I_{L_1}\geq c\ |\cdot|$ in $L_1\cap \B (\de )$.}
\newline
The contradiction in this case will be found after the application of the already proven
parts of Theorem~\ref{tttt} to an appropriate sequence of homothetic expansions of $L_1(\ve )$
from the origin.
Since ${\cal S}'\cap \B (2\de )=\{ \vec{0}\} $, we can
apply Theorem~\ref{tt2} to the minimal lamination
\[
\left( [\overline{L_1(\ve )}\cup P]-\{ \vec{0}\} \right) \cap \B (\de )
\]
of $\B (\de )-\{ \vec{0}\} $ to conclude that there exists a sequence of points
$\{p_n \}_n \subset L_1(\ve )$ converging to $\vec{0}$ such that
$|K_{L_1}| (p_n)|p_n|^2\geq n$ for all $n$. Consider the sequence of
 embedded minimal surfaces with boundary
\[
\widehat{L}'(n) = \frac{1}{|p_n|} [L_1(\ve)\cap \B (\de )],
\]
all of which have genus zero. Note that the boundary of $\wh{L}'(n)$
lies in the horizontal plane at height $\ve /|p_n|\to \infty $ and that
by our hypothesis in (E2-B), the injectivity radius function
$I_{\widehat{L}'(n)}$ of $\widehat{L}'(n)$ satisfies
\begin{equation}
\label{eq:inj}
I_{\widehat{L}'(n)}(x)=\frac{1}{|p_n|}I_{L_1}\left( |p_n|x\right) \geq c\, |x|, \qquad \mbox{ for all }x\in
\widehat{L}'(n).
\end{equation}
Given $p\in \R^3$ and $\tau > 0$,
consider the conical region
\begin{equation}
\label{eq:cone}
C^+(p,\tau ) = \{ (x_1 ,x_2 ,x_3 ) \ | \ (x_1 - x_1 (p))^2 + (x_2 - x_2(p))^2 < \tau ^{-2} (x_3 - x_3(p))^2 \}
\end{equation}
with vertex $p$ and opening angle $\a $ with respect to the positive $x_3$-axis such that $\cot \a =\tau $.
A consequence of the scale invariant lower bound (\ref{eq:inj}) for $I_{\wh{L}'(n)}$
together with the intrinsic version of the one-sided curvature estimate by Colding-Minicozzi
(Corollary~0.8 in~\cite{cm35}) is that
\begin{enumerate}[(H)]
\item Given $a>0$ small, there exists $\tau >0$ such that for every $r>0$,
$\wh{L}'(n)-[\B (r)\cup C^+(\vec{0}, \tau )]$
has bounded
Gaussian curvature on compact sets, with the bound independent of $n$, and for $n$ sufficiently large, the
components in this set consist of graphs and multivalued graphs
over their projections to the $(x_1,x_2)$-plane $P$, with the norms of the gradients
of the graphing functions being less than $a$.

\end{enumerate}
In particular, the points $\frac{p_n}{|p_n|}$
lie in $\esf^2(1)\cap C^+(\vec{0},\tau )$ for $n$ sufficiently large.

Another consequence of (\ref{eq:inj}) is that the
sequence $\{ \widehat{L}'(n)\}_n$ has locally positive
injectivity radius in the open set $B=\R^3-\{ \vec{0}\} $. Also observe that
$\vec{0}$ lies in the closure of $\widehat{L}'(n)$ for all $n$.
Applying the already proven conclusions of Theorem~\ref{tttt}
before the list of items $1,\ldots ,7$ to the closed countable set
$W=\{ \vec{0}\} $ and to the sequence of minimal surfaces
$\{ \widehat{L}'(n)\} _n$, we conclude that there exists a
(possibly empty) closed subset ${\cal S}^B$ of $\R^3-\{ \vec{0}\} $
 and a (regular) minimal lamination $\widehat{\cal L}$ of
 $\R^3-[\{ \vec{0}\} \cup {\cal S}^B]$ such that:
\begin{enumerate}[(H1)]
\item After passing to a subsequence (denoted in the same way),
$\{ \widehat{L}'(n)\} _n$ converges $C^{\a}, \, \a\in (0,1)$, to $\widehat{\cal L}$ in
$\R^3-[\{ \vec{0}\} \cup {\cal S}^B\cup S(\widehat{\cal L})]$, where
$S(\widehat{\cal L})\subset \widehat{\cal L}$ is
the singular set of convergence of the $\widehat{L}'(n)$ to
$\widehat{\cal L}$.
\item The closure of $\widehat{\cal L}$ relative to $B$
has the structure of a possibly singular minimal lamination of $B$
with related regular lamination $\widehat{\cal L}$ and singular set
${\cal S}^{B}$.
\item The closure $\overline{\widehat{\cal L}}$ of $\widehat{\cal L}$ in $\R^3$ has the structure of a
possibly singular minimal lamination of $\R^3$, whose singular set
$\widehat{\cal S}$ satisfies $\widehat{\cal S}\subset {\cal S}^B\cup
\{ \vec{0}\} $.
\end{enumerate}
Clearly, $\widehat{\cal L},{\cal S}^B$ and $\widehat{\cal S}$ are contained
in $\{ x_3\geq 0\} $. Since by construction the curvatures of the surfaces
$\widehat{L}'(n)$ are unbounded on $\esf ^2(1)$,
then ${\cal S}^B \cup S(\widehat{\cal L})\neq \mbox{\rm \O}$.

By property (H) above, we deduce that ${\cal S}^B\cup S(\widehat{\cal L})$
lies in $C^+(\vec{0},\tau )$.
As the absolute Gaussian curvature of the $\widehat{L}'(n)$ at
$\frac{p_n}{|p_n|}$ is at least $n$, we also deduce that ${\cal S}^B\cup S(\widehat{\cal L})$
intersects $\esf^2(1)$ at some point $x_0\in C^+(\vec{0},\tau )$.

\begin{claim}
\label{claim3.9}
 $\overline{\widehat{\cal L}}$ is the foliation of the closed upper
halfspace of $\rth $ by horizontal planes
(in particular, $\widehat{\cal S}={\cal S}^B=\mbox{\rm \O}$),
and $S(\widehat{\cal L})$ is the positive $x_3$-axis.
\end{claim}
\begin{proof}
Observe that we cannot apply item~7 of
Theorem~\ref{tttt} since it has not
been proved yet. Instead, we argue as follows.
Suppose for the moment that
$\overline{\widehat{\cal L}}$ does not contain nonflat leaves.
Then, the leaves of $\overline{\widehat{\cal L}}$ are horizontal planes and
$\widehat{\cal S}={\cal S}^B=\mbox{\O }$.
Since $x_0\in [{\cal S}^B\cup S(\widehat{\cal L})]\cap\esf^2(1)$
  and the sequence $\{ \widehat{L}'(n)\} _n\subset \{ x_3>0\} $
  is locally simply connected in
$\R^3-\{ \vec{0}\} $, then it follows from Corollary~0.8 in~\cite{cm35}
and from Meeks' $C^{1,1}$-regularity
theorem~\cite{me25} that $S(\widehat{\cal L})=\{ (0,0,x_3)\ | \ x_3 > 0\} $.
Thus, in order to finish the proof of Claim~\ref{claim3.9} we will
suppose that $\overline{\widehat{\cal L}}$ contains a nonflat leaf $L'$
and we will find a contradiction.

By definition of leaf of a singular lamination, we can decompose
$L'=L_1'\cup {\cal S}_{L_1'}$ where
$L_1'$ is a leaf of the regular lamination $\widehat{\cal L}$ and
${\cal S}_{L_1'}$ is the set of singular leaf points of $L_1'$.
Note that $L_1'$ is not flat, and so the convergence of
portions of the $\widehat{L}'(n)$ to $L_1'$ is of multiplicity
one (see item~5 of Theorem~\ref{tttt}).
As the $\widehat{L}'(n)$ have genus zero, then the same
holds for $L_1'$.

Since $x_0\in {\cal S}^B\cup S(\widehat{\cal L})$, then item~4 of Theorem~\ref{tttt}
implies that there passes a plane $P_{x_0}$ through $x_0$, such that $P_{x_0}$
is a leaf of $\overline{\widehat{\cal L}}$. Recall that we have also proven item~6
of Theorem~\ref{tttt} except for the property
of $L_1'(\ve')=L_1'\cap \{ 0<x_3\leq \ve'\} $ having infinite
genus for every $\ve'>0$ when $\ov{L}_1'$
has a singularity in the $(x_1,x_2)$-plane.
Consider the closure $\overline{L'}$ of $L'$ in $\R^3$, which has the
structure of a possibly singular minimal lamination.
As by construction $L'$ is contained in $\{ x_3\geq 0\} $, then the
already proven part of item~6 of Theorem~\ref{tttt} applies to $L'$  and gives that
$\overline{L'}$ is contained in a closed slab or halfspace of $\R^3$, which must be
contained in $\{ x_3\geq 0\}$. Therefore, there
exists a collection ${\cal P}(L')\subset \overline{L'}$
consisting of one or two planes contained in $\{ x_3\geq 0\} $,
such that $\overline{L'}=L'\cup {\cal P}(L')$, $L'$ is proper in a component
$C(L')$ of $\R^3-{\cal P}(L')$ and $C(L')\cap \overline{\widehat{\cal L}}=L'$.

Let Sing$(\overline{L'})$ be the set of singularities of $\overline{L'}$.
Clearly, Sing$(\overline{L'})\subset \widehat{\cal S}\subset {\cal S}^B\cup \{ \vec{0}\} $.
Note that the following properties hold.
\begin{enumerate}[(J1)]
\item Through every point $y\in \overline{L'}\cap [{\cal S}^B\cup \{ \vec{0}\}] $
there passes a plane in $\overline{L'}$: this follows from item~4 of Theorem~\ref{tttt} if
$y\neq \vec{0}$. In the case that $y=\vec{0}$, then $\vec{0}$ cannot be a singular leaf point
of $L'_1$ by Corollary~\ref{corol2.3}; hence $L'_1$ limits to $P$ and we can take
$P$ as the desired plane in $\overline{L'}$ passing through $y$.

\item
Let $P'$ be a horizontal plane contained in $\overline{L'}$ such that $x_3(P')>0$. Then,
$P'\cap \left[ {\cal S}^B\cup \mbox{Sing}(\overline{L'})\right] $ $=P'\cap {\cal S}^B$
is a discrete set by the locally simply connected
property of $\widehat{L}'(n)$ in $\{ x_3>0\} $, and hence it is a finite set (since
$P'\cap {\cal S}^B$ lies in $C^+(\vec{0},\tau )$). Actually,
$P'\cap {\cal S}^B$ consists of at most two points:
this follows from 
a straightforward adaptation of
the connecting loop argument in the proof of Lemma~3.3 in~\cite{mpr14}
using that $L_1'$ has genus zero.
\item There are no singularities of $\overline{L'}$ in $\{ x_3=0\} -\{
\vec{0}\} $, by property (H).
\end{enumerate}

From properties (J1), (J2) and (J3) we conclude that:
\begin{description}
\item[{\rm (J4)}] Sing$(\overline{L'})\cap C(L')=\mbox{\O }$,
because  there are no planes of $\overline{L'}$ in $C(L')$.
\item[{\rm (J5)}] Either $P=\{ x_3=0\} $ is the lower boundary plane
of $C(L')$, in which case Sing$(\overline{L'})$
consists of at most three points, or $P$ lies strictly below the planes
in ${\cal P}(L')$, in which
case Sing$(\overline{L'})$ consists of at most four points.
\end{description}

Let $P'$ be a plane
in ${\cal P}(L')$ with positive height, which
exists because $P_{x_0}$ exists.
The final contradiction that will finish the proof of
Claim~\ref{claim3.9} will be a consequence
of the following three contradictory
properties:
\begin{enumerate}[(K1)]
\item $\overline{L'}$ has at most one singularity on~$P'$.

\item $\overline{L'}$ cannot have exactly one singularity on~$P'$.

\item $\overline{L'}$ has at least one singularity on $P'$.
\end{enumerate}

We now prove (K1) by contradiction: assume
that $\overline{L'}$ has two singularities occurring at distinct points
$q_1,q_2$ in a plane $P'\in {\cal P}(L')$ at positive height. As the injectivity radius functions of
the surfaces
$ \widehat{L}'(n)$ are uniformly bounded away from zero nearby $q_1,q_2$ by (\ref{eq:inj}),
then the description in (D2) gives that $L'$ has the appearance around the point
$q_i$, $i=1,2$, of a disk with the geometry of a
spiraling double staircase that limits from above or below to a disk in $P'$
centered at the singularity $q_i$.
The same type of connecting loop argument mentioned in (J2) above
together with the fact that the surfaces $\widehat{L}'(n)$ have genus zero imply that the
spiraling double staircases at $q_1,q_2$ are oppositely handed.
Furthermore, we can modify slightly the flux-type arguments in the proof of
Proposition~4.18 in~\cite{mpr14} to find a contradiction in this case; roughly
speaking, these arguments
use connecting loops $\G '_n\subset \widehat{L}'(n)$ that converge as
$n\to \infty $ to the twice covered horizontal segment that joins
$q_1$ to $q_2$, and lead to a contradiction to
the fact that the absolute value of the scalar flux of $\nabla x_3$
along the curves $\G'_n$ tends to zero as
$n\to \infty $. 
This proves  (K1).

To prove property (K2), suppose that $\overline{L'}$ has exactly one singularity $p$
at a plane $P'\in {\cal P}(L')$ at positive height. Suppose for the moment that $P'$
is the lower boundary plane of $C(L')$. Observe that as $x_3(P')>0$, then
there exists a positive constant $c'$ such that for every point $x=(x_1,x_2,x_3)\in
\overline{C(L')}$,  it holds
$|x|\geq c'\, |x-p|$ (observe that this inequality uses that the distance from $x$ to the
origin  is greater than some positive number).
This inequality and~(\ref{eq:inj}) imply that
if $x\in \widehat{L}'(n)\cap \{ x_3\geq x_3(P')\} $, then
\begin{equation}
\label{eq:inj1}
I_{\widehat{L}'(n)}(x) \geq c_1\, |x-p|,
\end{equation}
for $c_1=c\cdot c'>0$.

Given $\tau > 0$, consider a sufficiently shallow conical
region $C^+(p, \tau )$ defined in (\ref{eq:cone})
so that the singularities of $\overline{L'}$ above $P'$,
if they exist, lie in $C^+(p,2\tau)$ (recall we have at most two
of these singularities, both at the same horizontal plane $P''$
strictly above $P'$). It is straightforward to deduce that there
exists $a=a(\tau )\in (0,1)$ such that
\begin{equation}
\label{dist}
\mbox{dist}_{\R^3}\left( x, \mbox{Sing}(\overline{L'})\right) \geq a\ | x-p|
\quad \mbox{ for all $x\in x_3^{-1}([x_3(P'),x_3(P'')])-C^+(p,\tau )$,}
\end{equation}
where by convention $x_3(P'')=+\infty $ if $P''$ does not exist. Clearly, we can assume $c_1\in (0,a)$
(because inequality (\ref{eq:inj1}) stays true if we choose a smaller positive constant $c_1$).

We claim that the injectivity radius function of $L_1'$ admits a scale invariant bound of the form
\begin{equation}
\label{eq:inj2}
\frac{I_{L_1'}(x)}{|x-p|}\geq c_2>0 \qquad \mbox{in $L_1'\cap [\R^3-C^+(p,\tau )]$,}
\end{equation}
for some $c_2>0$. The proof of (\ref{eq:inj2}) is by contradiction: suppose on the contrary that
$\displaystyle \frac{I_{L_1'}(x_k)}{|x_k-p|}\to 0$ as $k\to \infty $ for some sequence of points $x_k\in
L_1'\cap [\R^3-C^+(p,\tau )]$. For $k\in \N$ fixed, take a sequence of points $x_{k,n}\in \wh{L}'(n)
\cap \{ x_3> x_3(P')\} $ converging to $x_k$ as $n\to \infty $ (recall that $x_k$ lies in the
regular part $L_1'$ of $L'$). Then,
\begin{equation}
\label{eq:inj3}
I_{\wh{L}'(n)}(x_{k,n})\stackrel{(\ref{eq:inj1})}{\geq }c_1\, |x_{k,n}-p|\stackrel{(\star )}
{>}\frac{c_1}{2}\, |x_k-p|,
\end{equation}
where $(\star )$ holds for $n$ sufficiently large. Inequality (\ref{eq:inj3})
ensures that the intrinsic metric ball $B_{\wh{L}'(n)}(x_{k,n},\frac{c_1}{2}\, |x_k-p|)$ is a
geodesic disk in $\wh{L}'(n)$, i.e., the exponential map in $\wh{L}'(n)$ with base point $x_{k,n}$
restricts as a diffeomorphism to the disk of radius $\frac{c_1}{2}\, |x_k-p|$
centered at the origin in the tangent plane $T_{x_{k,n}}\wh{L}'(n)$. As $n\to
\infty $, the disks $B_{\wh{L}'(n)}(x_{k,n},\frac{c_1}{2}\, |x_k-p|)$ converge smoothly
(with multiplicity one) to
the intrinsic metric ball $B_{L_1'}(x_k,\frac{c_1}{2}\, |x_k-p|)$,
because both surfaces $B_{L_1'}(x_k,\frac{c_1}{2}\, |x_k-p|)$,
$B_{\wh{L}'(n)}(x_{k,n},\frac{c_1}{2}\, |x_k-p|)$ are contained for $n$ sufficiently large
in the extrinsic ball $\B (x_k,c_1\, |x_k-p|)$, and this extrinsic ball does not contain
points of Sing$(\overline{L'})$ by (\ref{dist}) (recall that $0<c_1<a$).
The continuity of the injectivity radius function under smooth limits
(Erlich~\cite{ehr1} and Sakai~\cite{sa2}) implies that
$I_{L_1'}(x_k)\geq  \frac{c_1}{2}\, |x_k-p|$, which contradicts our hypothesis.
This proves our claimed inequality (\ref{eq:inj2}) for some $c_2>0$.

Observe that as the surfaces $\wh{L}'(n)$ are
locally simply connected away from $\{ \vec{0}\} $,
then the description (D)-(D2) applies. In particular, $L_1'$ contains a main component locally
around $p$, whose intersection
with $x_3^{-1}(x_3(P'),x_3(P''))-C^+(p,\tau _1)$ is a pair of $\infty $-valued
graphs over its projection to the punctured plane $P'-\{ p\} $, for some $\tau _1\in (0,\tau)$.
These two $\infty $-valued graphs can be connected by curves of uniformly bounded length arbitrarily close to $p$,
by the local description (D2). The scale invariant lower bound (\ref{eq:inj2})
for $I_{L_1'}$ is sufficient to apply
the arguments in page 45 of Colding and Minicozzi~\cite{cm25}.
The existence of such $\infty $-valued graphs over the punctured plane $P'-\{ p\} $ contradicts
Corollary~1.2 in~\cite{cm26}, thereby finishing the proof of property (K2) in the particular
case that $P'$ is the lower boundary plane of $C(L')$. If $P'$ is the upper boundary plane of
$C(L')$ and the lower boundary plane of $C(L')$ is not $P=\{ x_3=0\} $,
then the same reasoning holds after reflecting in $P'$ (because $|x|\geq c'\, |x-p|$ still holds in this case
for some $c'>0$ and for all $x\in C(L')$).
Finally, if $P'$ is the upper boundary
plane of $C(L)$ and its lower boundary plane is $P$, then $|x|\geq c'\, |x-p|$ holds for
each point in $\overline{C(L')}-C^+(p,\tau )$ where $\de $ is chosen so
that $\vec{0}\in C^+(p,2\tau)$
(in other words, with $\vec{0}$ playing the role of the singularities of
$\overline{L'}$ above $P'$ of the previous case); therefore, in this
last case the arguments above also hold after reflecting in $P'$. This finishes the proof of property (K2).

To demonstrate property (K3), suppose that $\overline{L'}$ has no singularities in $P'$.
Then,  the curvature of $L'$ in a neighborhood
of $P'$ is bounded (because the singular set of
$\overline{L'}$ is disjoint from $P'$ and the injectivity radius function of $L'$
is bounded away from zero in a neighborhood of $P'$, as follows from (\ref{eq:inj})
and from the continuity of the injectivity radius function
under smooth limits~\cite{ehr1,sa2}, so the one-sided
curvature estimates of Colding and Minicozzi apply).
In this situation, Lemma~1.4 in~\cite{mr8} together with the proof of the
halfspace theorem produce a contradiction, which proves (K3).  Now Claim~\ref{claim3.9}
follows.
\end{proof}

\begin{claim}
\label{claim3.11}
  Given $\{ \l _n\} _n\subset (1,\infty )$ with $\l _n\to \infty $ as $n\to \infty $, the sequence
$\{ \l _nL_1(\ve ) \} _n$ converges  to the foliation of
the closed upper halfspace of $\R^3-\{ \vec{0}\} $ by horizontal
planes, and the singular set of convergence
of $\{ \l _nL_1(\ve ) \} _n$ is the positive $x_3$-axis.
\end{claim}
\begin{proof}
Assume that $\tau,\de>0$ are chosen sufficiently small so that the tangent
planes to points of $L_1$ make an angle of less than $\frac{\pi}{4}$
with the horizontal
at points of $L_1\cap [\B(\de)-C^+(\vec{0}, \tau )]$; this is possible by the intrinsic
one-sided curvature estimates in~\cite{cm35} and our assumption (E2-B).
Furthermore, these one-sided curvature estimates and the fact that $L_1(\ve)$ is proper in the
$\{0<x_3\leq\ve\}$ imply that for any  $r\in (0,\frac{\de}{2})$ fixed
and for $\be>0$ sufficiently small, each point in the set
\[
\a _r(\be ):=[L_1(\ve )\cap\{(x_1,x_2,x_3)\mid x_1^2+x_2^2=r^2, \, x_3\leq \be r\}]- C^+(\vec{0}, \tau )
\]
lies on a component of $\a _r(\be )$
that is either a graph over the circle  $C(r)=P\cap \{x_1^2+x_2^2=r^2\}$
or an infinite spiraling  ``graphical'' arc,  limiting from above
to $C(r)$. Recall
that we chose points $p_n\in L_1(\ve)$ that converge to
$\vec{0}$ and that are blow-up points on the scale of curvature,
around which one has a Colding-Minicozzi
picture with a pair of multigraphs that extends sideways by Claim~\ref{claim3.9}, for $n$ large,
all the way to the cylinder $\{x_1^2+x_2^2=r^2\}$ for any $r\in (0,\frac{\de}{2})$ fixed.
Therefore, we
deduce that for $\de>0$ sufficiently small and for any $r\in (0,\frac{\de}{2})$ fixed,
we have that
\[
\a _{r}:=(L_1(\ve )\cap\{x_1^2+x_2^2=r^2\})- C^+(\vec{0}, \tau )
\]
contains  a pair of spiraling ``graphical'' arcs, each one limiting from above
to the circle $C(r)$.

 We first show that for $r>0$ sufficiently small,
every component of $L_1(\ve )\cap \B (2r)$ is
simply connected.
Otherwise for every $m\in \N$ sufficiently large, there exists a
simple closed curve $\G_m \subset L_1(\ve )\cap \B(1/m)$
that separates $L_1(\ve )$,  $\G_m$
is homologically nontrivial at one of the sides of $L_1(\ve )$
 and $\G_m$ bounds a complete, noncompact,
 embedded, stable minimal surface
$F_m\subset \{0<x_3\leq \ve\}$, which is properly embedded
in the closure of a component of $ \{0<x_3\leq \ve\}-L_1(\ve )$,
such that
$\G_m\subset \partial F_m\subset [\{x_3=\ve\} \cap L_1(\ve )]\cup \G_m$.
By curvature estimates for stable minimal surfaces,
for $m$ sufficiently large,
 $F_m$
intersects $\{ x_1^2+x_2^2=r^2\} $
in at least one almost-horizontal
circle\footnote{
By curvature estimates, $F_m$ is very flat
nearby $\{ x_1^2+x_2^2=r\} $ for $m$ sufficiently large; as $\G_m$ is arbitrarily close to
height zero but $F_m$ is contained in $\{ x_3>0\} $, then
$F_m$ is very horizontal nearby $\{ x_1^2+x_2^2=r\} $; hence either
$F_m\cap \{ x_1^2+x_2^2=r\} $ contains an almost-horizontal circle, or it contains an
almost-horizontal long spiraling arc. This last possibility can be ruled out as $F_m$ has been
constructed by a standard procedure as a limit of area-minimizing surfaces, which cannot
have this multigraph appearance.}
and these circles converge to $C(r)$ as $m\to \infty $.
However, for $m$ large $F_m$
would intersect the pair of spiraling arcs in
$\a_{r}$, giving a contradiction.  This contradiction proves that
every component of $L_1(\ve )\cap \B(2r)$ is simply connected (for $r>0$ chosen sufficiently small).

We now prove Claim~\ref{claim3.11} by contradiction.
Suppose that for some sequence $\l _n\to \infty $, a subsequence of
$\{\l _nL_1(\ve )\} _n$ does not converge to the foliation of $\{ x_3\geq 0\} -\{
\vec{0}\} $
by horizontal planes. Observe that the sequence
$\l _nL_1(\ve )$ has locally positive injectivity radius in $\R^3-\{ \vec{0}\} $,
as we are assuming $I_{L_1}\geq c\ |\cdot |$ in this case (E2-B).
Also observe that $\vec{0}$
lies in the closure of $\l _nL_1(\ve )$ for all $n$. Applying the already proven conclusions of
Theorem~\ref{tttt} before the list of items 1,...,7 to the closed countable set $W = \{ \vec{0}\} $
and to the sequence of minimal surfaces $\{ \l _nL_1(\ve )\} _n$ , we
conclude that there exists a (possibly empty) closed subset ${\cal S}'$ of $\R^3-\{ \vec{0}\} $
 and a (regular) minimal lamination $\widehat{\cal L}'$ of
 $\R^3-[\{ \vec{0}\} \cup {\cal S}']$ such that after passing to a subsequence (denoted in the same way),
$\{ \l _nL_1(\ve )\} _n$ converges $C^{\a}$, \, $\a\in (0,1)$, to $\widehat{\cal L}'$ in
$\R^3-[\{ \vec{0}\} \cup {\cal S}'\cup S(\widehat{\cal L}')]$, where
$S(\widehat{\cal L}')\subset \widehat{\cal L}'$ is
the singular set of convergence of the $\l _nL_1(\ve )$ to
$\widehat{\cal L}'$. The arguments in the proof of Claim~\ref{claim3.9} can be adapted to
demonstrate that
there are no points of ${\cal S}'\cup S(\widehat{\cal L}')$
in $\{ x_3>0\} $, and thus,
$\widehat{\cal L}'$ consists of $P-\{ \vec{0}\} $ together with a single leaf $L'$,
which is a smooth, nonflat surface that is
proper in $\{ x_3>0\} $. Furthermore, the multiplicity of the convergence of
$\{ \l _nL_1(\ve )\} $ to $L'$ is one, as $L'$ is not flat. Since every component of
$L_1(\ve )\cap \B(2r)$ is simply connected for $r>0$ sufficiently small, a standard
lifting argument and the convex hull property imply that $L'$ is simply connected.

As $L'$ does not extend through the origin, then Theorem~\ref{tt2} produces a sequence
of points $q_n\in L'$ such that $q_n\to \vec{0}$ and
$|K_{L'}|(q_n)|q_n|^2\to \infty $ as $n\to \infty $.
Without loss of generality,
we can assume that the $q_n$ are points of
almost-maximal Gaussian curvature, in the sense of Theorem~1.1
in~\cite{mpr20}: in particular, the sequence of translated and rescaled surfaces
$\mu _n(L'-q_n)$ converges (after passing to a subsequence) to
a helicoid, where $\mu _n=1/\sqrt{|K_{L'}|(q_n)}$.
This limit helicoid is vertical, since the scaled surfaces
$\frac{1}{|q_n|}L'$ converge to the foliation of $\{ x_3\geq 0\} -\{ \vec{0}\} $ by horizontal
planes, as follows from Claim~\ref{claim3.9} adapted to this situation. This implies that nearby
$q_n$, $L'$ has a point $y_n$ where the tangent plane $T_{y_n}L'$ is vertical and the intersection
$L'\cap (T_{y_n}L')\cap \B(2)$
contains an analytic arc $\be_n$ that is arbitrarily close (for $n$ sufficiently
large)
in the $C^1$ norm to the straight line segment $(T_{y_n}L')\cap \{(x_1,x_2,x_3)\mid x_3=x_3(y_n)\}
\cap \B(2)$, and so we may assume that $\be_n$ has length less than $5$.
Also, note that the
arcs $\be _n$ converge after passing to a subsequence to a straight line segment of length $4$
 passing through the origin, and that for $n$ large enough, $\wt{\be }_n:=\be _n\cap
\{ (x_1,x_2,x_3)\ | \ x_1^2+x_2^2\leq 1\} $
joins the two spiraling arcs in $L'\cap
\{ (x_1,x_2,x_3)\ | \ x_1^2+x_2^2=1\} $. After replacing by a subsequence, we will assume that
the last sentence holds for all $n\in \N$.

Consider the arc $\G \subset L'$ consisting of $\wt{\be }_1$ together
with the two infinitely spiraling arcs in $L'\cap
\{ (x_1,x_2,x_3)\ | \ x_1^2+x_2^2=1\} $ with the same end points as
$\wt{\be }_1$ and which limit to the circle
$\{ x_1^2+x_2^2=1\} \subset P$. Observe that $\G $ is a proper arc in $L'$,
and since
$L'$ is simply connected, then $\G $ separates $L'$
into two components. Let $D$ be the closure of the component of $L'-\G$ that, near its boundary,
lies in $\{ (x_1,x_2,x_3)\ | \ x_1^2+x_2^2\leq 1\}$. Note that $D$ is topologically a disk with one end
(the boundary $\partial D$ is a Jordan arc).
For every $n\geq 2$, let $D_n$ be the compact subdisk of $D$ bounded by $\wt{\be}_1\cup \wt{\be }_n$.
Thus, $\{ D_n\} _{n\geq 2}$ is a compact increasing exhaustion of $D$.
Assume for the moment that the following property (P)
holds and we will finish the proof of Claim~\ref{claim3.11} (we will prove property (P) later).
\begin{enumerate}[(P)]
\item The function $x\in D\mapsto I_{L'}(x)/|x|$ is bounded.
\end{enumerate}

We next prove that the diameter of $D$ is finite. Let $d_D(\cdot,\cdot)$ denote the Riemannian
distance function in $D$.
Arguing by contradiction, we can find sequences
of points $a_n,b_n\in D$, such that $d_D(a_n,b_n)$
becomes arbitrarily large. Since $L'-C^+(\vec{0},\tau )$ consists of multivalued graphs with bounded gradient
and the arcs $\wt{\be }_n$ have uniformly bounded lengths,
then we can assume after replacement that $a_n,b_n$ both lie in $C^+(\vec{0},\tau )$.
Since $\{ D_n\} _{n\geq 2}$ is a compact
exhaustion of $D$ and the diameter of each $D_n$ is finite, then we can assume, without
loss of generality that after choosing a subsequence, $b_n\in D-D_n$ and
$d_D(a_n,b_n)>n$.  Let $\sigma_n$ be a smooth arc in $D$ with length less than $ d_D(a_n,b_n)+1$.
We claim that the points $a_n$ can also be chosen to diverge in $D$.  Otherwise we may assume
 that after replacing by a subsequence,
for all $n\in \N$, $a_n\in D_{j_0}$ for some $j_0\in \N,$ $j_0\geq 2$.
Thus for each integer $k\in [j_0,n-1]\in \N$, there is a point $a(n,k)\in \sigma_n\cap (D_{k+1}-D_k)$,
and then,
\[
n<d_D(a_n,b_n)\leq d_D(a_n,a(n,k))+d_D(a(n,k),b_n)\leq \mbox{diameter}(D_{k+1})+d_D(a(n,k),b_n).
\]
As $n-\mbox{diameter}(D_{k+1})$ can be made arbitrarily large for some choices of $k,n$ with $n\to \infty $
and $k\in [j_0,n-1]$, then we conclude that
given $k\in \N$, $k\geq 2$, there is an $n(k)\in \N$ such that
$d_D(a(n(k),k),b_{n(k)})>k$. Clearly this shows that we may assume that the sequence $a_n$ can also be chosen to diverge in $D$.
 Hence, we now have that the sequences $\{ a_n\} _n$ and $\{ b_n\} _n$ are both diverging in $D$ and
as they lie in $C^+(\vec{0},\de )$, then in $\rth$  both sequences are converging to $\vec{0}$ as $n\to \infty $. In particular,
property (P) gives that for $n$ large, both $I_{L'}(a_n)$, $I_{L'}(b_n)$ can be taken arbitrarily small.
Since the extrinsic distance from $a_n$ to the boundary $\partial D$ is greater than some positive number
independent of $n$, then $D$ contains geodesic arcs parameterized by arc length
\[
\g _{a_n}\colon [0,L(\g _{a_n}))\to D,\quad \g _{b_n}\colon
[0,L(\g _{b_b}))\to D,
\]
starting respectively at $a_n,b_n$, with finite lengths
$L(\g _{a_n})=I_{L'}(a_n)$, $L(\g _{b_n})=I_{L'}(b_n)$ that can be both taken
arbitrarily small for $n$ large, and with limiting end points the origin.
Fix $n$ large and let $m(n)$ be an integer such that $a_n,b_n\in D(m(n))$.
As $D(m(n))$ is a compact minimal disk, then the Gauss-Bonnet formula
ensures that both $\g _{a_n},\g _{b_n}$ can be assumed to exit $D(m(n))$,
which implies that these geodesic arcs
must intersect $\wt{\be }_{m(n)}$. As the length of $\wt{\be }_{m(n)}$ is not larger than
5, then there exists a piecewise smooth path $\Delta _n\subset D$
joining $a_n,b_n$ with length not larger than 6 for $n$ large
($\Delta _n$ is a union of arcs in $\g _{a_n},\wt{\be }_{m(n)},\g _{b_n}$). This is a contradiction, because
the intrinsic distance in $D$ from $a_n$ to $b_n$
was supposed to be arbitrarily large. Therefore, the diameter of $D$ is finite.

The finiteness of the diameter of $D$ insures that we can join the two multivalued graphs in $L'\cap C^+(\vec{0},
\de )$ by curves of uniformly bounded lengths, which contradicts Corollary~1.2 in~\cite{cm26} (specifically see
the paragraph just after Corollary 1.2 in~\cite{cm26}). This contradiction finishes the proof of Claim~\ref{claim3.11},
modulo proving property (P), which we demonstrate next.

First note that $L'$ is not complete (otherwise it would be proper by~\cite{cm35} hence it could not be contained in
a halfspace by the halfspace theorem). Therefore, the injectivity radius function $I_{L'}$ is finite valued and continuous
on $L'$. Arguing by contradiction, suppose that $I_{L'}(z_n)/|z_n|$ tends to infinity at a sequence of points
$z_n\in D$. By continuity of $I_{L'}/|\cdot |$, the
points $z_n$ must leave every compact subset of $D$. This property  and the properness of $L'$ in $\{ x_3>0\} $  imply
that after passing to a subsequence, the $z_n$ converge in $\R^3$ to a point $z_{\infty }\in P$, $|z_{\infty }|\leq 1$.

Suppose that $z_{\infty }\neq \vec{0}$. Therefore, $I_{L'}(z_n)\to \infty $ as $n\to \infty $. By
Corollary~0.8 in~\cite{cm35}, we can find disks in $L'$ centered at $z_n$ of intrinsic radius arbitrarily large,
such that the second fundamental form of such disks is arbitrarily small (for $n$ sufficiently large). This is impossible,
as such disks would intersect the spiraling curves in $[L'\cap \{ x_1^2+x_2^2=1\} ]-C^+(\vec{0},\tau )$.
Therefore, $z_{\infty }=\vec{0}$. Consider the sequence of rescalings $L''_n:=\frac{1}{|z_n|}L'$. Since
\[
\frac{I_{L''_n}\left(\frac{1}{|z_n|}x\right) }{ \frac{1}{|z_n|}|x|} =\frac{I_{L'}(x)}{|x|},
\]
then $I_{L''_n}\left(\frac{1}{|z_n|}z_n\right) $ tends to infinity as $n\to \infty $.
As in the previous case, there exist
disks in $L''_n$ centered at $\frac{1}{|z_n|}z_n$ of intrinsic radius arbitrarily large,
such that the second fundamental form of such disks is arbitrarily small (for $n$ sufficiently large). This is again
impossible, because such disks would intersect the multivalued graphs in $L''_n-C^+(\vec{0},\tau )$
for $n$ large.
Now Claim~\ref{claim3.11} is proved.
\end{proof}

We next apply Claim~\ref{claim3.11} to find a contradiction in Case (E2-B), that in turn
will imply that the genus of $L_1(\ve )$ is infinite (finishing the proof of Assertion~\ref{ass3.8}).
From Claim~\ref{claim3.11} we deduce that for any
$r>0$ sufficiently small, $L_1(\ve )\cap \B(r)$ has the
appearance of a spiraling double staircase limiting
from above to the closed horizontal disk $\B(r)\cap P$ minus its center.

Since $L_1(\ve )$ is proper in $\{ 0<x_3\leq \ve \} $,
we conclude from the last paragraph that given $k$
isolated points $q_1, q_2, ..., q_k$ in the singular set ${\cal S}'
\subset \{ x_3=0\} $ of the lamination $\overline{L_1(\ve)}$, there exist
pairwise disjoint disks $D(q_1, \de_1),\ldots ,D(q_k, \de_k) \subset \{ x_3=0\} $
such that each of the vertical cylinders $\partial D(q_j, \de_j) \times (0,\ve]$
intersects $L_1(\ve)$ in two spiraling curves that limit to the
horizontal  circle
$\partial D (q_j, \de_j) \times \{0\}$. As  $L_1(\ve )$ has finite genus,
then
the proof of Lemma~3.3 in~\cite{mpr14} implies that there exist at most two of these isolated
singular points of ${\cal S}'$, which in turn implies that
${\cal S}'$ contains at most two points (since ${\cal S}'$ is closed and
countable, the subset of its isolated points is dense in ${\cal S}'$
by Baire's Theorem), and that if ${\cal S}'$ consists of exactly two points,
then we find a contradiction with the flux-type arguments along connecting loops
as in the proof of Proposition~4.18 in~\cite{mpr14} (also see the proof of property (K1)
above). Hence, ${\cal S}'$ consists of a unique point, and in this case we find a
contradiction as in the last paragraph of the proof of property (K2).
 This contradiction proves that the genus of $L_1(\ve )$
 is infinite provided that Case (E2-B) holds,
thereby finishing the proof of Assertion~\ref{ass3.8}, which in turn demonstrates
Proposition~\ref{propos3.5}.
\end{proof}

In the remainder of this section, we will assume that the surfaces
$M_n$ that appear in the statement of Theorem~\ref{tttt}
have uniformly bounded genus
and ${\cal S} \cup S(\lc) \neq \mbox{\O}$; our goal will be to prove
the following statement, which will finish the proof of Theorem~\ref{tttt}.
\begin{proposition}
\label{propos3.10}
Item~7 of Theorem~\ref{tttt} holds.
\end{proposition}
\begin{proof}
Suppose that the surfaces $M_n$ have uniformly
bounded genus and ${\cal S} \cup S({\lc}) \neq \mbox{\rm \O}$.
As before, we will organize the proof in assertions.
\begin{assertion}
\label{asser3.11}
Through every point $p\in {\cal S} \cup S(\lc)$,
there passes a plane of ${\cal P}$ (in particular, ${\cal P}\neq
\mbox{\rm \O }$).
\end{assertion}
{\it Proof. }
Fix a point $p\in
{\cal S} \cup S(\lc)\stackrel{(\ref{eq:S})}{=}
{\cal S}^A\cup (W\cap \overline{\cal
L})^{\mbox{\footnotesize sing}}\cup S({\cal
L})$. By item~4 of Theorem~\ref{tttt}, the assertion holds if
$p\in {\cal S}^A \cup S({\cal L})$; hence in the sequel
we will assume that $p\in {\cal S}\cap W$.
We will discuss two possibilities for $p$, depending on whether or not $p$ is isolated
in ${\cal S}\cap W$.
\begin{enumerate}[(L1)]
\item Assume $p$ is an isolated point of ${\cal S}\cap W$.
Arguing by contradiction, suppose no plane of ${\cal P}$
passes through $p$. By item~5 of Theorem~\ref{tttt},
neither of the conditions 5.1, 5.2 hold.
Since 5.1 does not occur and $p$ is
isolated in the closed set ${\cal S}\cap W$, we can find $\ve >0$
such that such that
$\lc \cap \overline{\B}(p, \ve)$ consists of a
finite number of noncompact, connected,
smooth, properly embedded minimal surfaces $\{\Sigma_1, \ldots ,\Sigma_m\}$
in $\overline{\B}(p, \ve) - \{p \}$ (otherwise there would be a
limit leaf of $\lc \cap (\B (p, \ve) - \{p \})$, contradicting~5.1), together with
finitely many compact, connected, smooth minimal surfaces with boundary on $\partial
\overline{\B} (p, \ve)$. By Corollary~\ref{corol2.3} we have
 that $m=1$ and the surface $\Sigma _1$ has just one
end. Since the surfaces $M_n$ have uniformly bounded genus and converge
with multiplicity one to $\Sigma_1$ (because 5.2 does not occur at $p$), then $\Sigma_1$
has finite genus  not greater than the uniform bound on the genus of
the $M_n$. By the last statement in
Corollary~\ref{corol2.3}, then $\Sigma _1$ extends smoothly across $p$,
contradicting that $p \in {\cal S}$.

\item Assume that $p\in {\cal S}\cap W$ is not an
isolated point of ${\cal S}\cap W$. Since ${\cal S} \cap W$ is a countable closed set
of $\R^3$, then $p$ must be a limit of isolated points $p_k\in {\cal
S}\cap W$. By (L1), there pass planes in ${\cal P}$
through the points $p_k$, $k\in \N$. Our assertion holds in this
case by taking limits of these planes. This finishes the proof of
Assertion~\ref{asser3.11}.
{\hfill\penalty10000\raisebox{-.09em}{$\Box$}\par\medskip}
\end{enumerate}

\begin{assertion}
\label{asser3.12}
 $\overline{\cal L}={\cal P}$ (hence item~7.1 of Theorem~\ref{tttt}
holds).
\end{assertion}
\begin{proof}
Arguing by contradiction, choose a leaf $L$ of
$\overline{\lc}$ in $\overline{\cal L}-{\cal P}$.
By definition of a leaf of a singular lamination, we can
decompose $L=L_1\cup {\cal S}_{L_1}$ where
$L_1$ is a leaf of the related regular lamination ${\cal
L}_1=\overline{\cal L}- {\cal S}$ of $\R^3-{\cal S}$ defined in (\ref{eq:reglam}),
and ${\cal S}_{L_1}$ is the set of
singular leaf points of $L_1$.
Since ${\cal S} \cup S(\lc) \neq \mbox{\O}$, then Assertion~\ref{asser3.11} implies that
${\cal P} \neq \mbox{\O}$. By item~6 of Theorem~\ref{tttt},
$L$ does not intersect ${\cal S}^A\cup S({\cal L})$, there exists a
subcollection ${\cal P}(L)$ of one of two planes in ${\cal P}$ such
that $\overline{L}=L\cup {\cal P}(L)$, $L$ is proper in the open
slab or halfspace component $C(L)$ of $\R^3 - {\cal P}(L)$ that
contains $L$, $C(L)\cap \overline{\cal L}=L$, and $L_1$ has infinite genus.
But since the convergence of portions of
the $M_n$ to $L_1$ has multiplicity one (otherwise $L_1$ and $L$
would be flat), then we conclude $L_1$ has finite genus.
This contradiction finishes the proof of the assertion.
\end{proof}

We now prove item~7.2 of Theorem~\ref{tttt}.
By Assertion~\ref{asser3.12} we have $\overline{\cal L}={\cal P}$, which
implies that ${\cal S} = \mbox{\O}$. Since by hypothesis ${\cal S}\cup
S({\cal L})\neq \mbox{\rm \O}$, it follows that $S({\cal L})\neq
\mbox{\rm \O}$.

\begin{assertion}
\label{asser3.13}
Let $P\in{\cal P}$ be a plane such that $P\cap
S(\lc)\neq\mbox{\rm \O}$ (note that $P$ exists by item~4 of Theorem~\ref{tttt}).
Then,  $P \cap S(\lc)$ contains at
most two points, and if $P \cap S(\lc)=\{ p_1,p_2\} $,
then the two multivalued graphs occurring inside the surfaces $M_n$ near $p_1,
p_2$ are oppositely handed.
\end{assertion}
\begin{proof}
Description (D1) implies that $P \cap S(\lc)$ is discrete in
$P-W$. Reasoning by contradiction, suppose that
$P \cap S(\lc)$ contains three isolated points $p_1,p_2,p_3$.
Let $\Gamma\subset P-W$ be a smooth, embedded compact arc joining $p_1$
to $p_2$  and disjoint from $S(\lc)-\{ p_1,p_2\} $. Note that the
corresponding  two multivalued graphs forming in the surfaces $M_n$
around the points $p_1,p_2$ are oppositely handed (otherwise, for
$n$ large in a fixed size small neighborhood of $\Gamma$ in $\rth
-W$, the surfaces $M_n$ would have unbounded genus, see for example
the proof of Lemma~3.3 in~\cite{mpr14} and also see
the proof of property
(K1) above). Using an analogous local
picture of the $M_n$ near $p_3$, one sees that the handedness of the
multivalued graph
in $M_n$ near $p_3$ must be opposite to the handedness of
the multivalued graph
in $M_n$ near $p_1$ and near $p_2$, which is
impossible since they have opposite handedness. Hence the assertion follows.
\end{proof}

\begin{assertion}
\label{ass3.14}
Item~7.2 of Theorem~\ref{tttt} holds.
\end{assertion}
\begin{proof}
Consider a point $p\in S({\cal L})$. By the local description (D1),
it follows that locally around $p$, the set
$S(\lc)$ is a $C^{1,1}$ arc $\G _p$ that is orthogonal to a local foliation
of disks contained in planes of ${\cal P}$. Thus $\G _p$ is an open
straight line segment orthogonal to the planes in ${\cal P}$. Consider the
collection ${\cal P}(\G _p)$ of planes in ${\cal P}$ that intersect
$\G _p$. If ${\cal P}(\G _p)\cap S({\cal L})=\G _p$, then item~7.2 of
Theorem~\ref{tttt} follows. Otherwise, there is a point $q\in
\left[ {\cal P}(\G _p)\cap S({\cal L})\right] -\G _p$, and the previous arguments
show that there exists a related open line segment $\G _q\subset S({\cal L})$
passing through $q$. $\G_q$ is clearly parallel to (and disjoint from) $\G _p$.
Note that we do not require any maximality on $\G _p$ or $\G _q$ as line
segments contained in $S({\cal L})$ passing respectively through $p$ or $q$.
Choose a plane $P\in {\cal P}$ so that it intersects both $\G _p$ and $\G _q$
and $P$ is disjoint from $W$, which is possible since $W$ is
countable. By  Assertion~\ref{asser3.13}, ${ P} \cap S(\lc)
$ contains exactly two points, one in each segment $\G _p,\G _q$.
Thus, there is a related limiting minimal parking garage
structure ${\cal F}$ in some $\ve$-neighborhood of $P$ (see the
first paragraph just after the statement of Theorem~\ref{tttt} for
an explanation of this limiting parking garage structure). Also,
Assertion~\ref{asser3.13} gives that the two multivalued graphs forming in
$M_n$ near these two points are oppositely
handed. Assertion~\ref{ass3.14} now follows.
\end{proof}

\begin{assertion}
\label{ass3.15}
If the surfaces $M_n$ are compact with boundary, then
item~7.3 of Theorem~\ref{tttt} holds (this completes the proof of
Proposition~\ref{propos3.10}
and of Theorem~\ref{tttt}).
\end{assertion}
{\it Proof. }
Recall that by Assertion~\ref{asser3.12}, $\overline{\cal L}={\cal P}$
(and thus, ${\cal L}=\overline{\cal L}$).
After possibly a rotation, assume that the planes in ${\cal P}$ are
horizontal. The proof of Assertion~\ref{ass3.14} insures that $S({\cal L})$ consists of a
nonempty set of open vertical segments (possibly halflines or
lines). By Assertion~\ref{asser3.13}, every horizontal plane in ${\cal P}$ intersects
$S({\cal L})$ in at most two points. We now
consider two cases, depending upon whether or not some plane $P\in
{\cal P}$ intersects $S({\cal L})$ in exactly two points.

\begin{enumerate}[(M1)]
\item Suppose that some plane $P\in {\cal P}$ intersects $S({\cal L})$ in
exactly two points.\newline
By the proof of
item~7.2 of Theorem~\ref{tttt} and after replacing $P$
by another plane, we may assume
that there exists an open slab $Y$ containing $P$, which is foliated by
planes in ${\cal P}$, and such that $Y\cap S({\cal L})=Y\cap \overline{S({\cal L})}$
consists of two
connected, vertical line segments with boundary end points in the boundary
$\partial Y$ of $Y$. We next exchange $Y$ by the largest such open slab; more precisely,
let $\Delta (Y)$ be the collection of all open slabs $Y'$ in $\R^3$ such that
$Y\subset Y'$, $Y'$ intersects $\overline{S({\cal L})}$ in two
open vertical line segments, and the ends points of each line segment
in $Y'\cap \overline{S({\cal L})}$ are contained in $\partial Y'$.
Define
\begin{equation}
\label{eq:X}
X=\bigcup _{Y'\in \Delta (Y)}Y'.
\end{equation}
Note that $X$ is either an open slab in $\Delta (Y)$,
an open horizontal halfspace or $\R^3$. Furthermore, $X$ intersects
$\overline{S({\cal L})}$ in exactly two connected components, which are
either vertical segments, rays or lines with boundary end points in $\partial X$,
depending on whether or not $X$ is a slab, a halfspace or $\R^3$.
Our goal to prove Assertion~\ref{ass3.15} in this case (M1) is to show that $X=\R^3$.

Suppose $X\neq \R^3$ and we will find a contradiction.
Since $X\neq \R^3$, we may assume without
loss of generality the following properties:
\begin{description}
\item[{\rm (M1.1)}] $\partial X$ contains $P_0=\{ x_3=0\} $
as one of its boundary planes and $X$ lies below $P_0$.
\item[{\rm (M1.2)}] $X$ intersects $\overline{S({\cal L})}$ in exactly two connected
vertical line segments or rays that have end points $p_1,p_2\in
P_0$. One of these two points, say $p_1$, does not lie in the
interior of a line segment in $\overline{S({\cal L})}$ (this fact follows from
the maximality of $X$). Note that this implies that either $p_1$ is
isolated as an end point of a maximal segment in $\overline{S({\cal L})}$ (see
Figure~\ref{figslabs}-Up), or there exists a sequence of maximal
segments in $\overline{S({\cal L})}$ with end points converging to $p_1$
(Figure~\ref{figslabs}-Down).
\end{description}
\begin{figure}
\begin{center}
\includegraphics[width=10.4cm]{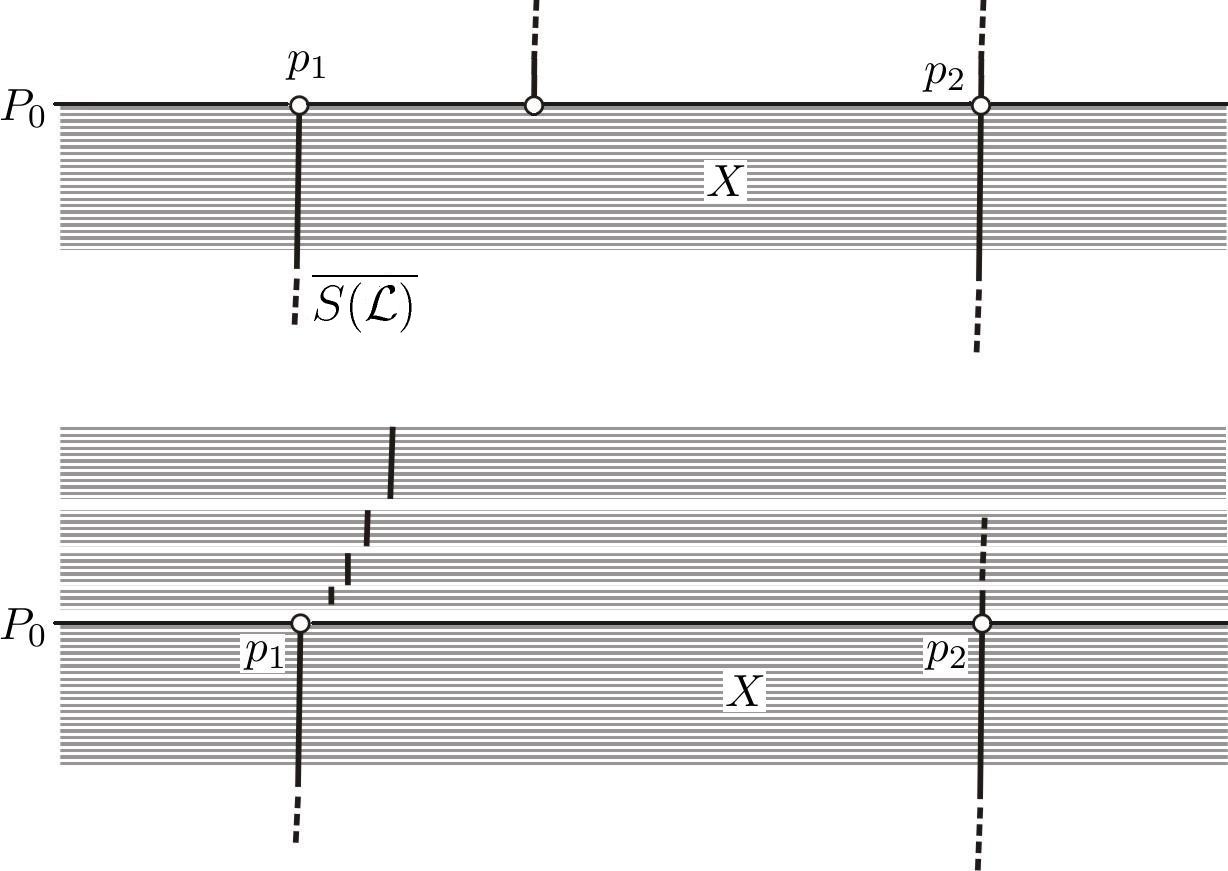}
\caption{Up: $p_1$ is isolated as an end point of maximal segments in
$\overline{S({\cal L})}$. Down: A sequence of maximal segments in
$\overline{S({\cal L})}$ with end points converging to $p_1$.}
 \label{figslabs}
\end{center}
\end{figure}

For $i=1,2$ and $\ve>0$, let
$\overline{D}(p_i,\ve )=\overline{\B}(p_i,\ve )\cap P_0$ be the closed disk in $P_0$
centered at $p_i$ of radius $\be$.
For $k\in \N$, let $C(i,\ve,k)=
\overline{D}(p_i,\ve )\times [-k ,k]$ denote the related compact solid cylinder
in $\R^3$, $i=1,2$. We also denote by $\de C(i,\ve,k)=\partial \overline{D}(p_i,\ve )
\times [-k,k] $ the side of this cylinder.
For a generic fixed small value of $\ve >0$, we have
\begin{enumerate}[(N.1)]
\item
$\overline{D}(p_1,\ve )\cap \ov{D}(p_2,\ve )=\mbox{\O }$, and
\item
$\partial \overline{D}(p_1,\ve )\cup \partial \overline{D}(p_2,\ve )$ is
disjoint for all $n\in \N$ from the compact countable set
\[
A(\ve,k):=
\Pi\left( [W\cup
\overline{S({\cal L})})]\cap \cup_{i=1}^2 C(i,\ve,k)\right) \subset P_0,
\]
where $\Pi$ is the orthogonal projection to $P_0$; hence,
for such a value of $\ve$, $\partial \overline{D}(p_1,\ve )\cup \partial \overline{D}(p_2,\ve )$  is at
least a positive distance from the compact set $A(\ve ,k)$.
\end{enumerate}

Since for each $k\in \N$, $\de C(1,\ve,k)\cup \de C(2,\ve,k)$ is also at a
positive distance
from the closed set $W\cup\overline{S({\cal L})}$, then the compact surfaces $M_n$ converge
$C^1$ as $n\to \infty $ to a subset of the collection of
horizontal planes ${\cal P}$ near the compact set
$\de C(1,\ve,k)\cup \de C(2,\ve,k)$. Also
recall that the $M_n$ converge below $P_0$ to a limiting minimal
parking garage structure and so, for
$i=1,2$ fixed and for
$n$ large enough depending on $k$,
$M_n\cap \de C(i,\ve,k)$  contains a pair of
pairwise disjoint,
long, almost-horizontal,
embedded spiral arcs
$\a_1^k(i,n)$, $\a_2^k(i,n)$, each of which joins
a point in the bottom boundary circle of
$\de C(i,\ve,k)$ to another point in its top boundary circle,
and both $\a _1^k(i,n)$, $\a _2^k(i,n)$ rotate together around $\de C(i,\ve,k)$.
Furthermore, the Gauss map of $M_n$ along $\a _1^k(i,n)$
is arbitrarily close to the north pole of the sphere (for $n$
large) and to the south pole along $\a _2^k(i,n)$.  Since
these spiral arcs intersect every horizontal plane of height between $-k
$ and $k$, then we deduce that the slab $\{ -k \leq x_3\leq k \} $
is foliated by planes in ${\cal P}$. By letting $k$ vary, we deduce that ${\cal P}$  is a foliation
of $\rth$ by horizontal planes.

Fix an integer $k\in \N$. We claim that every horizontal plane $P_t=P_0+t\, (0,0,1)$, $|t|<k$,
intersects $W\cup \overline{S({\cal L})}$ in at least one point of
$C(1,\ve,k)$ and in at least one point of $C(2,\ve,k)$.
If this intersection property does not hold for some $P_t$ with
say $C(1,\ve,k)$, then since $W\cup \overline{S({\cal L})}$ is closed,
the compact disk $D_t=P_t\cap C(1,\ve,k)$
is at a positive distance
from $W\cup \overline{S({\cal L})}$.  Observe that there
exists a sequence of minimal disks $E(n,t)\subset C(1,\ve,k)$ with
$\partial E(n,t)\subset \de C(1,\ve,k)$, that
are vertical graphs over $D_t$,
 converge $C^2$ to $D_t$
as $n\to \infty$ and such that  for $n$ large, $\partial E(n,t)$ intersects each of the
arcs $\a_1^k (i,n), \a_2^k(i,n)$ transversely in exactly  one point (see the perturbation foliation argument in
page 737 of~\cite{mr8} for the construction of the boundary curves $\partial E(n,t)$, so that $E(n,t)$ is found by solving
the Plateau problem for these curves).
After possibly replacing the graphs $E(n,t)$ by sufficiently
small vertical translations  $E(n,t)+(0,0,t_n)$ with $t_n\to 0$, we may assume that $E(n,t)$
intersects $M_n$ transversely in a compact 1-manifold whose
boundary is contained in
$\partial M_n\cup [M_n \cap \partial E(n,t)]$.
Since points in the boundary curves of the surfaces $M_n$
diverge in $\rth$ or converge to points in $W$
 as $n\to \infty$ and since for $n$ sufficiently large,
 $W\cup \overline{S({\cal L})}$ is a positive distance from
$E(n,t)$, then also, for $n$ sufficiently large,
$M_n\cap E(n,t)$ contains exactly  one component with boundary and this component
is a compact arc $c_n$ that joins points of the two long spirals
$\a_1^k (i,n), \a_2^k(i,n)$.
The property of the Gauss map of $M_n$ along
$\a_1^k (i,n)\cup \a_2^k(i,n)$
explained in the previous paragraph implies that the normal vector to $M_n$
at some point $q_n\in c_n$ is horizontal. After extracting a subsequence,
the points $q_n\in E(n,t)$ converge as $n\to \infty $ to a point
 in $[W\cup \overline{S({\cal L})}]\cap P_t\cap C(1,\ve,k)$, which is a contradiction. Hence, our claim holds.

Since $W$ is countable, our claim proved in the previous paragraph implies that
every plane $P_t$,  with $|t|<k$, intersects $\overline{S({\cal L})}$ in
at least one point of $C(1,\ve,k)$
and in at least one point of $C(2,\ve,k)$.
Since, for any $k\in \N,$ the generic\footnote{In
the sense of properties (N.1), (N.2) above.}
radii values $\ve$ of $C(1,\ve,k),C(2,\ve,k) $ can be chosen arbitrarily small,
it follows that $\overline{\cS(\cL)}$ contains every point in the two infinite vertical lines
passing through the points $p_1,p_2$. By Assertion~\ref{asser3.13},
 $\overline{\cS(\cL)}$  is the union of the two infinite vertical lines
passing through the points $p_1,p_2$, which proves item~7.3 of Theorem~\ref{tttt} holds in case (M1) holds
(with two vertical lines for $S({\cal L})$).

\item Suppose that every plane in ${\cal P}$ intersects $S({\cal L})$
in at most one point.
\newline
By item~7.2 of Theorem~\ref{tttt}, either $\overline{S({\cal
L})}$ consists of a single vertical line (and we are done),
or there exists a maximal segment in $\overline{S({\cal L})}$ with some
end point $p$.  The arguments in the case of (M1)
can be easily modified to prove that
for every $k\in \N$ and every small generic\footnotemark[7]
radius, the compact cylinder $C(\ve ,k)=\overline{D}(p,\ve )\times [-k,k]$
intersects $W\cup \overline{S({\cal L})}$ at any height $t$ with $|t|<k$, and thus,
the infinite vertical line $l_p$ passing through $p$ is contained in $\overline{S({\cal L})}$.
Since the cases (M1) and (M2) are mutually exclusive, then
$\overline{S({\cal L})}=l_p$ and ${\cal P}$ is the foliation of $\rth$ by horizontal planes.

This finishes the proofs of Assertion~\ref{ass3.15},
of Proposition~\ref{propos3.10} and of Theorem~\ref{tttt}.
\end{enumerate}
\par
\vspace{-.8cm}
\end{proof}

\section{The structure theorem for singular minimal laminations of $\R^3$ with countable singular set.}
\label{secproofthm1.3}

We next state the
following general structure theorem for possibly singular minimal laminations
of $\R^3$ whose singular set is countable. Theorem~\ref{global} below
is useful in applications because of the following situation. Suppose
that $L$ is a nonplanar leaf of a minimal lamination $\lc$ of $\R^3
-{\cal S}$, with ${\cal S}\subset \R^3$ being closed. In this case,
its closure $\overline{L}$ has the structure of a possibly singular
minimal lamination of $\R^3$, which under certain additional hypotheses,
can be shown to have a countable singular set. If $L$ has  finite genus,
then item~6 of the next
theorem demonstrates that $\overline{L}$
is a smooth, properly embedded minimal surface in $\R^3$, $ L$ is the unique
leaf of $\cL $
and $\cS=\mbox{\O }$.

\begin{theorem} [Structure Theorem for Singular Minimal Laminations of $\R^3$]
\label{global}
\mbox{}\\
Suppose that $\overline{\lc}={\cal L}\cup
{\cal S}$ is a possibly singular minimal lamination of $\R^3$ with
a countable set ${\cal S}$ of singularities. Then:

\begin{description}
\item[{\it 1.}]
\label{12.3.1}
The set ${\cal P}$ of leaves in $\overline{\lc}$
that are planes forms a closed subset of $\R^3$.
\item[{\it 2.}]
\label{12.3.2} The set $\mbox{\rm Lim}(\overline{\cal L})$ of limit
leaves of $\overline{\lc}$ forms a closed set in $\R^3$ and
satisfies $\mbox{\rm Lim}(\overline{\cal L})\subset {\cal P}$.
\item[{\it 3.}]
\label{12.3.3} If $P$ is a plane in ${\cal P} - \mbox{\rm
Lim}(\overline{\cal L})$, then there exists  $\de >0$ such that
$P({\delta}) \cap \overline{\lc}=P$, where $P({\delta})$ is the
$\delta$-neighborhood  of $P$. In particular, ${\cal S}\cap [{\cal
P}- \mbox{\rm Lim}(\overline{\cal L})]=\mbox{\rm \O }$.
\item[{\it 4.}]
\label{12.3.4} Suppose  $p \in {\cal S}-\bigcup _{P\in {\cal P}}P$.
Then for almost all $\ve>0$ sufficiently small, $\lc(p, \ve) = \lc \cap
\overline{\B} (p, \ve) $ has the following description.
\begin{enumerate}[4.1.]
\item
 $\lc(p, \ve) $
 consists of a finite number of leaves, each of which
 is a properly embedded smooth surface in
$\overline{\B}(p, \ve) -{\cal S}$
with compact boundary in $\esf^2(p,\ve )$.

\item
 All of the leaves of $\lc(p, \ve)$ lie on the
same leaf of $\overline{\lc}.$
\item
 Each point $q\in {\B}(p, \ve) \cap {\cal S}$ represents the
end of a unique leaf $L_q$ of $\lc(p, \ve)$,
in the sense that there is a proper arc $\a\colon [0,1)\to L_q$
with $q=\lim_{t\to 1} \a(t)$. Furthermore, this end of $L_q$ has
infinite genus ($L_q=L_{q'}$ may occur if $q,q'$ are distinct points
in ${\B}(p, \ve) \cap {\cal S}$, for example this occurs if ${\B}(p,
\ve) \cap {\cal S}$ is infinite for all small $\ve>0$). In fact, if $p$ is an
isolated point of ${\cal S}$, then $\ve$ can be chosen small enough
so that $\lc (p, \ve)$ is contained in the leaf of $\cL$ that contains $L_p$, and
$L_p$ has infinite genus and exactly one end.
\end{enumerate}
\end{description}
\noindent In items~5, 6 below, suppose that $L=
\overline{\cal L}(L_1)=L_1\cup {\cal S}_{L_1}$ is a leaf of
$\overline{\cal L}$ that is not contained in ${\cal P}$, where $L_1$ is
the related leaf of ${\cal L}$,
and ${\cal S}_{L_1}$ is the set of singular leaf points of $L_1$.
\begin{description}
\item[{\it 5}.] One of the following possibilities holds.
\begin{enumerate}[5.1.]
\item
$L$ is proper in $\R^3$, and  $ L$ is the unique
leaf of $\cL $.
\begin{figure}
\begin{center}
\includegraphics[width=15.1cm]{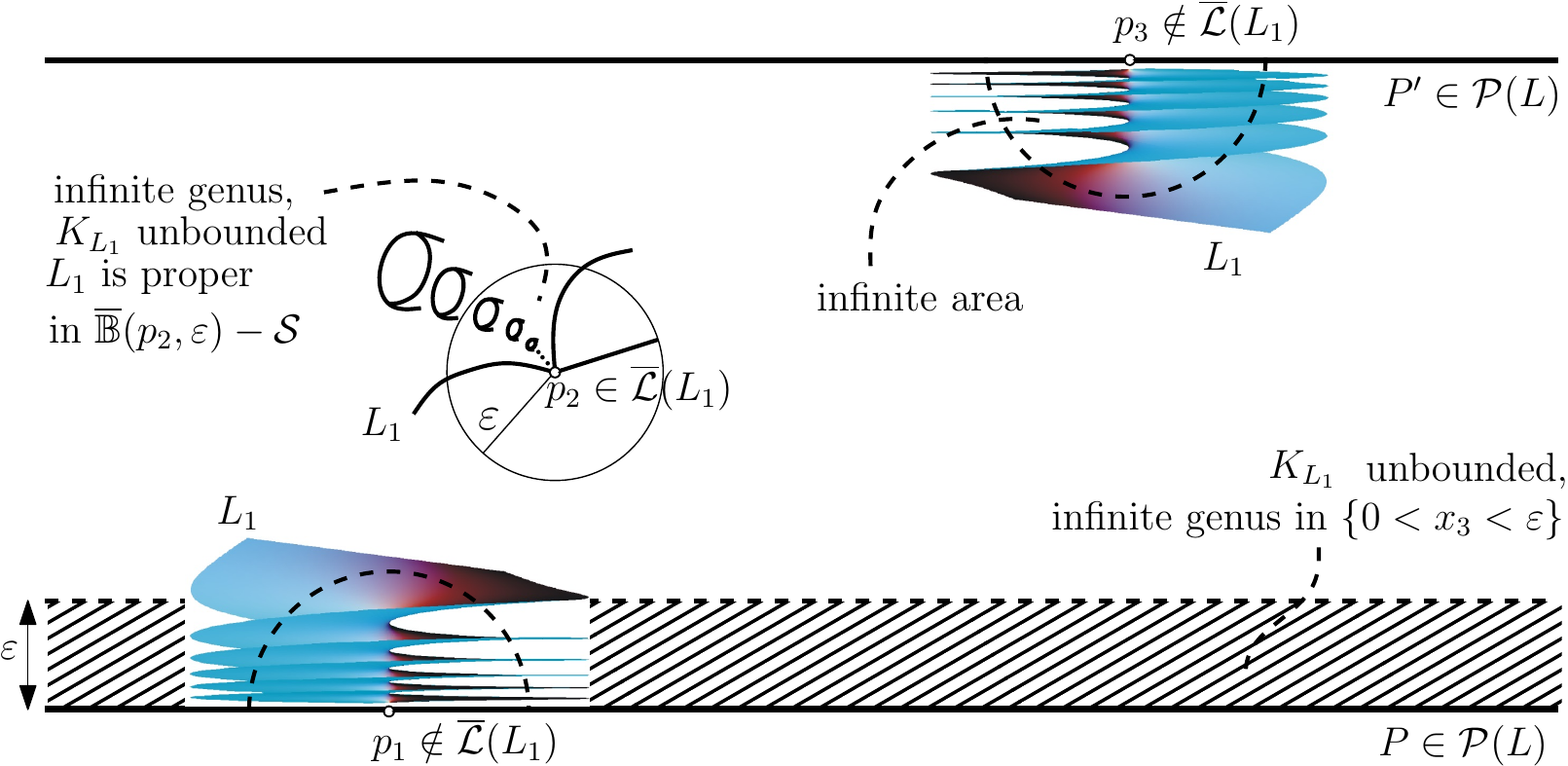}
\caption{Behavior of any
nonproper leaf $L=L_1\cup {\cal S}_{L_1}$
of the singular lamination $\overline{\cal L}$ in item~5.2 of Theorem~\ref{global}:
$L_1$ has infinite genus
in any neighborhood
of a singularity of type $p_2$ and in any slab neighborhood of a plane containing a
singularity of type $p_1$,
and no other leaves of $\overline{\cal L}$ can occur between
$P,P'$. }
\label{fig2}
\end{center}
\end{figure}
\item
 $L$ is not proper in $\R^3$
and ${\cal P}\neq \mbox{\rm \O}$. In this case, the closure
$\overline{L}$ of $L$ in $\R^3$ has the structure of a possibly
singular minimal lamination of $\R^3$ (with singular set contained
in $\overline{L}\cap {\cal S}$) and there exists a subcollection
${\cal P}(L)\subset {\cal P}$ consisting of one or two planes, such
that $\overline{L}=L \cup {\cal P}(L)$ and  $L$
is proper in a component $C(L)$ of $\R^3-{\cal P}(L)$ and
$C(L)\cap \overline{\cal L}=L$. 
Furthermore (see
Figure~\ref{fig2}):\ben[a.]
\item Every open $\ve$-neighborhood $P(\ve)$ of a plane $P\in {\cal
P}(L)$ intersects $L_1$ in a connected surface with unbounded Gaussian curvature.
\item   If some open $\ve$-neighborhood $P(\ve)$ of a plane $P\in {\cal
P}(L)$ intersects $L_1$ in a surface with finite genus, then  $P(\ve)$  is
disjoint from the singular set of $\ov{L}$. 
\item  $L_1$ has infinite genus.
\een
\end{enumerate}

In particular, $\overline{\lc}$ is the disjoint union of its leaves,
regardless of whether case 5.1 or 5.2 occurs.
\item[{\it 6}.]
\label{12.3.6}
 If $L_1$ has finite genus, then $L=L_1$ is a smooth,
   properly embedded minimal surface in $\R^3$ (thus ${\cal
L}=\overline{\cal L}$, $ L  $ is the unique leaf of $\cL$ and ${\cal S}=\mbox{\rm \O}$).
\end{description}
\end{theorem}

As in the previous section, we will divide the proof of Theorem~\ref{global} into several lemmas.


\begin{lemma}
\label{lemma4.1bis}
Items 1, 2 and 3 of Theorem~\ref{global} hold.
\end{lemma}
\begin{proof}
The proof of item~1 of Theorem~\ref{global} is the same as the one of item~1
of Theorem~\ref{tttt}. As for the proof of item~2 of Theorem~\ref{global}, the second paragraph
of the proof of Lemma~\ref{lemma4.1} can be applied without changes to show that if
$L=L_1\cup {\cal S}_{L_1}$ is a limit leaf of $\overline{\lc}$, then the two-sided cover of
$L_1$ is stable. By items~1 and  2 of Corollary~\ref{corrs}, $L_1$
extends across the countable set ${\cal S}$ to a plane. Hence,
$L$ is a plane itself. Since the limit of planes in Lim$(\overline{\cal L})$
is clearly a limit leaf of ${\cal L}$, it follows that Lim$(\overline{\cal L})$
is closed in $\R^3$. These observations prove item~2 of Theorem~\ref{global}.

Next we prove item~{3} of Theorem~\ref{global}. The argument will be similar to that in the
proof of Lemma~\ref{lemma3.4}. Let $P$ be a plane in
${\cal P}-\mbox{Lim}(\overline{\cal L})$. By item~2 of Theorem~\ref{global}, we can
choose $\de >0$ such that the $2\de $-neighborhood
$P(2\de )$ of $P$ does not intersect $\cup _{P'\in {\cal P}-\{ P\} }P'$.
Suppose, arguing by contradiction, that the closed
slab $\overline{P(\de )}$ intersects $\overline{\cal L}$ in a
portion of a leaf $L$ of $\overline{\cal L}$
different from $P$. Note that as a set, $L\cap \overline{P(\de )}$
is proper in $\overline{P(\de )}$
(otherwise we produce a limit leaf of $\overline{\cal L}$ in
$\overline{P(\de )}$, hence a plane which cannot be $P$,
as $P\notin \mbox{Lim}(\overline{\cal L})$). Proposition~\ref{propos1}
implies that $L$ is disjoint from $P$. Now, a standard application
of the proof of the halfspace theorem~\cite{hm10})
using catenoid barriers still works in this setting to
obtain a contradiction to the existence of $L$, thereby finishing
the proof of Lemma~\ref{lemma4.1bis}.
\end{proof}

\begin{lemma}
\label{lemma4.2}
Item 4 of Theorem~\ref{global} holds.
\end{lemma}
\begin{proof}
Take a point $p\in {\cal S}- \bigcup _{P\in {\cal P}}P$. By item~1 of Theorem~\ref{global},
we can choose $\ve '>0$ small
enough so that $\overline{\B} (p,3\ve ')$ does not intersect
${\cal P}$; hence $\overline{\B} (p,3\ve ')$ does not intersect
$\mbox{Lim}(\overline{\cal L})$ as
well, by item~2 of Theorem~\ref{global}.
This last property implies that leaves of ${\cal L}$ are proper in
$\overline{\B }(p,2\ve ')$. Consider a number $\ve \in (\ve ',2\ve ')$
such that the sphere$\esf^2(p,\ve )$
is at a positive distance from ${\cal S}$ and is transverse to $\cL$.
Therefore, $\esf^2(p,\ve )$ intersects ${\cal
L}$ in a finite number of smooth closed curves. Since every
component of ${\cal L}(p,\ve )={\cal L}\cap \overline{\B }(p,\ve )$
intersects $\esf^2(p,\ve )$ (a leaf $L_1$ of ${\cal L}(p,\ve )$
completely contained in $\B (p,\ve )$ would contradict Proposition~\ref{propos1}
applied to $L_1$ and to a plane passing through a point in $L_1$ at
maximum distance from $p$),
then we conclude that
item~4.1 of Theorem~\ref{global} holds.

To prove item~4.2,
note that as all of the components of ${\cal L}(p,\ve )$ are proper
as sets in $\overline{\B }(p,\ve )-{\cal S}$, then
$p$ is a singular leaf point of any leaf of $\overline{L}\cap
\overline{\B }(p,\ve )$ that has $p$ in its closure.
 By Proposition~\ref{propos1}, only one of the components
 of ${\cal L}(p,\ve )$, say $C(p,\ve )$, has $p$ in its closure.
Hence, we can reduce $\ve $ to $\ve_1>0$ so that
${\cal L}(p,\ve _1)\subset C(p,\ve )$ and item~4.2 is proved.

Regarding item~4.3, its first statement follows from Proposition~\ref{propos1}.
Recall that if $e$ is an end of a noncompact surface $\Sigma $ and $\a\colon [0,1)\to \Sigma$
is a proper arc representing $e$, then $e$ has infinite genus if
every proper subdomain $\Omega \subset \Sigma$ with compact boundary
that contains the end of $\a$, has infinite genus. To prove the second statement
in item~4.3 we argue by contradiction:
take $q\in \B(p,\ve)\cap \cS$ and let $\Sigma $
be the component of $\cL(p,\ve)$ that contains the point $q$.
Suppose that $\a$ is a proper arc representing the end of
 $\Sigma $ corresponding to $q$, such that $\Sigma $ contains a
 proper subdomain $\Omega $ with finite genus
 and compact boundary, in such a way that the end of $\a$ is contained in $\Omega$.
Choose $\de\in (0,\ve)$ sufficiently small
so that $\partial \Omega $ lies outside
$\ov{\B}(q,\de)\subset \ov{\B}(p,\ve)$ and $\partial \ov{\B}(q,\de)$ is
transverse to $\Sigma $. Let $\Omega'$ be the component of
$\Omega\cap \ov{\B}(q,\de)$ that contains the
end of $\a$. Since  $\Omega'$  is properly embedded in
$\ov{\B}(q,\de)-\cS$, then the set of points $ \ov{\Omega'}\cap \cS$
is a nonempty closed countable subset of
$\ov{\B}(q,\de)$.  Baire's Theorem implies
that the set of isolated singularities in
$ \ov{\Omega'}\cap \cS$ is dense in $ \ov{\Omega'}\cap \cS$.
But Corollary~\ref{corol2.3} applied around an isolated
singularity of $ \ov{\Omega'}\cap \cS$ in $\B (q,\de )\cap {\cal S}$
gives a contradiction since $\Omega'$ has finite genus.
This contradiction completes the proof of the second statement in item~4.3.
Finally, the last statement in item~4.3 is a consequence of the previously proved parts
of this item. Now the lemma holds.
\end{proof}

\begin{proposition}
\label{propos5.1}
Items 5, 6 of Theorem~\ref{global} hold (and so, the proof of Theorem~\ref{global} is complete).
\end{proposition}
\begin{proof}
Suppose that $L=L_1\cup {\cal S}_{L_1}$ is a leaf of ${\cal L}$ not contained in ${\cal P}$
(with the notation of Theorem~\ref{global}). Following the reasoning in the proof of
Proposition~\ref{propos3.5}, we will distinguish two cases,
depending on whether or not $L$ is proper as a set in $\R^3$. If
$L$ is proper in $\R^3$, then the arguments in
case (E1) of the proof of Proposition~\ref{propos3.5}
are now valid and prove that item~5.1 of Theorem~\ref{global} holds.

Now assume that $L$ is not proper in $\R^3$, and we will deduce that
item~5.2 of Theorem~\ref{global} holds. As before, we will only comment on how to adapt
the arguments in (E2) of the proof of Proposition~\ref{propos3.5} to our current setting.
The property that through every limit point of $L$ there passes a plane in ${\cal P}$
(that is, Assertion~\ref{ass3.6})
follows verbatim, with the only change of ${\cal L}_1$ by ${\cal L}$ in the proof of
Assertion~\ref{ass3.6}. This implies that $\overline{L}=L\cup {\cal P}(L)$ with ${\cal
P}(L)\subset {\cal P}$ consisting of one or two planes, and $L$ is
proper in the component $C(L)$ of $\R^3-{\cal P}(L)$ that contains $L$.
Assertion~\ref{ass3.7} also holds in our new setting, with the only change in its
proof occurring when demonstrating
the countability of the set of points of $\overline{\Sigma }$ where the
least-area surface
$\Sigma $ is possibly incomplete, which is easier now as this set is clearly contained in
the countable set ${\cal S}\cap [L\cup L'\cup
{\cal P}(L)]$). Assertion~\ref{ass3.88}
also holds true now, with the only change in its proof that
incompleteness of the surface $L_1(\ve )=L_1\cap \{ 0<x_3\leq \ve \} $ (we assume
the same normalization as at the beginning of the proof of Assertion~\ref{ass3.88})
may fail at the set ${\cal S}\cap \{ 0\leq x_3\leq \ve \} $, which is countable.
Hence, the proof of the main statement of item~5.2 and item~5.2(a) of
Theorem~\ref{global} are proved.

Assertion~\ref{ass3.8} and its proof are valid in our current setting without changes,
as all their arguments rely on the limit singular lamination $\overline{\cal L}$ of $\R^3$
and not in the sequence of minimal surfaces $\{ M_n\} _n$ that appear in the statement of
Theorem~\ref{tttt}. Therefore, items~5.2(b) and~5.2(c) of Theorem~\ref{global} are also proved.

Finally, item~6 of Theorem~\ref{global} follows
directly from item~{5} of the same theorem.
\end{proof}

\section{A convergence result for embedded minimal surfaces of uniformly bounded genus.}
\label{sec7}


\begin{theorem}
\label{thmcm} Suppose $\{M_n\}_n$ is a sequence of compact, embedded
minimal surfaces of finite genus at most $g\in \N\cup \{ 0\}$, with $\partial M_n
\subset \esf^2(n)$ for each $n$. Suppose that some subsequence
of disks $\{D_n \subset M_n\}_n$ converges $C^2$ to a nonflat
minimal disk. Then, a subsequence of the $M_n$ converges smoothly on
compact subsets of $\R^3$ with multiplicity one to a connected,
properly embedded, nonflat minimal surface $M_{\infty}\subset \R^3$
of genus at most $g$, that is either a surface of finite total
curvature, a helicoid with handles or a two-limit-ended minimal
surface. Furthermore, $M_{\infty }$ has bounded Gaussian curvature.
\end{theorem}
\begin{proof}
Suppose for the moment that there exists $R>0$ such that
\[
\sup _{M_n\cap \B (R)}|K_{M_n}|\to \infty \mbox{ as }n\to \infty.
\]
By Theorem~0.6 in~\cite{cm25} (see also Footnote 3 in the statement of
that theorem), then after a rotation in $\R^3$, there exists a subsequence
of these compact minimal surfaces, also denoted by $\{ M_n\} _n$,
a lamination ${\cal L}_1=\{ x_3=t\} _{t\in {\cal I}}$ by parallel planes
(where ${\cal I}\subset \R $ is a closed set), and a nonempty closed set $S({\cal L}_1)$
in the union of the leaves of ${\cal L}_1$ such that:
\begin{enumerate}[(N1)]
\item For each $\a \in (0,1)$, $\{ M_n-S({\cal L}_1)\} _n$ converges in the $C^{\a }$-topology
to the lamination ${\cal L}_1-S({\cal L}_1)$.
\item ${\displaystyle \sup _{M_n\cap \B(x,r)}|K_{M_n}|\to \infty }$ as $n\to \infty $ for
all $x\in S({\cal L}_1)$ and $r>0$.
\end{enumerate}

Our hypothesis that a sequence of disks
$D_n\subset M_n$ converges $C^2$ to a nonflat minimal disk contradicts that
${\cal L}_1$ consists of planar leaves. This contradiction shows that
the $M_n$ have uniformly bounded Gaussian curvatures on compact subsets of $\R^3$.
Therefore, a standard diagonal argument implies that after passing to a subsequence,
the $M_n$ converge to a (regular) minimal lamination ${\cal L}$ of $\R^3$.

By the structure theorem for (regular) minimal laminations of $\R^3$ (see Theorem~1.6
in~\cite{mr8}), the collection ${\cal P}$ of planes in ${\cal L}$
forms a possibly empty, closed set of $\R^3$, each of  the
components $X$ of $\R^3-{\cal P}$ contains at most one leaf $L_X$ of
${\cal L}$, and such a leaf $L_X$ is not flat and proper in $X$.
Since every such a $L_X$ is nonflat, its universal cover cannot
be stable and the proof of Lemma~A.1 in~\cite{mr13} implies that
the multiplicity of the convergence of portions of the $M_n$ to
$L_X$ is one. In this setting, standard lifting arguments give that
one can lift any handle on $L_X$ to a nearby handle on
an approximating surface $M_n$ for $n$
large, such that any fixed finite collection of pairwise disjoint
handles in $L_X$ lifts to a collection of disjoint handles on the
nearby surface
$M_n$. Since the genus of each $M_n$ is at most $g$,  then $L_X$
has genus at most $g$. By Corollary~1 in~\cite{mpr3}, $L_X$ is the
only leaf in ${\cal L}$ and $L_X$ is properly embedded in $\R^3$.
This proves the first item in~Theorem~\ref{thmcm} with $M_{\infty }$ being
$L_X$.

By Theorem~1 in~\cite{mpr4}, the surface $M_{\infty}$ is either a
surface of finite total curvature, a helicoid with handles or a
minimal surface with two limit ends. The same theorem states that
$M_{\infty }$ has  bounded curvature, which completes the proof of
Theorem~\ref{thmcm}.
\end{proof}

\center{William H. Meeks, III at profmeeks@gmail.com\\
Mathematics Department, University of Massachusetts, Amherst, MA
01003}
\center{Joaqu\'\i n P\'{e}rez at jperez@ugr.es \qquad\qquad Antonio Ros at aros@ugr.es\\
Department of Geometry and Topology and Institute of Mathematics
(IEMath-GR), University of Granada, 18071, Granada, Spain}

\bibliographystyle{plain}
\bibliography{bill}

\end{document}